%% file: main.tex
\newif\ifHAL
\HALtrue
 
\ifHAL
\documentclass[11pt, a4paper, twoside, oneside]{article}
\else
\documentclass{siamltex}
\fi

\input{preamble.tex}
\input mydef.sty

\graphicspath{{images/}}

\newcommand{\cuthereq}{\\ & \qquad}
\newcommand{\cuthere}{\\ &}

\newcommand{\rev}[1]{{\color{black}#1}}

\begin{document}

\newcommand\footnotemarkfromtitle[1]{%
\renewcommand{\thefootnote}{\fnsymbol{footnote}}%
\footnotemark[#1]%
\renewcommand{\thefootnote}{\arabic{footnote}}}

\ifHAL
\makeatletter
\renewcommand{\and}{\unskip\enspace \qquad} 
\renewcommand{\author}[1]{\gdef\@author{#1}}
\makeatother
\date{}
\newenvironment{keywords}
{\vspace{-0.25cm} \begin{adjustwidth}{1cm}{1cm}
\noindent \textbf{Keywords.} }
{\end{adjustwidth}}
\newenvironment{AMS}
{\begin{adjustwidth}{1cm}{1cm}
\noindent \textbf{Mathematics Subjects Classification.} }
{\end{adjustwidth}}
\else
\date{Draft version \today}
\footnotetext[1]{Draft version, \today}
\fi

\title{\ifHAL \vspace{-1.75cm} \else \fi Unfitted hybrid high-order methods stabilized by polynomial extension for elliptic interface problems \ifHAL \else \footnotemark[1] \fi }

\author{Erik Burman\footnotemark[2] \and Alexandre Ern\footnotemark[3] \and
Romain Mottier\footnotemark[4]}

\maketitle
\ifHAL
\vspace{-1.5cm}
\fi

\renewcommand{\thefootnote}{\fnsymbol{footnote}}
\footnotetext[2]{Department of Mathematics, University College London, London WC1E 6BT, UK.}
\footnotetext[3]{CERMICS, ENPC, Institut Polytechnique de Paris, F-77455 Marne-la-Vall\'ee cedex 2, and Centre Inria de Paris, 48 rue Barrault, CS 61534, F-75647 Paris, France.}
\footnotetext[4]{CEA, DAM, DIF, F-91297 Arpajon, France, and CERMICS, ENPC, Institut Polytechnique de Paris, F-77455 Marne-la-Vall\'ee cedex 2, and Centre Inria de Paris, 48 rue Barrault, CS 61534, F-75647 Paris, France.}
\renewcommand{\thefootnote}{\arabic{footnote}}
\begin{abstract}
In this work, we study the design and analysis of a novel hybrid high-order (HHO) method on unfitted meshes. HHO methods rely on a pair of unknowns, combining polynomials attached to the mesh faces and the mesh cells. In the unfitted framework, the interface can cut through the mesh cells in a very general fashion, and the polynomial unknowns are doubled in the cut cells and the cut faces. In order to avoid the ill-conditioning issues caused by the presence of small cut cells, the novel approach introduced herein is to use polynomial extensions in the definition of the gradient reconstruction operator. Stability and consistency results are established, leading to optimally decaying error estimates. The theory is illustrated by numerical experiments.
\end{abstract}
\ifHAL
\else
\pagestyle{myheadings} 
\thispagestyle{plain}
\markboth{E. Burman, A. Ern and R. Mottier}{Unfitted HHO methods stabilized by polynomial extension}
\fi

\section{Introduction}

The meshing of complex domains or of internal interfaces is an important bottleneck in the large-scale computation of solutions to partial differential equations. It is well-known that, for standard finite element methods, the mesh must match the geometry sufficiently well, or accuracy will be lost. This is particularly important for high-order methods. Recently, many methods have been developed aiming at simplifying the meshing procedure by allowing for a larger set of geometric shapes for the computational cells. Examples of discretization methods supporting fairly general polygonal (or polyhedral) cells are the discontinuous Galerkin (dG) method \cite{DiPEr:12,CDGH17}, the virtual element method (VEM) \cite{BBCMMR13}, 
and hybrid nonconforming methods in various flavors, such as the hybridizable dG (HDG) method \cite{CGL09}, the weak Galerkin (WG) method \cite{WY13}, the nonconforming VEM (ncVEM) method \cite{AyLiM:16} and the hybrid high-order (HHO) method \cite{DPEL14}. The tight connections between HDG, WG, and HHO methods are highlighted in \cite{CoDPE16}. 
These three methods are formulated by means of degrees of freedom (dofs) both in the bulk of the elements and on the element faces, whereas dG methods only use bulk dofs. Using face dofs leads to advantageous properties, such as the elimination of bulk dofs by static condensation and the absence of minimal threshold to weigh the stabilization. 

Nevertheless, difficulties remain when using hybrid nonconforming methods in the presence of curved boundaries or interfaces, whereas the treatment of cells with curved faces in dG methods is more straightforward. One possibility is to enrich (with non-polynomial functions) the set of unknowns on curved faces, as in \cite{Yemm:24} for HHO and in \cite{BdVLMR:24} for ncVEM. Another possibility is to work with mesh cells having flat faces, but to build the geometric representation directly into the discretization, in such a way that solution accuracy is not lost although the mesh does not respect the geometry. This idea has been developed first in the framework of unfitted finite element methods. These methods originate from the penalty method \cite{Bab73, BE86} and were further developed for conforming FEM in \cite{BE84}. Combining the ideas of unfitted FEM with Nitsche's method for interface problems led to a consistent unfitted finite element method for interface problems in \cite{Han02}. These ideas were then generalized to include more complex coupling problems including PDEs on surfaces and multiphysics coupling, leading to the cutFEM framework \cite{BCHLM15}. Unfitted methods have also been considered for dG \cite{BE09,JL13,Mass12}, HDG \cite{CocQS14,QiuSV16}, and HHO \cite{BE18} methods. A crucial problem for unfitted methods is to handle the instabilities that can occur for certain intersections of the mesh and the interface. In the cutFEM framework using a conforming approximation in each subdomain, this is solved by the addition of certain stabilized terms \cite{Bu10}, whereas in dG methods, which are more flexible with respect to cell geometry, this can be solved by cell agglomeration \cite{SBvdV11,JL13}. This technique can also be used in the context of conforming finite elements \cite{BVM18} (therein called aggregated FEM).

In the present work, we continue the development of unfitted HHO methods. The combination of cell agglomeration and a local weak coupling derived from the HHO framework has been successfully applied to elliptic interface problems \cite{BCDE21}, incompressible Stokes flows \cite{BDE21}, and wave propagation in heterogeneous media \cite{BDE22}. Nevertheless, if these methods have to be integrated in an existing HHO code, the cell-agglomeration procedure becomes a disruptive element since it requires modifying part of the mesh in the interface zone. It is often desirable to keep the underlying mesh fixed and modify instead the algebraic structure of the bulk unknowns. One technique that fits this purpose is the idea of polynomial extension from interior elements to stabilize the method. The idea was first introduced in the context of Lagrange multipliers \cite{HR09} and then using isogeometric approximation and Nitsche's method in \cite{BPV20}.
The purpose of this work is to develop and analyze this technique for hybrid nonconforming methods, focusing on HHO methods. Owing to the close links between HDG, WG, ncVEM and HHO methods, the present results are relevant to the development of these other methods as well. Moreover, we focus on elliptic interface problems, but observe that elliptic problems posed on curved domains can be handled by the same techniques.

Let us briefly outline the main idea. Recall that a central ingredient in HHO methods is a local gradient reconstruction from the local cell and face unknowns. The stability and consistency properties of this reconstructed gradient crucially hinge on a discrete trace inequality on the normal derivatives at the cell faces. In the case of ill-cut cells, the constant in this inequality degenerates. One remedy, explored in \cite{BE18,BCDE21}, is to use cell agglomeration around ill-cut cells. This procedure produces an aggregated mesh solely composed of well-cut cells, whose size is close to that of the original mesh. The remedy explored herein is to use the cell and face polynomials of ill-cut cells in the gradient reconstruction of a neighboring well-cut cell. With this approach, the mesh cells remain unmodified, but the stencil associated with the gradient reconstruction operator is slightly extended. \rev{Moreover, the cell polynomials of the ill-cut cells are stabilized by using the cell polynomials of the associated well-cut cells, in the same spirit as the so-called direct ghost penalty method devised in \cite{Preuss:18,LehOl:19}.} The main question is then whether the resulting method maintains the above stability and optimal approximation properties. The proof of these properties is by no means straightforward and is the main contribution of this work. Interestingly, the main analysis tools (discrete trace inequality and polynomial approximation) are of independent interest. 

This work is organized as follows. In Section~\ref{sec:model}, we introduce the model problem.
In Section~\ref{sec:unfitted_HHO}, we present the unfitted HHO method stabilized
using polynomial extensions. 
The stability and error analysis of the method is covered in Section~\ref{sec:analysis}.
Numerical results are presented in Section~\ref{sec:res}. Finally, the proofs of the 
analysis tools (discrete trace inequality and polynomial approximation)
on unfitted meshes is outlined in Section~\ref{sec:proofs}.

\section{Model problem}
\label{sec:model}

Let $\Omega$ be a polyhedral domain in $\R^d$, $d \in\{2,3\}$ (open, bounded, connected, Lipschitz subset of
$\R^d$), and consider a partition of $\Omega$ into two disjoint subdomains $\overline{\Omega} = \overline
{\Omega_1} \cup \overline{\Omega_2}$ with interface $\Gamma := \partial \Omega_1 \cap \partial
\Omega_2$, as illustrated in Figure~\ref{fig:model}. 
For all $i\in\{1,2\}$, we set $\ibar := 3-i$, so that $\Omega\setminus \Omega^i = \Omega^{\ibar} \cup \Gamma$.
For simplicity, we assume that the interface $\Gamma$ is of class $C^2$ and that it does not touch the boundary of $\Omega$. The unit normal vector $\nG$ to $\Gamma$ conventionally points from $\Omega_1$ to $\Omega_2$.
For a smooth enough function $v$ defined on $\Omega_1 \cup \Omega_2$, setting $v_i:=v|_{\Omega_i}$ for all $i\in\{1,2\}$,
we denote its jump across $\Gamma$ as
$\llbracket v \rrbracket_{\Gamma} := v_{1}|_\Gamma-v_{2}|_\Gamma$.

\begin{figure}[htb]
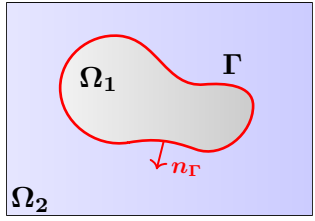

\centering
\input model.tikz
\caption{Model problem}
\label{fig:model}
\end{figure}

Our goal is to approximate the \rev{exact} solution $\rev{u\upex} \in H^1(\Omega_1 \cup \Omega_2)$ of
the following elliptic interface problem:
\begin{subequations} \begin{alignat}{2}
-\nabla \SCAL(\kappa \nabla \rev{u\upex}) & = f &\quad &\text{ in } \Omega_1 \cup \Omega_2,\label{eq:strong.PDE}\\
\llbracket \rev{u\upex} \rrbracket_{\Gamma} & = g_D &\quad &\text{ on } \Gamma,\label{eq:strong.D}\\
\llbracket \kappa \nabla \rev{u\upex} \rrbracket_{\Gamma} \cdot \nG & = g_N &\quad &\text{ on } \Gamma, \label{eq:strong.N}\\
\rev{u\upex} & = 0 &\quad &\text{ on } \partial \Omega,\label{eq:strong.BC}
\end{alignat}
\end{subequations}
with $ f \in L^2(\Omega)$, $g_D \in H^{\frac{1}{2}}(\Gamma)$ and $g_N \in L^2(\Gamma)$. For simplicity, we
consider a homogeneous Dirichlet boundary condition on $\partial \Omega$ (see~\eqref{eq:strong.BC}). We assume that
the diffusion coefficient $\kappa$ is scalar-valued and that $\kappa_i := \kappa|_{\Omega_i}$ is constant for each
$i \in\{1,2\}$. To fix the ideas, we assume that $\kappa_1 \leq \kappa_2$. Our analysis covers the strongly
contrasted case where $\kappa_1 \ll \kappa_2$. We emphasize that it is the property 
$\kappa_1 \leq \kappa_2$ that fixes the numbering of the subdomains, and not the fact that 
one of them touches the boundary. Thus, the numbering in Figure~\ref{fig:model} is only illustrative.
We also notice that the above assumptions on $\Gamma$ and $\kappa$ 
can be lifted with additional technicalities.

The weak formulation can be written as
\begin{equation} \label{eq:weak}
\text{Find } \rev{u\upex} \in V_{g_D}, \qquad a(\rev{u\upex},w) = \ell(w) \qquad \forall w \in V_0,
\end{equation}
with $V_{g_D} := \{ v\in H^1(\Omega_1 \cup \Omega_2) \mid \llbracket v \rrbracket_{\Gamma} = g_D \: \text{on $\Gamma$}, \: v=0 \: \text{on $\partial\Omega$}\}$.
Notice that $V_0=H^1_0(\Omega)$.
The symmetric bilinear form $a$ and the linear form $\ell$ are defined as
\begin{subequations}
\begin{empheq}[left=\empheqlbrace]{align}
a(v,w) & := \sum_{i\in\{1,2\}} \kappa_i (\nabla v_i, \nabla w_i)_{\Omega_i}\\
\ell(w) & := (f,w)_{\Omega} + (g_N, w)_{\Gamma}.
\end{empheq}
\end{subequations}
Here and in what follows, for a measurable subset $S\subset \Omega$, we denote 
by $(\cdot,\cdot)_S$ the $L^2(S)$-inner product with appropriate Lebesgue measure
and by $\|\SCAL\|_S$ the corresponding norm.
We notice that, in the weak formulation~\eqref{eq:weak}, the jump condition~\eqref{eq:strong.D} and the boundary condition~\eqref{eq:strong.BC} are enforced explicitly in the definition of the functional space $V_{g_D}$, whereas the jump condition~\eqref{eq:strong.N} is enforced weakly (i.e., results from the weak formulation). 

\section{Unfitted HHO methods}
\label{sec:unfitted_HHO}

In this section, we present the unfitted HHO method with stabilization of ill-cut cells 
using polynomial extensions. 

\subsection{Unfitted meshes}
\label{subsec:unfitted_meshes}

Let $\mesh$ be a mesh that covers $\Omega$ exactly. A generic mesh cell is denoted 
$T \in \mesh$, it is considered to be an open set,
$h_T$ denotes its diameter, and $\nT$ its unit outward normal. 
The mesh size is defined as $h := \max_{T \in \mesh} h_T$. 
For all $T \in \mesh$, its neighboring layers $\Delta_j(T) \subset \mesh$ are 
defined by induction as $\Delta_0
(T):=T$ and $\Delta_{j}(T):=\left\{T^{\prime} \in \mesh \mid \overline{T^{\prime}} \cap \overline{\Delta_{j-1}(T)} \neq \emptyset\right\}$ for all \rev{$j \in \rev{\mathbb{N}}$}.
In what follows, we use $j\in\{1,2\}$.
To alleviate the notation, we set $\Delta(T):=\Delta_1(T)$. We make the mild 
assumption that, for every mesh cell $T\in\mesh$, its convex hull, $\conv(T)$, satisfies
\begin{equation} \label{eq:conv_delta}
\conv(T) \subset \Delta(T). 
\end{equation}
For instance, this assumption trivially holds true if all the
mesh cells are convex sets, and it remains a reasonable assumption if the mesh cells 
are nonconvex. One can also assume more generally 
that $\conv(T)\subset \Delta_{n_0}(T)$ for some fixed integer $n_0$.

The mesh $\mesh$ is typically composed of cells having a simple shape (although this is 
not a necessary assumption); in
our illustrations, we consider a structured mesh composed of squares.
Quite importantly, the mesh $\mesh$ is not fitted to the interface $\Gamma$, so that
$\Gamma$ can cut arbitrarily across any mesh cell. 
For all $T\in\mesh$, we set $T^{\Gamma} := T \cap \Gamma$ and
$T^i:= T\cap \Omega^i$ for all $i\in\{1,2\}$ (some of these sets may be empty).
\rev{The current implementation assumes that both sub-cells $T^i$ are connected,
although this assumption is not needed from a theoretical viewpoint.}

The first obvious partition among the mesh cells is between uncut and cut cells.
Moreover, among cut cells, we distinguish between well-cut and ill-cut cells. Specifically,
we fix a parameter $\vartheta\in(0,1)$ and say that a cut cell $T$ is a well-cut cell
if, for all $i\in\{1,2\}$, $T^i$ contains a ball of radius $\vartheta h_T$,
whereas $T$ is an ill-cut cell if this condition fails. 
It is shown in \cite[Lemma 6.2]{BE18} that, it the mesh size $h$ is small enough 
(with respect to the curvature of the interface) and the parameter $\vartheta$ is small enough
(with respect to the shape-regularity parameter of the mesh), then, for every ill-cut cell,
the above condition can fail on one and only one side of $T$. In what follows, we assume 
that this is always the case. The above considerations lead to 
the following partitions:
\begin{equation} \label{eq:partition3}
\mesh := \uncutmesh \cup \cutmesh, \qquad \cutmesh :=\TOKmesh \cup \TKOmesh,
\end{equation}
where $\uncutmesh$ collects the uncut cells, $\cutmesh$ the cut cells,
$\TOKmesh$ the well-cut cells, and $\TKOmesh$ the ill-cut cells.
Setting $\mesh^i := \{T \in \mesh \mid T \subset \Omega_i \}$ for all 
$i \in \{1,2\}$, we have $\uncutmesh = \mesh^1 \cup \mesh^2$. 
For every $T\in\uncutmesh$, we define 
$\iota(T) \in \{1,2\}$ as the index such that $T\in \mesh^{\iota(T)}$. 
Moreover, for all $T\in \TKOmesh$, we define $\iota(T)\in\{1,2\}$ as the
index for which the above ball condition fails. We then consider the partition
$\TKOmesh= \TKOonemesh \cup \TKOtwomesh$, where 
$T\in \TKOimesh$ if $\iota(T)=i$. 
In conclusion, a mesh partition that is finer than~\eqref{eq:partition3} is 
\begin{equation} \label{eq:partition5}
\mesh := \underbrace{\mesh^1 \cup \mesh^2}_{\uncutmesh} \cup 
\underbrace{\TOKmesh \cup \TKOonemesh \cup \TKOtwomesh}_{=\cutmesh}.
\end{equation}

\begin{figure}[htb]
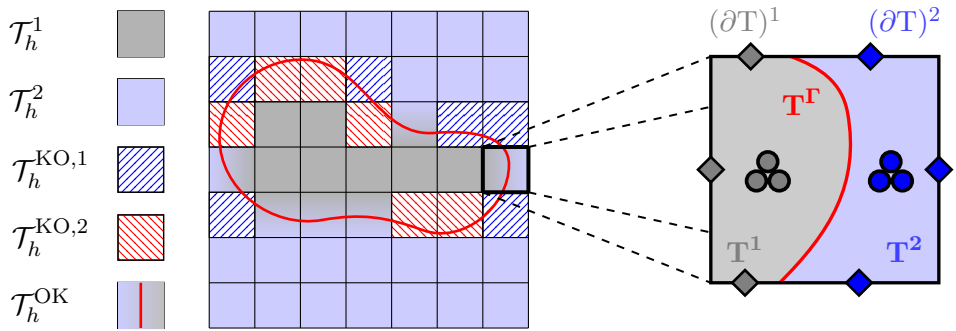

\centering
\input cut_mesh.tikz
\caption{Left: Illustration of the different types of cell in the unfitted mesh.
Right: Zoom on the local degrees of freedom in 
a cut cell $T \in \TOKmesh$; here, the approximation is affine in the sub-cells and
constant on the sub-faces.}
\label{cut_configuration}
\end{figure}

The polynomial extension technique we propose relies on a pairing operator, which maps every 
ill-cut cell $S\in\TKOmesh$ to an uncut or a well-cut cell that belongs to the neighborhood $\Delta(S)$. 
We denote the pairing operator as $\N : \TKOmesh \longrightarrow \mesh$, so that we have
\begin{equation}
\N : \TKOmesh \ni S \longmapsto T \in (\mesh^i \cup \TOKmesh \cup \TKOibarmesh) \cap \Delta(S), \quad i:=\iota(T).
\end{equation}
Letting $\N_i := {\N}|_{\TKOimesh}$ for all $i\in\{1,2\}$, we have more specifically
\begin{equation}
\N_i : \TKOimesh \ni S \longmapsto T \in (\mesh^i \cup \TOKmesh \cup \TKOibarmesh) \cap \Delta(S).
\end{equation}
We notice that the set $\TKOimesh$ can be partitioned as follows:
\begin{equation}
\TKOimesh =  \N_i^{-1}(\mesh^i) \cup \N_i^{-1}(\TOKmesh) \cup \N_i^{-1}(\TKOibarmesh),
\qquad \forall i \in \{1,2\}.
\label{decomposition}
\end{equation}
The pairing operator is essentially the one constructed in \cite[Algorithm~1]{BCDE21}. One difference, however, is that after Step~3 of this algorithm, we need to add an additional Step~4 performing the same operations as in Step~3, but for any mesh cell in $\TKOtwomesh$. One way of proceeding is to set $\N_2(T):=S$, whenever $T:=\N_1(S)\in \TKOtwomesh$ for some cell $S\in \TKOonemesh$. Thus, the well-cut part of $S$ in $\Omega_2$ will be used for $T$, and the well-cut part of $T$ in $\Omega_1$ will be used for $S$. The reason why Step~4 is not needed in the context of cell-agglomeration is that by agglomerating $T$ and $S$ in Step 3, $T$ does not require any further stabilization. 
An illustration of the pairing operators $\N_1$ and $\N_2$ is shown in Figure~\ref{figure_decomposition}, and two concrete examples are displayed in the central column of Figure~\ref{fig:meshes} below. Notice that it may happen that some cells in $\TOKmesh$ are in the image of both $\N_1$ and $\N_2$. \rev{The above choice in Step 4 aims at reducing the possible impact of the 
pairing operator on the resulting stiffness matrix. There may be room for some further improvement, but we do not delve on this aspect here since the impact of the pairing operator on the stencil is already quite modest, as highlighted in Figure~\ref{fig:sparsity} below.}

\begin{figure}[htb]
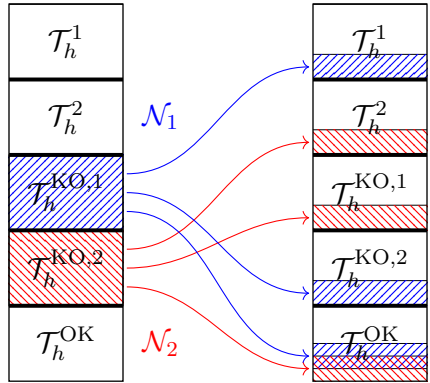

\centering
\input pairing.tikz
\caption{Pairing operators $\N_1$ and $\N_2$ acting on the ill-cut cells.}
\label{figure_decomposition}
\end{figure}

The mesh faces are collected in the set $\meshF$, and we set
$\meshF^i := \{F \in \meshF \mid F \subset \Omega_i \}$ for all $i\in\{1,2\}$.
For every mesh face $F \in \meshF^\Gamma := \meshF \setminus (\meshF^1\cup \meshF^2)$ 
cut by the interface $\Gamma$, we define its two sub-faces as $F^i := F \cap \Omega_i$ 
for all $i\in\{1,2\}$.
For all $T\in\mesh$, $\localfaces$ is the collection of mesh faces composing
the boundary $\partial T$.  
Moreover, for every cut cell $T\in \TOKmesh \cup \TKOmesh$, 
recalling that $T^{\Gamma} = T \cap \Gamma$,
the boundary $\partial(T^i)$ of the sub-cell $T^i$ is
decomposed as
\begin{equation}
\partial(T^i) := \dTi \cup \TG, \quad \dTi := \partial T \cap (\overline{\Omega_i} \backslash \Gamma), \quad \forall i\in\{1,2\}.
\end{equation}
To unify the notation, for every uncut cell $T \in \mesh^i$ with $i:=\iota(T) 
\in\{1,2\}$, we set (recall that $\ibar=3-i$)
\begin{equation}
T^i := T, \quad T^{\ibar}:=\emptyset, \quad\dTi:=\partial T, \quad \dTibar:=\emptyset, \quad \TG
:=\emptyset.
\end{equation}
Finally, for all $T\in\mesh$ and all $i\in\{1,2\}$, the (possibly empty) set
$\localfacesi := \left\{F^i = F \cap \Omega_i \mid F \in \localfaces \right\}$ is 
the collection of the (sub-)faces composing $\dTi$.

\subsection{Unfitted HHO spaces}

HHO methods utilize polynomial unknowns attached to the mesh cells and to the mesh faces. 
We consider here the mixed-order HHO method, where the degree of the face
unknowns is $k \geq 0$ and that of the cell unknowns is $(k+1)$. For the motivation of using a mixed-order setting, we refer the reader to \cite[Remark~4.1]{BE18}.

For every uncut cell $T \in \mesh^i$, $i \in\{1,2\}$, the local discrete HHO unknowns are a pair composed of one cell polynomial and a collection of face polynomials (one for each face of $T$), so that 
\begin{equation}
\uThat := (\uT, \ufT) \in \UThat := \P^{k+1}(T;\R) \times \P^k(\localfaces;\R),
\quad \P^k(\localfaces;\R) := 
\bigtimes_{F \in \localfaces} \P^k(F;\R). 
\end{equation}
Here, $\P^{k+1}(T;\R)$ (resp.,
$\P^k(F;\R)$) consists of the restriction to $T$ (resp., $F$) of 
scalar-valued $d$ variate polynomials
of degree at most $(k+1)$ (resp., $(d-1)$-variate polynomials of degree at most $k$ 
composed with any
affine geometric mapping from the hyperplane supporting $F$ to $\R^{d-1}$).
For every cut cell $T \in \TOKmesh \cup \TKOmesh$, following \cite{BE18,BCDE21},
we double the HHO unknowns so as to have
the usual unknowns available in each sub-cell, 
except on $\TG$ where there are no unknowns.
Thus, the local HHO unknowns in every cut cell are
\begin{equation}
\uThat := (\uThatone, \uThattwo) := (\uTone, \uFone, \uTtwo, \uFtwo) \in \UThat := \UThatone \times
\UThattwo,
\label{unknowns_cut}
\end{equation}
with $\UTihat:=\P^{k+1}(T^i)\times \P^k(\localfacesi)$ and
$\P^k(\localfacesi):= \bigtimes_{F^i \in \localfacesi} \P^k(F^i)$, for all
$i \in\{1,2\}$. To unify the notation between cut and uncut cells, we write $\uThat$ in the same format as in
(\ref{unknowns_cut}) for every uncut cell as well. Thus, we write
$\uThat := (\uT, \ufT, 0,0)$ for all $T \in \mesh^1$ and 
$\uThat := (0,0, \uT, \ufT)$ for all $T \in
\mesh^2$.

The global HHO space on the unfitted mesh is
\begin{equation}
\Uhat := \UTHone \times \UFhone \times \UTHtwo \times \UFhtwo,
\label{HHO_spaces}
\end{equation}
where
\begin{equation}
\UTHi := \underset{T \in \mesh}{\bigtimes} \P^{k+1}(T^i
;\R), \quad
\UFHi := \underset{F \in \meshF}{\bigtimes} \P^k(F^i;\R),
\quad \forall i \in \{1,2\}.
\end{equation} 
To enforce the homogeneous Dirichlet boundary condition on $\partial\Omega$, 
we consider the subspace $\Uhatz := \UTHone \times \UFhonez \times \UTHtwo \times \UFhtwoz$,
where $\UFhonez$ and $\UFhtwoz$ are the subspaces of $\UFhone$ and $\UFhtwo$, respectively, 
where all the unknowns attached to faces located on $\partial\Omega$ are set to zero. 
Recall that only one of the two subdomains $\Omega_i$ touches the boundary, so that,
in practice, only one of the two face spaces is modified to account for Dirichlet conditions.

\subsection{Local reconstruction operator}

The local reconstruction operator is a central tool in the devising of HHO methods.
Recall that, in the fitted case, one possibility (see \cite{AbErPi:18,DiPDr:17}) 
is to consider, for all $T\in\mesh$, 
the operator $\GT: \UThat \rightarrow \P^k(T;\R^d)$ such that, for all 
$\uThat=(\uT, \ufT) \in \UThat$ 
and all $\q \in \P^k(T;\R^d)$,
\begin{equation}
(\GT(\uThat),\q)_{T} := (\nabla \uT, \q)_{T} + (\uF - \uT, \q \SCAL \nT)_{\dT}.
\end{equation}
In the unfitted case, the idea proposed in \cite{BCDE21} is to reconstruct a gradient
in each sub-cell $T^i$. Thus, one defines 
a pair of local operators $\GTi: \UThat \rightarrow \P^k(T^i;\R^d)$, $i\in\{1,2\}$, 
such that, for all 
$\uThat=(\uTone, \uFone, \uTtwo, \uFtwo)\in \UThat$ and all $\q \in \P^k(T^i;\R^d)$,
\begin{equation} \label{eq:old_def_GTi}
(\GTi(\uThat),\q)_{T^i} := (\nabla \uTi, \q)_{T^i} + (\udTi - \uTi, \q \SCAL \nT)_{\dTi}
- \delta_{i1}(\llbracket \uT \rrbracket_{\Gamma},\q \SCAL \nG)_{\TG},
\end{equation}
with $\llbracket \uT \rrbracket_{\Gamma} := u_{T^1}|_{\TG}-u_{T^2}|_{\TG}$
and $\delta_{i1} $ denotes the Kronecker delta.
Notice that the two gradients are reconstructed independently, but $\GTo$ 
depends on the cell components on both sides of the interface owing to the last term
in~\eqref{eq:old_def_GTi}. This term is only present when $i=1$ because we assumed 
$\kappa_1\le \kappa_2$ (to fix the ideas, otherwise the term is present only when $i=2$).
As discussed in \cite[Section~2.5]{BCDE21}, this is instrumental 
to achieve robustness of the
error estimate in the highly contrasted case $\kappa_1 \ll \kappa_2$. 

Owing to the use of polynomial extensions in the present work, 
the stencil of the local gradient reconstruction operator needs to be extended. 
Indeed, on a given mesh cell $T\in\mesh$, this operator now involves the local cell and 
face unknowns in $T$ (as before), but also the cell and face unknowns of the ill-cut cells 
$S\in \N^{-1}(T)$ (those are the cells $S\in\mesh$ such that there exists 
$i\in\{1,2\}$ so that $\N_i(S)=T$). 
Therefore, for all $T\in\mesh$, we introduce the notation
\begin{equation}
\uThatplus := (\uThat,(\uShat)_{S\in\N^{-1}(T)}) \in \UThatEX := \UThat \times \bigtimes_{S\in\N^{-1}(T)}\UShat.
\end{equation}

We can now define, for all $T \in \mesh$, the local gradient
reconstruction operators $\GTi: \UThatEX \rightarrow \P^k(T^i;\R^d)$, for all $i \in \{1,2\}$.
For all $(T,i)\in\pmesh$ and every polynomial $\q \in \P^k(T^i;\R^d)$,  
$\q^+$ denotes its extension to \rev{$\Delta(T)$ (observe that 
$T^i \cup \bigcup_{S\in\N_i^{-1}(T)} S^i\subset \Delta(T)$)}.
Then, for all $\uThatplus \in \UThatEX$,
\begin{enumerate}[label=\roman*)]
\item If $T \in \uncutmesh$, we set, for all $\q \in \P^k(T^i;\R^d)$ with $i:=\iota(T)$,
\begin{subequations}
\begin{align}
(\GTi(\uThatplus),\q)_{T^i} := {}& (\nabla \uTi, \q)_{T^i} + (\udTi - \uTi, \q \SCAL \nT)_{\dTi}
\nonumber \\
& + \sumexi \Big\{ (\udSi - \uSi, \q^+ \SCAL \nS)_{\dSi} - \delta_{i1}
(\jumpuS_{\Gamma}, \q^+ \SCAL \nG)_{S^{\Gamma}}\Big\}, \label{grad_thi} \\
\GTibar(\uThatplus):={}& \bzero,
\end{align} \end{subequations}
where the summation over $S\in \N_i^{-1}(T)$ is void if $T$ is not in the image of $\N_i$;
\item If $T \in \TOKmesh$, we set, for all $i \in \{1,2\}$ and all $\q \in \P^k(T^i;\R^d)$,
\begin{align}
(\GTi(\uThatplus),\q)_{T^i} := {}& (\nabla \uTi, \q)_{T^i} + (\udTi-\uTi, \q \SCAL \nT)_
{\dTi} - \delta_{i1}(\llbracket \uT \rrbracket_{\Gamma}, \q \SCAL \nG)_{T^{\Gamma}} \nonumber \\
&  + \sumexi \Big\{ (\udSi - \uSi, \q^+ \SCAL \nS)_{\dSi} - \delta_{i1} 
(\jumpuS_{\Gamma}, \q^+ \SCAL \nG)_{S^{\Gamma}}\Big\}; \label{grad_tok}
\end{align}
\item If $T \in \TKOmesh$, we set, for all $\overline{\q} \in \P^k(T^{\ibar};\R^d)$ with $i:=\iota(T)$,
\begin{subequations}
\begin{align}
\GTi(\uThatplus) :={} & \nabla \uTi, \label{grad_tko_grad}\\
(\GTibar(\uThatplus), \overline{\q})_{T^{\ibar}} := {}& (\nabla \uTibar, \overline{\q}
)_{T^{\ibar}} + (\udTibar - \uTibar, \overline{\q} \SCAL \nT)_{\dTibar} 
- \delta_{{\ibar}1} (\llbracket \uT \rrbracket_{\Gamma}, \overline{\q} \SCAL \nG)_{T^{\Gamma}} \nonumber \\
& + \sumexibar \left\{ (\udSibar - \uSibar, \overline{\q}^+ \SCAL \nS
)_{\dSibar} -\delta_{{\ibar}1}(\jumpuS_{\Gamma}, \overline{\q}^+ \SCAL \nG)_{S^{\Gamma}} \right\}.
\label{grad_tko}
\end{align}
\end{subequations}
\end{enumerate}

\begin{figure}
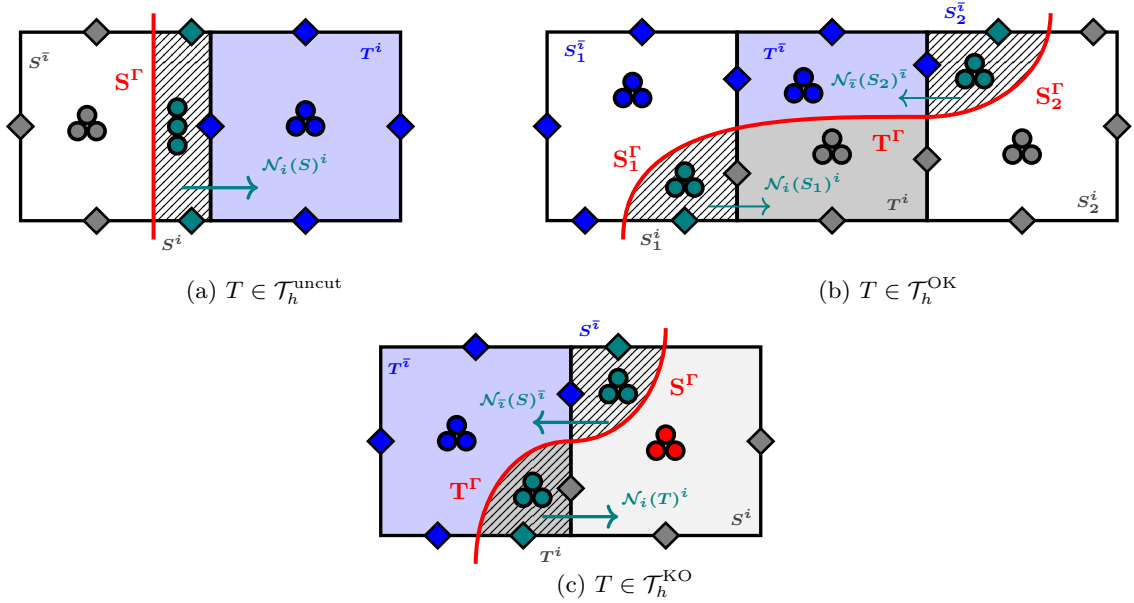

\centering
\begin{subfigure}{0.4\textwidth}
\input grad_uncut.tex
\subcaption{$T\in \uncutmesh$}
\end{subfigure}
\begin{subfigure}{0.55\textwidth}
\input grad_OK.tex
\subcaption{$T\in \TOKmesh$}
\end{subfigure}
\begin{subfigure}{0.4\textwidth}
\input grad_KO.tex
\subcaption{$T\in \TKOmesh$}
\end{subfigure}
\caption{Local stencil for gradient reconstruction operator}
\label{fig:reconstruction}
\end{figure}

Since the discretization method is formulated on the mesh sub-cells, we
consider the collection thereof, which is conveniently described as a subset of
the product set $\mesh \times \{1,2\}$. Specifically, we set 
\begin{subequations}
\begin{align}
\pmeshOK &:= \{ (T,\iota(T)) \,\mid\, T\in\uncutmesh\} \cup \{ (T,1),(T,2) \,\mid\, T\in\TOKmesh\} \cup \{ (T,\ibar), \,\mid\, T\in\TKOmesh, i:=\iota(T)\}, \\
\pmeshKO &:= \{ (T,i) \,\mid\, T\in\TKOmesh, i:=\iota(T)\}. 
\end{align}
\end{subequations}
Notice that $\pmesh := \pmeshOK\cup \pmeshKO$ is the collection of all the mesh sub-cells.
As a consequence of~\eqref{decomposition}, the subset $\pmeshKO$ can be enumerated 
as follows:
\begin{equation} \label{eq:enum_PmeshKO}
\pmeshKO = \{ (S,i) \in \pmesh \,\mid\, S\in \N_i^{-1}(T), \, (T,i)\in\pmeshOK\}.
\end{equation}

The above notation allows us to rewrite in a synthetic form the formula for the
gradient reconstruction while avoiding the proliferation of cases. Indeed,
one readily sees that the formulas \eqref{grad_thi}-\eqref{grad_tok}-\eqref{grad_tko}
are equivalent to setting, 
for all $(T,i)\in\pmeshOK$, all 
$\uThatplus \in \UThatEX$ and all $\q \in \P^k(T^i;\R^d)$,
\begin{align}
(\GTi(\uThatplus),\q)_{T^i} := {}& (\nabla \uTi, \q)_{T^i} + (\udTi-\uTi, \q \SCAL \nT)_
{\dTi} - \delta_{i1}(\llbracket \uT \rrbracket_{\Gamma}, \q \SCAL \nG)_{T^{\Gamma}} \nonumber \\
&  + \sumexi \Big\{ (\udSi - \uSi, \q^+ \SCAL \nS)_{\dSi} - \delta_{i1} 
(\jumpuS_{\Gamma}, \q^+ \SCAL \nG)_{S^{\Gamma}}\Big\},
\label{def_grad_pmeshOK}
\end{align}
where the summation over $S\in \N_i^{-1}(T)$ is void if $T$ is not in the image of $\N_i$.
Moreover, \eqref{grad_tko_grad} means that, for all $(T,i)\in\pmeshKO$, we set
\begin{equation}
\GTi(\uThatplus) := \nabla \uTi.
\label{def_grad_pmeshKO}
\end{equation}

For all $(T,1)\in\pmeshOK$, it is convenient to define the discrete lifting operator 
$\lift^k_{T^1}:
L^2(\Gamma) \rightarrow \P^k(T^1;\R^d)$ such that, for all $g\in L^2(\Gamma)$ and all 
$\q \in \P^k(T^1;\R^d)$,
\begin{equation} \label{eq:def_lift}
(\lift^k_{T^1}(g),\q)_{T^1} := (g,\q\SCAL\nG)_{\TG} + \sum_{S\in\N_1^{-1}(T)} (g,\q^+\SCAL\nG)_{\SG}.
\end{equation}
It is also convenient to define on $\Gamma$ the function given by the jump of the cell
components of a discrete HHO variable. Specifically, we define the 
function $\llbracket u_{\mesh} \rrbracket_{\Gamma}$ on $\Gamma$ such that, 
for all $T\in\cutmesh$,
$\llbracket u_{\mesh} \rrbracket_{\Gamma} := u_{T^1}|_{\TG} - u_{T^2}|_{\TG}$.
Then, for all $(T,i)\in\pmeshOK$, we can rewrite~\eqref{def_grad_pmeshOK} as follows:
\begin{align}
(\GTi(\uThatplus)+\delta_{i1}\lift^k_{T^1}(\llbracket u_{\mesh} \rrbracket_{\Gamma}),\q)_{T^i} := {}& (\nabla \uTi, \q)_{T^i} + (\udTi-\uTi, \q \SCAL \nT)_
{\dTi}  \nonumber \\
&  + \sumexi (\udSi - \uSi, \q^+ \SCAL \nS)_{\dSi}.
\label{def_grad_pmeshOK_lift}
\end{align}

\subsection{Stabilization}
\label{sec:stab}

We consider three local stabilization bilinear forms.
The first one is the usual HHO stabilization bilinear form in the 
mixed-order setting (in the spirit of the HDG
Lehrenfeld--Sch\"oberl stabilization \cite{LS:16}).
Its role is to weakly enforce the matching
between cell- and face-based HHO unknowns on the sub-faces of the mesh sub-cells
in each sub-domain. 
For all $\vhat, \what \in \Uhatz$,
we set
\begin{equation}
s_h^\circ(\vhat,\what) := \sum_{(T,i) \in \pmesh} \kappa_i 
h_T^{-1}(\projT(\vTi)-\vdTi, \projT(w_{T^i})-w_{\dTi})_{\dTi},
\end{equation}
where $\projT$ denotes the $L^2$-orthogonal projector onto $\P^k(\localfacesi;\R)$.

The second stabilization bilinear form is the same as the one
introduced in \cite{BE18,BCDE21} to control the jumps across the interface $\Gamma$. Specifically, we set, for all $\vhat, \what \in \Uhatz$,
\begin{equation}
s_h^\Gamma(\vhat,\what) := \sum_{T\in\cutmesh} \kappa_1 h_T^{-1}(\jumpvT_{\Gamma},\jumpwT_{\Gamma})_{\TG}.
\end{equation}

The third stabilization bilinear form is specific to the present setting
using polynomial extensions. Its aim\ is to provide some control on the cell
components in ill-cut cells \rev{and is devised in the same spirit as the so-called 
direct ghost penalty method from \cite{Preuss:18,LehOl:19}}. Specifically, we set, for all
$\vhat, \what \in \Uhatz$, 
\begin{equation} \label{eq:def_shN}
\rev{s_h^{\N}(\vhat,\what) := \sum_{(T,i)\in\pmeshOK} \sum_{S\in\N_i^{-1}(T)} 
\eta_{\N} \kappa_i h_T^{-2}(\vSi-\vTi^+,\wSi-\wTi^+)_{T^i},}
\end{equation}
recalling that the superscript ${}^+$ denotes the extension of a polynomial 
originally defined on $T^i$ to \rev{$\Delta(T)$, and $\eta_{\N}>0$ is a user-defined 
parameter.} 

Putting everything together, we define, for all $\vhat, \what \in \Uhatz$,
\begin{equation} \label{eq:def.sh}
s_h(\vhat, \what) := s_h^{\circ}(\vhat,\what) + s_h^{\Gamma}(\vhat,\what) + s_h^{\N}(\vhat,\what).
\end{equation}

\subsection{Discrete problem}

The global discrete problem reads as follows: Find $\uhat \in \Uhatz$ such that,
\begin{equation}
a_h(\uhat, \what) = \ell_h(\what) \quad \forall \what \in \Uhatz,
\label{problem}
\end{equation}
where the bilinear $a_h$ is defined as follows: For all
$\vhat, \what \in \Uhatz$,
\begin{equation}
a_h(\vhat,\what) := b_h(\vhat,\what)
+ s_h(\vhat, \what),
\end{equation}
with
\begin{equation} \label{eq:def.bh}
b_h(\vhat,\what) := \sum_{(T,i)\in\pmesh} \kappa_i (\GTi(\vThatplus), \GTi
(\wThatplus))_{T^i}.
\end{equation}
Moreover, the linear form $\ell_h$ is such that, for all $\what \in \Uhatz$,
\begin{align} 
\ell_h (\what) := {}&\!\!\sum_{(T,i) \in \pmesh}\!\! (f,w_{T^i})_{T^i} + \!\!\sum_{T\in\cutmesh}\!\! \Big\{ (g_N,w_{T^2})_{\TG} + \kappa_1h_T^{-1}(g_D,\jumpwT_\Gamma)_{\TG}\Big\} \nonumber \\
&- \!\!\sum_{(T,1)\in\pmeshOK}\!\!\kappa_1(\lift^k_{T^1}(g_D),\GTo(\wThatplus))_{T^1}, \label{eq:def.ellh}
\end{align}
where the lifting operator $\lift^k_{T^1}$ is defined in~\eqref{eq:def_lift}. 
Notice that, in the discrete problem~\eqref{problem}, only the Dirichlet boundary
condition is enforced explicitly (on the face unknowns located on the boundary 
$\partial\Omega$), whereas both jump conditions across the interface, \eqref{eq:strong.D}
and \eqref{eq:strong.N}, are enforced weakly. 
We notice that cell unknowns can still be eliminated from~\eqref{problem}
by static condensation. Whenever $\N^{-1}(T)$ is nonempty for a mesh cell $T\in\mesh$,
all the cell unknowns attached to $T\cup \bigcup_{S\in\N^{-1}(T)}S$ are blocked together.

To prepare for the consistency error analysis, we state a useful result on the 
discrete bilinear and linear forms. \rev{Although this lemma is only used in 
Section~\ref{sec:analysis}, it is given here so as to highlight why the
definitions~\eqref{def_grad_pmeshOK}-\eqref{def_grad_pmeshKO} of the reconstructed
gradient are meaningful.}

\begin{lemma}[Consistency: preparatory identity] \label{lem:preparatory}
Let $\rev{u\upex}$ be the weak solution to~\eqref{eq:weak}. 
Let $\vhat\in\Uhatz$ be arbitrary and set 
\begin{subequations} \begin{alignat}{2}
\bdelta_{T,i} &:= \GTi(\vThatplus)
+ \delta_{i1} \lift_{T^1}^k(\jumpu_\Gamma) - \nabla \rev{u\upex_i}|_{T^i}, 
&\qquad &\forall (T,i)\in\rev{\pmeshOK}, \\
\bdelta_{T,S,i} &:= \big\{\GTi(\vThatplus)
+ \delta_{i1} \lift_{T^1}^k(\jumpu_\Gamma)\big\}^+|_{S^i} - \nabla \rev{u\upex_i}|_{S^i}, 
&\qquad &\forall (T,i)\in\pmeshOK, \, \forall S\in \N_i^{-1}(T),\\
\rev{\bdelta_{T,i}} &:= \rev{\GTi(\vThatplus)- \nabla \rev{u\upex_i}|_{T^i} = \nabla \vTi - \nabla \rev{u\upex_i}|_{T^i},} 
&\qquad &\rev{\forall (T,i)\in\rev{\pmeshKO}}.
\end{alignat}\end{subequations}
Then, the following holds for all $\what \in \Uhatz$,
\begin{align}
a_h(\vhat,\what)-\ell_h(\what) = {}& 
\sum_{(T,i)\in\pmesh} \kappa_i (\bdelta_{T,i},\nabla \wTi)_{T^i} \nonumber \\
& + \sum_{(T,i)\in\pmeshOK} \kappa_i \bigg\{ (\bdelta_{T,i} \SCAL \nT,\wdTi-\wTi)_{\dTi} 
- \delta_{i1}(\bdelta_{T,1} \SCAL \nG,\jumpwT_{\Gamma})_{T^{\Gamma}} \nonumber \\
& + \sum_{S\in\N_i^{-1}(T)}  \big\{ 
(\bdelta_{T,S,i} \SCAL \nS,\wdSi - \wSi)_{\dSi} - \delta_{i1} 
(\bdelta_{T,S,1} \SCAL \nG,\jumpwS_{\Gamma})_{\SG}\big\}\bigg\} \nonumber \\
& + s_h^\circ(\vhat,\what) + s_h^\N(\vhat,\what) + \sum_{T\in\cutmesh} \kappa_1h_T^{-1}
(\llbracket v_T-\rev{u\upex}\rrbracket_\Gamma,\jumpwT_\Gamma)_{\TG}.
\label{eq:identity_consist}
\end{align}
\end{lemma}

\begin{proof}
(1) Recalling the properties~\eqref{eq:strong.PDE}-\eqref{eq:strong.N} satisfied by the exact 
solution and the definition~\eqref{eq:def.ellh} of the linear form $\ell_h$, we infer
that, for all $\what \in \Uhatz$, $\ell_h(\what) = L_1+L_2+L_3$ with
\begin{align*}
L_1 := {}& \sum_{(T,i)\in\pmesh} -(\nabla{\cdot}(\kappa_i\nabla \rev{u\upex_i}),w_{T^i})_{T^i}
+ \sum_{T\in\cutmesh} ((\kappa_1\nabla \rev{u\upex_1}-\kappa_2\nabla \rev{u\upex_2})\SCAL\nG,w_{T^2})_{\TG}, \\
L_2 := {}& \sum_{T\in\cutmesh} \kappa_1 h_T^{-1}(\jumpu_\Gamma,\jumpwT_\Gamma)_{\TG}, \qquad
L_3 :=- \sum_{(T,1)\in\pmeshOK} \kappa_1(\lift^k_{T^1}(\jumpu_\Gamma),\GTo(\wThatplus))_{T^1}.
\end{align*}
Integrating by parts and re-organizing the boundary term on $\Gamma$ (recall that
$\bn_{T^1}=-\bn_{T^2}=\nG$) gives 
\begin{align*}
\sum_{(T,i)\in\pmesh} -(\nabla{\cdot}(\kappa_i\nabla \rev{u\upex_i}),w_{T^i})_{T^i} = {}& 
\sum_{(T,i)\in\pmesh} \kappa_i\big\{ (\nabla \rev{u\upex_i},\nabla w_{T^i})_{T^i}
- (\nabla \rev{u\upex_i} \SCAL \nT,w_{T^i})_{\dTi} \big\} \\
&- \sum_{T\in\cutmesh}
\big\{ (\kappa_1\nabla \rev{u\upex_1}\SCAL\nG,w_{T^1})_{\TG}-(\kappa_2\nabla \rev{u\upex_2}\SCAL\nG,w_{T^2})_{\TG}\big\}.
\end{align*}
Since $\nabla \rev{u\upex_i}$ is single-valued 
on $\dTi\cap \Omega_i$ and $w_{\dTi}$ vanishes on the mesh boundary faces, we obtain
\begin{align*}
\sum_{(T,i)\in\pmesh} -(\nabla{\cdot}(\kappa_i\nabla \rev{u\upex_i}),w_{T^i})_{T^i}
= {}& \sum_{(T,i)\in\pmesh} \kappa_i\big\{ (\nabla \rev{u\upex_i},\nabla w_{T^i})_{T^i}
+ (\nabla \rev{u\upex_i} \SCAL \nT,w_{\dTi}-w_{T^i})_{\dTi} \big\} \\
&- \sum_{T\in\cutmesh}
\big\{ (\kappa_1\nabla \rev{u\upex_1}\SCAL\nG,w_{T^1})_{\TG}-(\kappa_2\nabla \rev{u\upex_2}\SCAL\nG,w_{T^2})_{\TG}\big\}.
\end{align*}
Therefore, we have
\[
L_1 = \sum_{(T,i)\in\pmesh} \kappa_i\Big\{ (\nabla \rev{u\upex_i},\nabla w_{T^i})_{T^i}
+ (\nabla \rev{u\upex_i} \SCAL \nT,w_{\dTi}-w_{T^i})_{\dTi} \Big\} - \sum_{T\in\cutmesh} \kappa_1(\nabla \rev{u\upex_1}\SCAL\nG,\jumpwT_{\Gamma})_{\TG}.
\]
A rewriting of the last term on the right-hand side gives
\[
L_1 = \sum_{(T,i)\in\pmesh} \kappa_i\Big\{ (\nabla \rev{u\upex_i},\nabla w_{T^i})_{T^i}
+ (\nabla \rev{u\upex_i} \SCAL \nT,w_{\dTi}-w_{T^i})_{\dTi} - \delta_{i1} (\nabla \rev{u\upex_1}\SCAL\nG,\jumpwT_{\Gamma})_{\TG}\Big\}.
\]
We now use the enumeration formula~\eqref{eq:enum_PmeshKO} for the pairs $(T,i)\in
\pmeshKO$. This gives
\begin{align*}
L_1 = {}& \sum_{(T,i)\in\pmesh} \kappa_i (\nabla \rev{u\upex_i},\nabla \wTi)_{T^i} \\
& + \sum_{(T,i)\in\pmeshOK} \kappa_i \bigg\{ (\nabla \rev{u\upex_i} \SCAL \nT,\wdTi-\wTi)_{\dTi} 
- \delta_{i1}(\nabla \rev{u\upex_1} \SCAL \nG,\jumpwT_{\Gamma})_{T^{\Gamma}} \\
& + \sum_{S\in\N_i^{-1}(T)}  \big\{ 
(\nabla \rev{u\upex_i} \SCAL \nS,\wdSi - \wSi)_{\dSi} - \delta_{i1} 
(\nabla \rev{u\upex_1} \SCAL \nG,\jumpwS_{\Gamma})_{\SG}\big\} \bigg\}.
\end{align*}

(2) For all $\vhat \in \Uhatz$, setting $\bgamma_{T,i}:= \GTi(\vThatplus)
+ \delta_{i1} \lift_{T^1}^k(\jumpu_\Gamma)$ for all $(T,i)\in\pmeshOK$
and $\bgamma_{T,i}:= \GTi(\vThatplus) \rev{=\nabla \vTi}$ for all $(T,i)\in\pmeshKO$, 
and using the definition~\eqref{eq:def.bh} of the bilinear form $b_h$ \rev{together
with the definition of $L_3$ given above}, we infer that
\begin{align*}
b_h(\vhat,\what) - L_3 = {}&\sum_{(T,i)\in\pmesh} \kappa_i (\GTi(\vThatplus), \GTi
(\wThatplus))_{T^i} + \sum_{(T,1)\in\pmeshOK} \kappa_1(\lift^k_{T^1}(\jumpu_\Gamma),\GTo(\wThatplus))_{T^1}\\
= {}& \sum_{(T,i)\in\pmeshOK} \kappa_i (\bgamma_{T,i}, \GTi
(\wThatplus))_{T^i} + \sum_{(T,i)\in\pmeshKO} \kappa_i (\bgamma_{T,i}, \GTi
(\wThatplus))_{T^i}.
\end{align*}
Using the definitions~\eqref{def_grad_pmeshOK}-\eqref{def_grad_pmeshKO} of the 
gradient reconstruction operator and 
since $\bgamma_{T,i}\in \P^k(T^i;\R^d)$ for all $(T,i)\in\pmesh$, we obtain
\begin{align*}
b_h(\vhat,\what) = {}& 
\sum_{(T,i)\in\pmesh} \kappa_i (\bgamma_{T,i},\nabla \wTi)_{T^i} \\
& + \sum_{(T,i)\in\pmeshOK} \kappa_i \bigg\{ (\bgamma_{T,i} \SCAL \nT,\wdTi-\wTi)_{\dTi} 
- \delta_{i1}(\bgamma_{T,1} \SCAL \nG,\jumpwT_{\Gamma})_{T^{\Gamma}} \\
& + \sum_{S\in\N_i^{-1}(T)}  \big\{ 
(\bgamma_{T,i}^+ \SCAL \nS,\wdSi - \wSi)_{\dSi} - \delta_{i1} 
(\bgamma_{T,1}^+ \SCAL \nG,\jumpwS_{\Gamma})_{\SG}\big\}\bigg\}.
\end{align*}

(3) Combining the identities from Steps (1) and (2) and recalling that $a_h=b_h+s_h$
completes the proof.
\end{proof}

\section{Stability, consistency, and error estimate}
\label{sec:analysis}

In this section, we perform the stability and error analysis for the unfitted HHO
method introduced in the previous section.  
We use the convention $A \lesssim B$ to abbreviate the inequality $A \leq C B$ for
positive real numbers $A$ and $B$, where the constant $C$ only depends on 
the polynomial degree, the mesh shape-regularity, the parameter $\vartheta$ used in the definition of the pairing operator, and the space dimension.

\subsection{Analysis tools}
\label{sec:tools}

In this section, we present the main tools to perform the error analysis. 
The proofs are postponed to Section~\ref{sec:proofs}.
For all $(T,i)\in \pmeshOK$, \rev{recall that} the superscript ${}^+$ is used here to
indicate the extension of a polynomial originally defined on $T^i$ to $\Delta(T)$
(observe that $T\cup \bigcup_{S\in\N_i^{-1}(T)}S\subset \Delta(T)$).

\begin{lemma}[Discrete inverse inequalities] \label{lem:disc_trace_app}
The following holds for all $(T,i)\in\pmeshOK$, all $\phi\in \P^\ell(T^i;\R)$, and all $\ell\ge0$, 
\begin{subequations} \begin{align}
&\sum_{S\in\{T\}\cup \N_i^{-1}(T)} \Big\{ \|\phi^+\|_{S} +
h_S^{\frac12} \|\phi^+\|_{\dSi\cup\SG} \Big\} \lesssim \|\phi\|_{T^i},
\label{eq:disc_trace_OK} \\
&\sum_{S\in\{T\}\cup \N_i^{-1}(T)} 
h_S^{-\frac12} \|(I-\projS)(\phi^+)\|_{\dSi} \lesssim \|\nabla \phi\|_{T^i}.
\label{eq:disc.poinc}
\end{align} \end{subequations}
\end{lemma}

For all $s\ge0$ and all $i\in\{1,2\}$, let $E_i^s:H^s(\Omega_i)\rightarrow H^s(\R^d)$ be a 
stable extension operator. 
For all $v\in H^s(\Omega_1\cup \Omega_2)$ with $s>\frac12$ and all $(T,i)\in\pmeshOK$, we define
\begin{equation} \label{eq:def_IkT}
I^{k+1}_{T^i}(v_i) := \Pi^{k+1}_{T}(E_i^s(v_i))|_{T^i} \in \P^{k+1}(T^i;\R), 
\end{equation}
where $\Pi^{k+1}_{T}$ denotes the $L^2$-orthogonal projection onto 
$\P^{k+1}(T;\R)$. (Notice that the operator $I^{k+1}_{T^i}$ depends on the Sobolev 
index $s$, but this dependency is not tracked to simplify the notation; notice also
that the extension operator is not needed if $T$ is an uncut cell.)
Let us set\rev{, for all $(T,i)\in\pmeshOK$,} 
\begin{subequations} \begin{align}
\epsilon_{T,i}(v_i) := {} & \sum_{S\in\{T\}\cup\N_i^{-1}(T)} \!\! \Big\{ \|v_i-I_{T^i}^{k+1}(v_i)^+\|_{S^i} 
+ h_S^{\frac12}\|v_i-I_{T^i}^{k+1}(v_i)^+\|_{\dSi}  \nonumber \\
& \hspace{1.75cm} + h_S \|\nabla(v_i-I_{T^i}^{k+1}(v_i)^+)\|_{S^i}
+ h_S^{\frac32} \|\nabla(v_i-I_{T^i}^{k+1}(v_i)^+)\|_{\dSi} \Big\}, \label{eq:def_epsilon} \\
\rev{\epsilon^\Gamma_{T,i}(v)} := {}& \rev{\sum_{S\in\{T\}\cup\N_i^{-1}(T)}  
h_S^{\frac12} \| \llbracket v-I_T^{k+1}(v)^+ \rrbracket_{\Gamma}\|_{\SG} }, \label{eq:def_epsilon_G}
\end{align} \end{subequations} 
\rev{where $\llbracket v-I_T^{k+1}(v)^+ \rrbracket_{\Gamma}|_{\SG}:=(v_1-I_{T^1}^{k+1}(v_1))|_{\SG}-(v_2-I_{T^2}^{k+1}(v_2))|_{\SG}$.}

\begin{lemma}[Approximation] \label{lem:approx.IkT}
For all $v\in H^s(\Omega_1\cup \Omega_2)$ with $s\in(\frac32,k+2]$ 
and all $(T,i)\in\pmeshOK$, we have
\begin{subequations} \begin{align} 
\epsilon_{T,i}(v_i) &\lesssim h_T^s |E_i^s(v_i)|_{H^s(\Delta(T))}, \label{eq:approx.i} \\
\rev{\epsilon^\Gamma_{T,i}(v)}
&\lesssim h_T^{s} \sum_{i\in\{1,2\}} |E_i^s(v_i)|_{H^s(\Delta_2(T))}. \label{eq:approx.jump}
\end{align} \end{subequations}
\end{lemma}

\subsection{Stability and well-posedness}
\label{Stability}

We introduce the following norm on $\Uhatz$: For all $\vhat\in\Uhatz$,
\begin{equation} \label{eq:stab.norm}
\|\vhat\|_{h0}^2 := \sum_{(T,i)\in\pmesh} 
\kappa_i\|\nabla \vTi\|_{T^i}^2 + |\vhat|_{\textsc{s}}^2 + s_h^\Gamma(\vhat,\vhat)
+ s_h^\N(\vhat,\vhat),
\end{equation}
where 
\begin{equation}
|\vhat|_{\textsc{s}}^2 := \sum_{(T,i)\in\pmesh} \kappa_ih_T^{-1}\|\vdTi - \vTi \|_{\dTi}^2,
\end{equation}
and the stabilization bilinear forms are defined in Section~\ref{sec:stab}. 
We observe that in general $|\vhat|_{\textsc{s}}^2 \ne s_h^\circ(\vhat,\vhat)$ since the 
cell component is projected when evaluating the latter. Furthermore, 
we notice the following rewriting:
\begin{align}
\|\vhat\|_{h0}^2 = {}& \sum_{(T,i)\in\pmesh} 
\kappa_i\Big\{ \|\nabla \vTi\|_{T^i}^2 + h_T^{-1}\|
\vdTi - \vTi \|_{\dTi}^2 + \delta_{i1} h_T^{-1} \|\jumpvT_\Gamma\|_{\TG}^2\Big\} \nonumber \\
& + \sum_{(T,i)\in\pmeshOK} \sum_{S\in \N_i^{-1}(T)} \eta_{\N}\kappa_i h_{T}^{-2} \|\vSi-\vTi^+\|_{T^i}^2.
\end{align}
It is straightforward to verify that~\eqref{eq:stab.norm} defines a norm on $\Uhatz$. Indeed,
if $\vhat\in\Uhatz$ satisfies $\|\vhat\|_{h0}=0$, then all the cell unknowns and all
the face unknowns are constant and take the same value inside each sub-domain $\Omega_i$ 
and globally in the domain $\Omega$. Since the face unknowns on the boundary $\partial\Omega$
vanish, we conclude that all the components of $\vhat$ are zero.

\begin{lemma}[Stability and boundedness]
\label{stabilitylemma}
We have, for all $\vhat \in \Uhatz$,
\begin{equation} \label{eq:stab.bnd}
\|\vhat\|_{h0}^2 \lesssim a_h(\vhat,\vhat) \lesssim \|\vhat\|_{h0}^2.
\end{equation}
\end{lemma}

\begin{proof}
(1) Bounds on reconstructed gradient. Let $(T,i)\in\pmeshOK$. Then, taking 
$\q:= \nabla \vTi \in \P^k(T^i;\R^d)$ in~\eqref{def_grad_pmeshOK}, we obtain
\begin{align*}
\|\nabla \vTi\|_{T^i}^2 = {}& (\GTi(\vThatplus),\nabla \vTi)_{T^i} - (\vdTi-\vTi, \nabla \vTi \SCAL \nT)_{\dTi} + \delta_{i1}(\jumpvT_{\Gamma}, \nabla \vTi \SCAL \nG)_{T^{\Gamma}} \\
&  - \sumexi \Big\{ (\vdSi - \vSi, (\nabla \vTi)^+ \SCAL \nS)_{\dSi} - \delta_{i1} 
(\jumpvS_{\Gamma}, (\nabla \vTi)^+ \SCAL \nG)_{S^{\Gamma}}\Big\}.
\end{align*}
Since $\nabla \vTi \SCAL \nT \in \P^k(\localfacesi;\R)$ and 
$(\nabla \vTi)^+ \SCAL \nS \in \P^k(\extendedfacesi;\R)$, we infer that
\begin{align*}
\|\nabla \vTi\|_{T^i}^2 = {}& (\GTi(\vThatplus),\nabla \vTi)_{T^i} - (\vdTi-\projT(\vTi), \nabla \vTi \SCAL \nT)_{\dTi} + \delta_{i1}(\jumpvT_{\Gamma}, \nabla \vTi \SCAL \nG)_{T^{\Gamma}} \\
&  - \sumexi \Big\{ (\vdSi - \projS(\vSi), (\nabla \vTi)^+ \SCAL \nS)_{\dSi} - \delta_{i1} 
(\jumpvS_{\Gamma}, (\nabla \vTi)^+ \SCAL \nG)_{S^{\Gamma}}\Big\}.
\end{align*}
Invoking the Cauchy--Schwarz inequality and the discrete trace 
inequality~\eqref{eq:disc_trace_OK}, we obtain
\[
\|\nabla \vTi\|_{T^i}^2 \lesssim \|\GTi(\vThatplus)\|_{T^i}^2 + 
\sum_{S\in\{T\}\cup \N_i^{-1}(T)} \Big\{ h_S^{-1}
\|\vdSi-\projS(\vSi)\|_{\dSi}^2 + \delta_{i1}h_S^{-1}\|\jumpvS_{\Gamma}\|_{\SG}^2\Big\}. 
\]
Proceeding similarly proves that
\[
\|\GTi(\vThatplus)\|_{T^i}^2 \lesssim \|\nabla \vTi\|_{T^i}^2 + 
\sum_{S\in\{T\}\cup \N_i^{-1}(T)} \Big\{ h_S^{-1}
\|\vdSi-\projS(\vSi)\|_{\dSi}^2 + \delta_{i1}h_S^{-1}\|\jumpvS_{\Gamma}\|_{\SG}^2\Big\}. 
\]
Moreover, for all $(T,i)\in\pmeshKO$, we have $\GTi(\vThatplus) = \nabla \vTi$ in $T^i$. 
Therefore, multiplying by $\kappa_i$ and summing over all $(T,i)\in\pmesh$, we infer that
\begin{subequations} \begin{align}
\sum_{(T,i)\in\pmesh} \kappa_i \|\nabla \vTi\|_{T^i}^2 & \lesssim b_h(\vhat,\vhat) + s_h^\circ(\vhat,\vhat) + s_h^\Gamma(\vhat,\vhat), \label{eq:stab.1} \\
b_h(\vhat,\vhat) &\lesssim \sum_{(T,i)\in\pmesh} \kappa_i \|\nabla \vTi\|_{T^i}^2 + s_h^\circ(\vhat,\vhat) + s_h^\Gamma(\vhat,\vhat). \label{eq:stab.2}
\end{align} \end{subequations}
Since $s_h^\circ(\vhat,\vhat) \le |\vhat|_{\textsc{s}}^2$, the upper bound in~\eqref{eq:stab.bnd} readily follows from~\eqref{eq:stab.2}. 

(2) To prove the lower bound in~\eqref{eq:stab.bnd}, it remains to estimate $|\vhat|_{\textsc{s}}^2$. We first observe that
\[
|\vhat|_{\textsc{s}}^2 = \sum_{(T,i)\in\pmeshOK} \kappa_i \sum_{S\in\{T\}\cup \N_i^{-1}(T)} 
h_S^{-1} \|\vdSi-\vSi\|_{\dSi}^2.
\]
Consider first the case $S=T$. The triangle inequality gives
\[
\|\vdTi-\vTi\|_{\dTi} \le \|\vdTi-\projT(\vTi)\|_{\dTi} + \|(I-\projT)(\vTi)\|_{\dTi},
\]
and owing to~\eqref{eq:disc.poinc}, we infer that
\begin{equation} \label{eq:bnd_on_Ti}
\|\vdTi-\vTi\|_{\dTi} \lesssim \|\vdTi-\projT(\vTi)\|_{\dTi} + h_T^{\frac12} \|\nabla v_{T^i}\|_{T^i}.
\end{equation}
Let now $S\in \N_i^{-1}(T)$. The triangle inequality implies that
\begin{align*}
\|\vdSi-\vSi\|_{\dSi} \le {}& \|\vdSi-\projS(\vSi)\|_{\dSi}
+ \|(I-\projS)(\vSi-\vTi^+)\|_{\dSi} \\
& + \|(I-\projS)(\vTi^+)\|_{\dSi}.
\end{align*}
For the second term on the right-hand side, we simply
notice that $\|(I-\projS)(\vSi-\vTi^+)\|_{\dSi}\le \|\vSi-\vTi^+\|_{\dSi}$, whereas we invoke~\eqref{eq:disc.poinc} to bound the third term. Altogether, this gives
\[
\|\vdSi-\vSi\|_{\dSi} \lesssim \|\vdSi-\projS(\vSi)\|_{\dSi} + \|\vSi-\vTi^+\|_{\dSi}
+ h_S^{\frac12} \|\nabla \vTi\|_{T^i}.
\]
\rev{The first term on the right-hand side is bounded using the $L^2$-stability of $\projS$ and~\eqref{eq:bnd_on_Ti}. The second term is controlled by means of the discrete trace inequality~\eqref{eq:disc_trace_OK} and the definition~\eqref{eq:def_shN} of the stabilization bilinear from $s_h{^\N}$.} Squaring the resulting inequality, multiplying by $\kappa_ih_S^{-1}$, summing over $S\in \N_i^{-1}(T)$, and finally summing over $(T,i)\in\pmeshOK$, we obtain
\[
|\vhat|_{\textsc{s}}^2 \lesssim s_h^\circ(\vhat,\vhat) + s_h^\N(\vhat,\vhat) + \sum_{(T,i)\in\pmesh} \kappa_i \|\nabla \vTi\|_{T^i}^2.
\]
Combining this bound with~\eqref{eq:stab.1} gives
\[
\sum_{(T,i)\in\pmesh} \kappa_i \|\nabla \vTi\|_{T^i}^2 + |\vhat|_{\textsc{s}}^2  
\lesssim b_h(\vhat,\vhat) + s_h^\circ(\vhat,\vhat) + s_h^\Gamma(\vhat,\vhat) + s_h^\N(\vhat,\vhat) = a_h(\vhat,\vhat),
\] 
whence the lower bound in~\eqref{eq:stab.bnd} readily follows.
\end{proof}

\subsection{Interpolation operator}

For all $T\in\mesh$, the local interpolation operator is defined as follows:
\begin{equation}
\hITk(v) := (\hITko(v),\hITkt(v)) := (I^{k+1}_{T^1}(v_1),\Pi_{\dTone}^k(v_1),I^{k+1}_{T^2}(v_2),\Pi_{\dTtwo}^k(v_2)) \in \UThat,
\end{equation}
where $I^{k+1}_{T^i}(v_i)$ is defined in~\eqref{eq:def_IkT} for all $(T,i)\in\pmeshOK$,
whereas for all $(S,i)\in\pmeshKO$, we set
\begin{equation} \label{eq:def_IkT_KO}
I^{k+1}_{S^i}(v_i) := I^{k+1}_{T^i}(v_i)^+|_{S^i}, \qquad T:=\N_i(S),
\end{equation} 
where we recall that the superscript ${}^+$ is used to indicate the extension of a polynomial originally defined on $T^i$ to \rev{$\Delta(T)$}.

The global interpolation operator is denoted by 
$\hIhk:H^s(\Omega_1\cup\Omega_2)\rightarrow \Uhat$ and is such that the local components 
of $\hIhk(v)$ on a mesh cell $T\in\mesh$ are those given by $\hITk(v|_T)$.
Notice that $\hIhk(u) \in \Uhatz$ for the exact solution since $u|_{\partial\Omega}=0$.

\begin{lemma}[Approximation] \label{lem:approx.interp}
For all $v\in H^s(\Omega_1\cup \Omega_2)$, $s>\frac32$, and all $(T,i)\in\pmeshOK$, set
\begin{subequations}  \label{eq:def_bdelta} \begin{align}
\bdelta_{T,i}(v)&:= \GTi\big(\hITk(v)^{\N}\big) + \delta_{i1} \lift_{T^1}^k(\llbracket v \rrbracket_\Gamma) - \nabla v_i|_{T^i},  \\
\bdelta_{T,S,i}(v) &:= \big\{ \GTi\big(\hITk(v)^{\N}\big) + \delta_{i1} \lift_{T^1}^k(\llbracket v \rrbracket_\Gamma)\big\}^+|_{S^i} - \nabla v_i|_{S^i}, \qquad \forall S\in \N_i^{-1}(T),
\end{align} \end{subequations}
and recall the error measure\rev{s $\epsilon_{T,i}(v_i)$ and $\epsilon^\Gamma_{T,i}(v)$ defined in~\eqref{eq:def_epsilon}-\eqref{eq:def_epsilon_G}}.
Then, the following holds:
\begin{subequations}
\begin{align} 
\|\bdelta_{T,i}(v)\|_{T^i}  + h_T^{\frac12} \|\bdelta_{T,i}(v)\|_{\dTi\cup \TG}
&\lesssim h_T^{-1} \big\{ \epsilon_{T,i}(v_i) + \delta_{i1}\rev{\epsilon^\Gamma_{T,i}(v)} \big\}, \label{eq:bnd_bdelta} \\
\label{eq:bnd_bdelta_S}
\sumexi \Big\{ \|\bdelta_{T,S,i}(v)\|_{S^i} + h_S^{\frac12} \|\bdelta_{T,S,i}(v)\|_{\dSi\cup\SG}\Big\} 
&\lesssim h_T^{-1} \big\{ \epsilon_{T,i}(v_i) + \delta_{i1}\rev{\epsilon^\Gamma_{T,i}(v)} \big\}.
\end{align}
\end{subequations}
\end{lemma}

\begin{proof}
Let $v\in H^s(\Omega_1\cup \Omega_2)$, $s>\frac32$, and 
let $(T,i)\in\pmeshOK$. Set 
\[
\bdelta'_{T,i}(v):= \GTi\big(\hITk(v)^{\N}\big) 
+ \delta_{i1} \lift_{T^1}^k(\llbracket v \rrbracket_\Gamma) - \nabla I_{T^i}^{k+1}(v_i)
\in \P^{k+1}(T^i;\R).
\] 

(1) Let $I_{\mesh}^{k+1}$ be composed of the two cell components of $\hIhk$.
We notice that
\begin{align*}
\|\bdelta'_{T,i}(v)\|_{T^i}^2 = {}& (\GTi\big(\hITk(v)^{\N}\big) 
+ \delta_{i1} \lift_{T^1}^k(\llbracket v \rrbracket_\Gamma) - \nabla I_{T^i}^{k+1}(v_i),
\bdelta'_{T,i}(v))_{T^i} \\ 
= {}& (\GTi\big(\hITk(v)^{\N}\big) 
+ \delta_{i1} \lift_{T^1}^k(\llbracket I_{\mesh}^{k+1}(v) \rrbracket_\Gamma) - \nabla I_{T^i}^{k+1}(v_i),
\bdelta'_{T,i}(v))_{T^i}  \\
& + 
\delta_{i1} (\lift_{T^1}^k(\llbracket v-I_{\mesh}^{k+1}(v) \rrbracket_\Gamma), \bdelta'_{T,i}(v))_{T^i}.
\end{align*}
Owing to~\eqref{def_grad_pmeshOK_lift} and the definition of the operators $\hITk$ and $\lift^k_{T^1}$, we infer that
\begin{align*}
\|\bdelta'_{T,i}(v)\|_{T^i}^2 = {}& (\projT(v_i)-I_{T^i}^{k+1}(v_i), \bdelta'_{T,i}(v) \SCAL \nT)_
{\dTi} + \delta_{i1}(\llbracket v-I_{T}^{k+1}(v) \rrbracket_{\Gamma}, \bdelta'_{T,1}(v) \SCAL \nG)_{T^{\Gamma}} \nonumber \\
&  + \sumexi \Big\{ (\projS(v_i) - I_{T^i}^{k+1}(v_i)^+, \bdelta'_{T,i}(v)^+ \SCAL \nS)_{\dSi} 
\\ & \qquad
+ \delta_{i1} (\llbracket v-I_{T}^{k+1}(v)^+ \rrbracket_{\Gamma}, 
\bdelta'_{T,1}(v)^+ \SCAL \nG)_{S^{\Gamma}}\Big\}.
\end{align*}
Invoking the definition of the $L^2$-orthogonal projections $\projT$ and $\projS$ gives
\begin{align*}
\|\bdelta'_{T,i}(v)\|_{T^i}^2 = {}&(v_i-I_{T^i}^{k+1}(v_i), \bdelta'_{T,i}(v) \SCAL \nT)_
{\dTi} + \delta_{i1}(\llbracket v-I_{T}^{k+1}(v) \rrbracket_{\Gamma}, \bdelta'_{T,1}(v) \SCAL \nG)_{T^{\Gamma}} \nonumber \\
&  + \sumexi \Big\{ (v_i - I_{T^i}^{k+1}(v_i)^+, \bdelta'_{T,i}(v)^+ \SCAL \nS)_{\dSi} 
\ifHAL \cuthereq \fi 
+ \delta_{i1} (\llbracket v-I_{T}^{k+1}(v)^+ \rrbracket_{\Gamma}, 
\bdelta'_{T,1}(v)^+ \SCAL \nG)_{S^{\Gamma}}\Big\}.
\end{align*}
Invoking the Cauchy--Schwarz inequality and the discrete trace inequality~\eqref{eq:disc_trace_OK}, we infer that
\begin{align*}
h_T^{\frac12} \|\bdelta'_{T,i}(v)\|_{T^i} \lesssim {}& \|v_i-I_{T^i}^{k+1}(v_i)\|_{\dTi}
+ \delta_{i1} \|\llbracket v-I_{T}^{k+1}(v) \rrbracket_{\Gamma}\|_{\TG} \\
&  + \sumexi \Big\{ \|v_i - I_{T^i}^{k+1}(v_i)^+\|_{\dSi} + \delta_{i1} \|\llbracket v-I_{T}^{k+1}(v)^+ \rrbracket_{\Gamma}\|_{\SG} \Big\}.
\end{align*}
Recalling the definition~\eqref{eq:def_epsilon} of $\epsilon_{T,i}(v)$ \rev{and the definition~\eqref{eq:def_epsilon_G} of $\epsilon^\Gamma_{T,i}(v)$} gives
\begin{equation} \label{eq:bnd_bdelta'}
\|\bdelta'_{T,i}(v)\|_{T^i} \lesssim h_T^{-1} \big\{ \epsilon_{T,i}(v_i) + \delta_{i1}\rev{\epsilon^\Gamma_{T,i}(v)} \big\},
\end{equation}
and owing to the discrete trace inequality~\eqref{eq:disc_trace_OK}, we obtain
\[
\|\bdelta'_{T,i}(v)\|_{T^i} + h_T^{\frac12} \|\bdelta'_{T,i}(v)\|_{\dTi\cup \TG}
\lesssim \|\bdelta'_{T,i}(v)\|_{T^i}
\lesssim h_T^{-1} \big\{ \epsilon_{T,i}(v_i) + \delta_{i1}\rev{\epsilon^\Gamma_{T,i}(v)} \big\}.
\]
Finally, since $\bdelta_{T,i}(v)=\bdelta'_{T,i}(v) - \nabla(v_i-I_{T^i}^{k+1}(v_i))$,
invoking the triangle inequality proves~\eqref{eq:bnd_bdelta}.

(2) Assuming that $\N_i^{-1}(T)$ is nonempty, for 
all $S\in \N_i^{-1}(T)$, we have
\[
\bdelta_{T,S,i}(v) = \bdelta'_{T,i}(v)^+|_{S^i} + \nabla (I_{T^i}^{k+1}(v_i)^+|_{S^i}
- v_i|_{S^i}).
\]
The discrete trace inequality~\eqref{eq:disc_trace_OK}
and \eqref{eq:bnd_bdelta'} imply that
\[
\|\bdelta'_{T,i}(v)^+\|_{S^i} + 
h_S^{\frac12} \|\bdelta'_{T,i}(v)^+\|_{\dSi\cup\SG} \lesssim \|\bdelta'_{T,i}(v)\|_{T^i}
\lesssim h_T^{-1} \big\{ \epsilon_{T,i}(v_i) + \delta_{i1}\rev{\epsilon^\Gamma_{T,i}(v)} \big\}.
\]
Invoking the triangle inequality and recalling~\eqref{eq:def_epsilon}
proves~\eqref{eq:bnd_bdelta_S}.
\end{proof}

\subsection{Consistency}

\begin{lemma}[Consistency] \label{lem:bnd.consist}
Let $\rev{u\upex}$ be the weak solution to~\eqref{eq:weak}. 
Assume that $\rev{u\upex}\in H^s(\Omega_1\cup \Omega_2)$, $s>\frac32$.
Set $\Phi_h(\what) := a_h(\hIhk(\rev{u\upex}),\what)-\ell_h(\what)$ for all
$\what \in \Uhatz$. The following holds:
\begin{align} \label{eq:bnd.consist}
|\Phi_h(\what)| \lesssim {}& 
\Bigg\{ \sum_{(T,i)\in\pmeshOK} \kappa_i h_T^{-2} 
\big\{ \epsilon_{T,i}(\rev{u\upex_i}) + \delta_{i1}\rev{\epsilon^\Gamma_{T,i}(u\upex)} \big\}^2 \Bigg\}^{\frac12} \|\what\|_{h0}.
\end{align}
\end{lemma}

\begin{proof}
Owing to~\eqref{eq:identity_consist} from Lemma~\ref{lem:preparatory}, we have
$\Phi_h(\what) = A_1+A_2+A_3$ with
\begin{align*}
A_1 = {}& 
\sum_{(T,i)\in\pmesh} \kappa_i (\bdelta_{T,i}(\rev{u\upex}),\nabla \wTi)_{T^i},  \\
A_2 = {}& \sum_{(T,i)\in\pmeshOK} \kappa_i \bigg\{ (\bdelta_{T,i}(\rev{u\upex}) \SCAL \nT,\wdTi-\wTi)_{\dTi} 
- \delta_{i1}(\bdelta_{T,1}(\rev{u\upex}) \SCAL \nG,\jumpwT_{\Gamma})_{T^{\Gamma}}  \\
& + \sum_{S\in\N_i^{-1}(T)}  \big\{ 
(\bdelta_{T,S,i}(\rev{u\upex}) \SCAL \nS,\wdSi - \wSi)_{\dSi} - \delta_{i1} 
(\bdelta_{T,S,1}(\rev{u\upex}) \SCAL \nG,\jumpwS_{\Gamma})_{\SG}\big\}\bigg\},  \\
A_3 = {}& s_h^\circ(\hIhk(\rev{u\upex}),\what) + \sum_{T\in\cutmesh} \kappa_1h_T^{-1}
(\llbracket I_{T}^{k+1}(\rev{u\upex})-\rev{u\upex}\rrbracket_\Gamma,\jumpwT_\Gamma)_{\TG} + s_h^\N(\hIhk(\rev{u\upex}),\what),
\end{align*}
where $\bdelta_{T,i}(\rev{u\upex})$ and $\bdelta_{T,S,i}(\rev{u\upex})$ are defined in~\eqref{eq:def_bdelta}
for all $(T,i)\in\pmeshOK$, whereas we set $\bdelta_{T,i}(\rev{u\upex}):=\nabla (I^{k+1}_{T^i}(\rev{u\upex_i})-\rev{u\upex_i})$
for all $(T,i)\in\pmeshKO$, and where $\llbracket I_{T}^{k+1}(\rev{u\upex})-\rev{u\upex}\rrbracket_\Gamma$
is defined in Lemma~\ref{lem:approx.IkT}.
Owing to~\eqref{eq:def_IkT_KO}, we infer that the term $A_1$ can be rewritten as follows:
\[
A_1 = \sum_{(T,i)\in\pmeshOK} \kappa_i \bigg\{ (\bdelta_{T,i}(\rev{u\upex}),\nabla \wTi)_{T^i}
+ \sumexi (\bdelta_{T,S,i}(\rev{u\upex}),\nabla \wSi)_{S^i} \bigg\}.
\]
Invoking the Cauchy--Schwarz inequality together with the approximation results from Lemma~\ref{lem:approx.interp} gives
\begin{align*}
|A_1| &\lesssim \Bigg\{ \sum_{(T,i)\in\pmeshOK} \kappa_i h_T^{-2} \big\{
\epsilon_{T,i}(\rev{u\upex_i}) + \delta_{i1}\rev{\epsilon^\Gamma_{T,i}(u\upex)} \big\}^2 
\Bigg\}^{\frac12}
\Bigg\{ \sum_{(T,i)\in\pmesh} \kappa_i\|\nabla \wTi\|_{T^i}^2 \Bigg\}^{\frac12} \\
&\le \Bigg\{ \sum_{(T,i)\in\pmeshOK} \kappa_i h_T^{-2} \big\{
\epsilon_{T,i}(\rev{u\upex_i}) + \delta_{i1}\rev{\epsilon^\Gamma_{T,i}(u\upex)} \big\}^2 
\Bigg\}^{\frac12} \|\what\|_{h0},
\end{align*} 
where the last bound follows from the definition~\eqref{eq:stab.norm} of the 
stability norm $\|\SCAL\|_{h0}$.
The same arguments prove that
\[
|A_2| \lesssim \Bigg\{ \sum_{(T,i)\in\pmeshOK} \kappa_i h_T^{-2} 
\big\{ \epsilon_{T,i}(\rev{u\upex_i}) + \delta_{i1}\rev{\epsilon^\Gamma_{T,i}(u\upex)} \big\}^2 \Bigg\}^{\frac12} \|\what\|_{h0}.
\]
Finally, let us write $A_3=A_{31}+A_{32}+A_{33}$ with obvious notation. First, observing
that $s_h^\circ(\what,\what) \le |\what|_{\textsc{s}}^2$, we obtain
\begin{align*}
|A_{31}| &\le \Bigg\{ \sum_{(T,i)\in\pmesh} \kappa_i h_T^{-1} \|\projT(\rev{u\upex_i}-I_{T^i}^{k+1}(\rev{u\upex_i}))\|_{\dTi}^2 \Bigg\}^{\frac12} |\what|_{\textsc{s}} \\
&\le \Bigg\{ \sum_{(T,i)\in\pmesh} \kappa_i h_T^{-1} \|\rev{u\upex_i}-I_{T^i}^{k+1}(\rev{u\upex_i})\|_{\dTi}^2 \Bigg\}^{\frac12} |\what|_{\textsc{s}},
\end{align*}
where the second bound follows from the $L^2$-stability of $\projT$. Owing to~\eqref{eq:def_IkT_KO}, the summation on the right-hand side can be rewritten as
\[
|A_{31}| \le \Bigg\{ \sum_{(T,i)\in\pmeshOK} \sum_{S\in\{T\}\cup \N_i^{-1}(T)} \kappa_i h_S^{-1} \|\rev{u\upex_i}-I_{T^i}^{k+1}(\rev{u\upex_i})^+\|_{\dSi}^2 \Bigg\}^{\frac12} |\what|_{\textsc{s}}.
\]
Therefore, we conclude that
\[
|A_{31}| \le \Bigg\{ \sum_{(T,i)\in\pmeshOK} \kappa_i h_T^{-2} 
\epsilon_{T,i}(\rev{u\upex_i})^2 \Bigg\}^{\frac12} |\what|_{\textsc{s}}.
\]
Concerning $A_{32}$, the Cauchy--Schwarz inequality and the same arguments as above give
\[
|A_{32}| \le \Bigg\{ \sum_{(T,i)\in\pmeshOK} \kappa_i h_T^{-2} 
\big\{ \epsilon_{T,i}(\rev{u\upex_i}) + \delta_{i1}\rev{\epsilon^\Gamma_{T,i}(u\upex)} \big\}^2\Bigg\}^{\frac12} s_h^\Gamma(\what,\what)^{\frac12}.
\]
Finally, $s_h^\N(\hIhk(\rev{u\upex}),\what)=0$ owing to~\eqref{eq:def_IkT_KO}. Putting everything together proves the claim.
\end{proof}

\subsection{Error estimate}

We are now ready to establish the main result of our error analysis.

\begin{theorem}[Error estimate] \label{thm:error}
Let $\rev{u\upex}$ be the weak solution to~\eqref{eq:weak}. 
Assume that $\rev{u\upex}\in H^s(\Omega_1\cup \Omega_2)$ with $s\in(\frac32,k+2]$.
The following holds:
\begin{equation} \label{eq:err.est}
\Bigg\{ \sum_{(T,i)\in\pmesh} \kappa_i\|\nabla(\rev{u\upex_i}-u_{T^i})\|_{T^i}^2 \Bigg\}^{\frac12}
\lesssim \Bigg\{ \sum_{(T,i)\in\pmeshOK}
\kappa_i h_T^{2(s-1)} |E_i^s(\rev{u\upex_i})|_{H^s(\Delta_2(T))}^2 \Bigg\}^{\frac12}
\lesssim h^{s-1} \sum_{i\in\{1,2\}} \kappa_i^{\frac12} |\rev{u\upex_i}|_{H^s(\Omega_i)}.
\end{equation}
\end{theorem}

\begin{proof}
Define the discrete error as $\hat{e}_h := \hIhk(\rev{u\upex})-\uhat \in \Uhatz$.
Since $a_h(\hat{e}_h,\hat{e}_h) = \Phi_h(\hat{e}_h)$, invoking the stability result from Lemma~\ref{stabilitylemma} and the bound on the consistency error from Lemma~\ref{lem:bnd.consist} gives
\[
\|\hat{e}_h\|_{h0}^2 \lesssim \sum_{(T,i)\in\pmeshOK} \kappa_i h_T^{-2} 
\big\{ \epsilon_{T,i}(\rev{u\upex_i}) + \delta_{i1}\rev{\epsilon^\Gamma_{T,i}(u\upex)} \big\}^2.
\] 
\rev{Owing to the approximation result from Lemma~\ref{lem:approx.IkT} 
and since $\kappa_1\le \kappa_2$ and $\Delta(T)\subset \Delta_2(T)$, we infer that
\[
\|\hat{e}_h\|_{h0}^2 \lesssim \sum_{(T,i)\in\pmeshOK} \kappa_i h_T^{2(s-1)} |E_i^s(\rev{u\upex_i})|_{H^s(\Delta_2(T))}^2.
\]
Invoking the triangle inequality then} establishes the first bound in~\eqref{eq:err.est}.
The second bound follows from the shape-regularity of the mesh, and the $H^s$-stability of the extension operators $E_i^s$. 
\end{proof}

\section{Numerical results}
\label{sec:res}

In this section, we present numerical results to illustrate the convergence rates established in Theorem~\ref{thm:error}. We also compare the present method to the one from \cite{BCDE21} stabilized by a cell-agglomeration procedure, and we briefly investigate some aspects related to the implementation (quadrature, local polynomial bases, pairing criterion). 

\begin{figure}[htb!]
\centering
\includegraphics[width=0.25\textwidth]{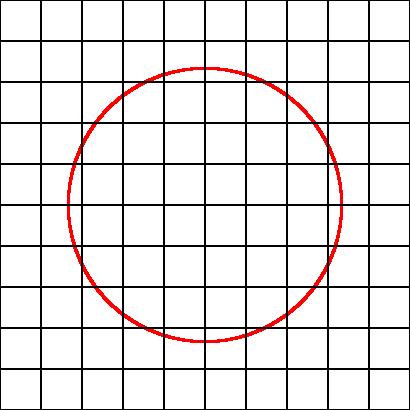}
\includegraphics[width=0.25\textwidth]{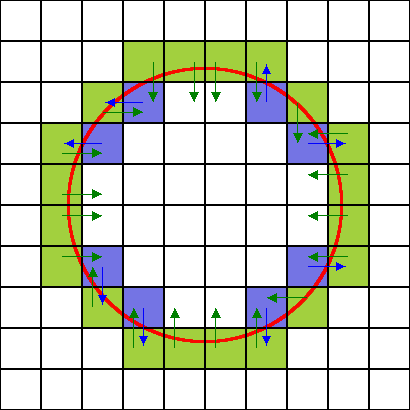}
\includegraphics[width=0.25\textwidth]{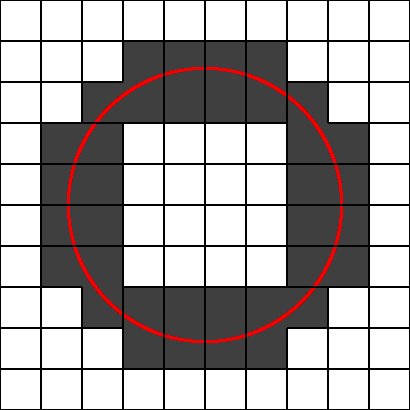}
\\[0.125cm]
\includegraphics[width=0.25\textwidth]{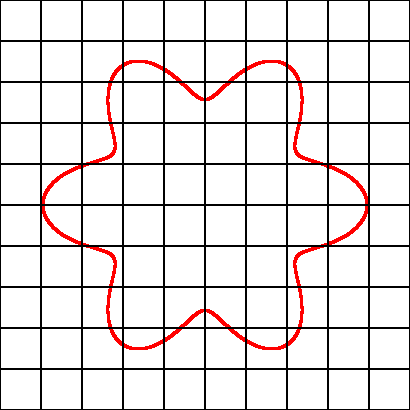}
\includegraphics[width=0.25\textwidth]{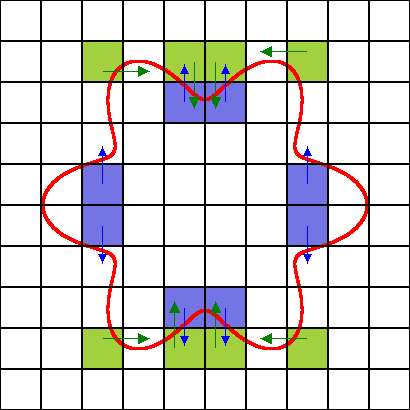}
\includegraphics[width=0.25\textwidth]{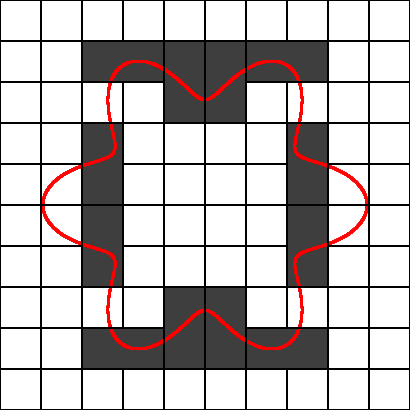}\\[0.25cm]
\caption{Circular (top row) and flower-like (bottom row) interfaces.
Left column: Coarsest mesh. Central column: Outcome of pairing procedure for polynomial extension stabilization, cells in $\TKOonemesh$ are colored in green, and cells in $\TKOtwomesh$ are colored in blue, arrows indicate the pairing operator. Right column: Outcome of pairing procedure for cell-agglomeration stabilization, agglomerated cells are colored in dark.}
\label{fig:meshes}
\end{figure} 

In all cases,  we consider the unit square domain $\Omega := (0,1)^2$ discretized with uniform Cartesian meshes of size $h := \sqrt{2} \times 0.1 \times 2^{-\ell}$ with $\ell \in \{0,\ldots,4\}$. The polynomial degree is taken such that $k\in\{0,\ldots,3\}$. 
The interface is defined  using a level-set function $\Phi$, with $\Gamma := \{(x, y) \in \Omega \mid \Phi(x, y) = 0\}$ and $\Omega_i = \{(x, y) \in \Omega \mid (-1)^i\Phi(x, y) > 0\}$ for all $i\in\{1,2\}$. We always take $\kappa_1:=1$, and modify the diffusivity contrast by taking $\kappa_2 = 10^{m}$, $m\in\{0,\ldots,4\}$, where $m=0$ corresponds to no contrast and $m=4$ to a highly contrasted setting.
We consider two shapes for the interface: a circular shape and a flower-like shape, defined by the following level-set functions:  
\begin{subequations}
\begin{align}
\Phi_{\textsc{c}}(x, y) &:= (x-a)^2 + (y-b)^2 - R^2, \\
\Phi_{\textsc{f}}(x, y) &:= (x-a)^2 + (y-b)^2 - R^2 + c \cos(n \theta),
\end{align}
\end{subequations}
with $\theta := \arctan \left(\frac{y-b}{x-a}\right)$ if $x \geq a$, and $\theta := \pi + \arctan \left(\frac{y-b}{x-a}\right)$ if $x < a$, $a:=b:=0.5$, $R = \frac{1}{3}$, $c:=0.03$, and $n=8$. 
Figure~\ref{fig:meshes} illustrates on the coarsest mesh ($\ell=0$) the two interfaces (left column), and the outcome of the pairing procedure for polynomial extension stabilization (central column) and for cell-agglomeration stabilization (right column).

\subsection{Implementation aspects}

Quadratures along the interface and in the cut sub-cells are realized by dividing the interface into $2^r$ segments, $r\in\mathbb{N}$, and creating a sub-triangulation of the two cut sub-cells. The construction of the sub-triangulation is performed as discussed in \cite{BCDE21}; see, in particular, Figure 4.1 therein for an illustration. Unless explicitly stated otherwise, we set $r:=8$. \rev{We also take $\eta_{\N}:=20$ for the stabilization bilinear form $s_h^{\N}$ in~\eqref{eq:def_shN}. Concerning the flagging of ill-cut cells, we use the parameter $\vartheta:=0.3$. Our numerical experiments below show a marginal impact of the choice of this parameter on stability and accuracy. As this numerical observation is not backed up by theoretical results, we recommend to use a flagging parameter $\vartheta$ as indicated above.} 

Another important aspect is the choice of the polynomial bases in the sub-cells and the sub-faces. Both bases are realized by considering centered and scaled monomials. In particular, the barycenter of each sub-cell is used for the centering of the corresponding basis. The barycenter is computed using the above sub-triangulation. The scaling of the monomials is isotropic and uses half of the diameter of the original cell. \rev{In the ill-cut cells, say $S_i$ with $i\in\{1,2\}$, the centering and scaling is defined by considering the merged cell $S_i\cup \N_i(S_i)$.}

\subsection{Convergence rates}

\begin{figure}[!htb]
\centering
\includegraphics[width=0.425\textwidth]{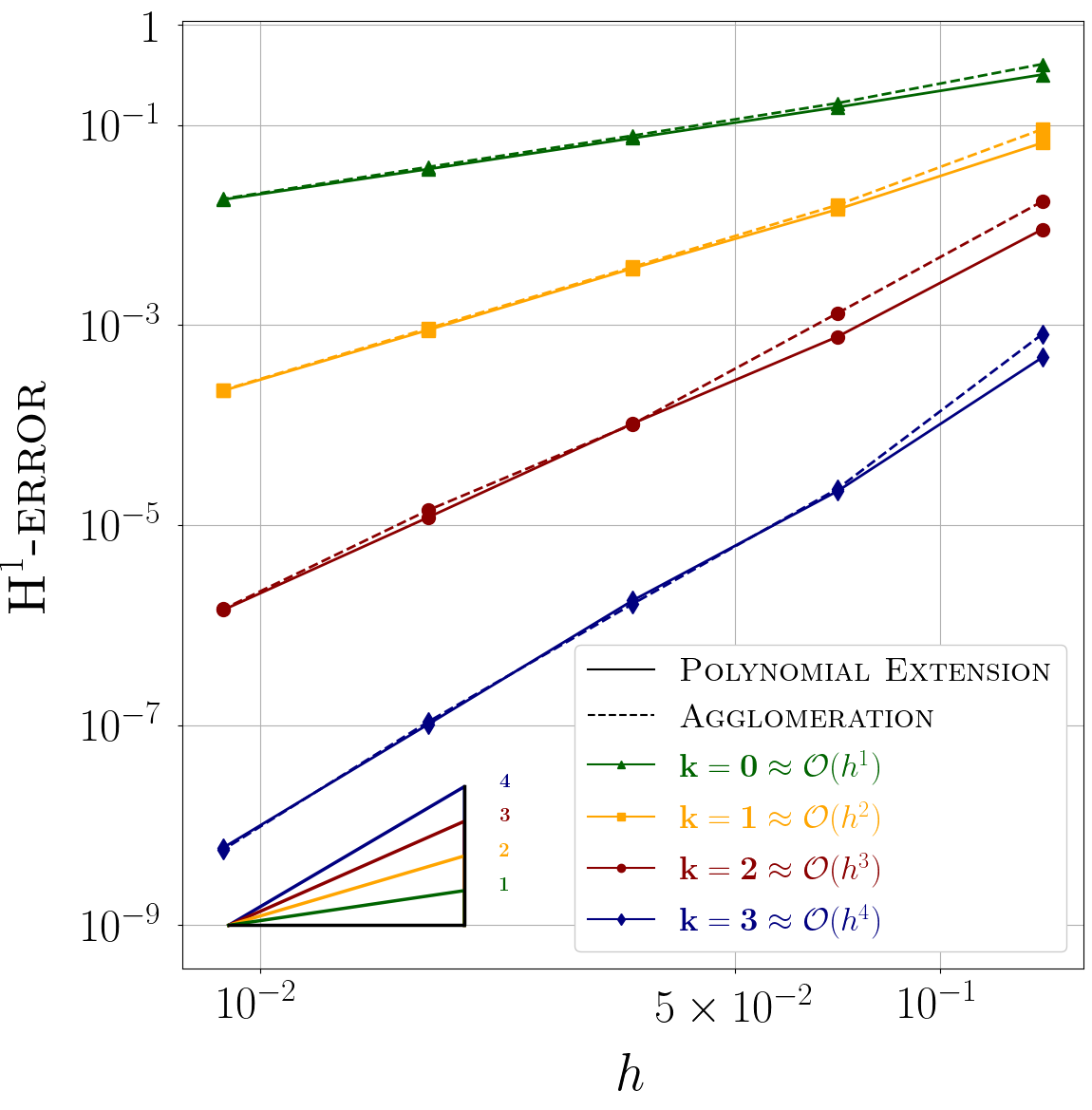}
\includegraphics[width=0.425\textwidth]{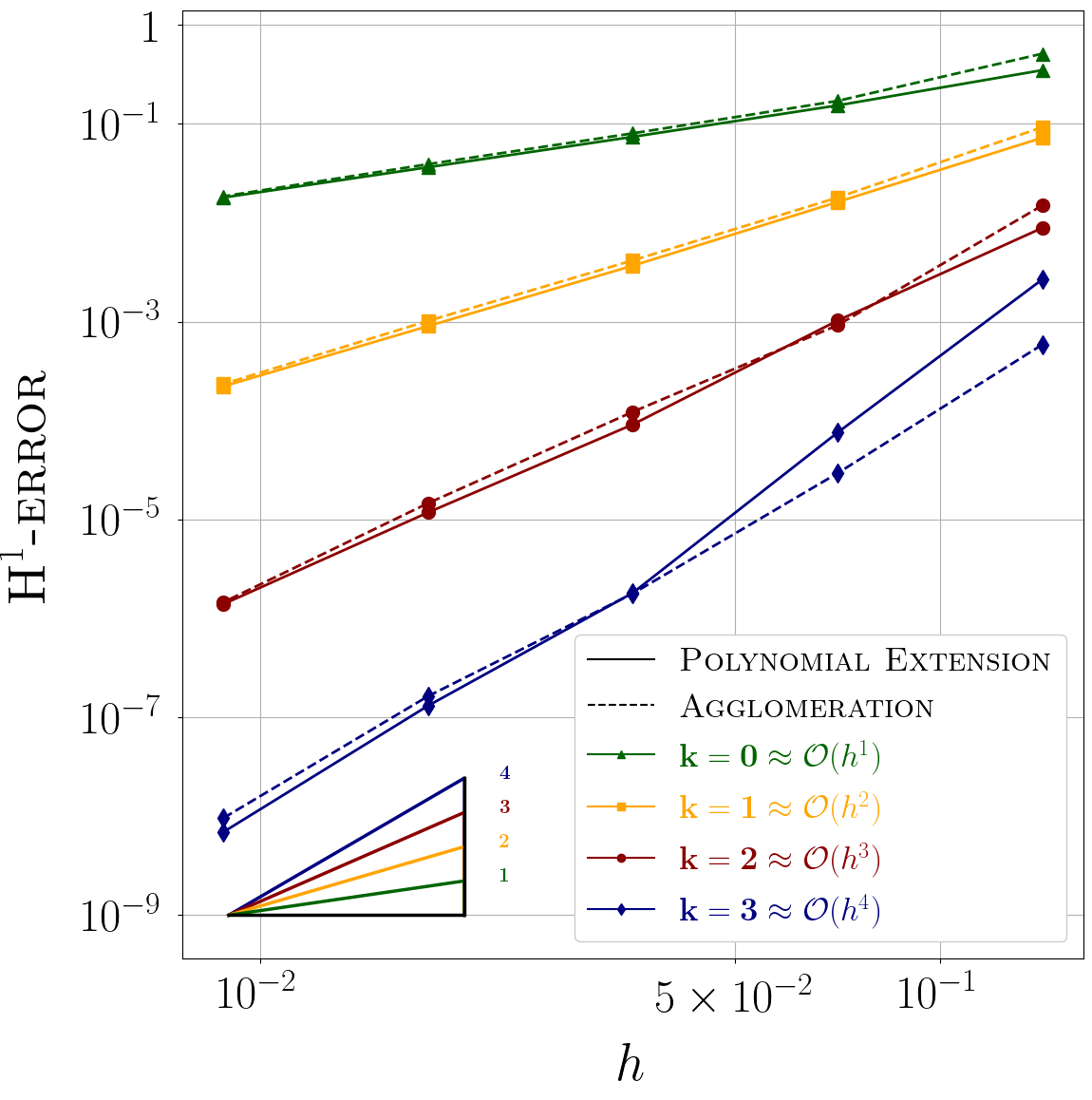} 
\caption{Errors as a function of the mesh-size for the exact solution \eqref{sol_sinsin}. 
Comparison between stabilization by polynomial extension (solid lines) and
by cell agglomeration (dashed lines).
Left: Circular interface. Right: Flower-like interface. 
}
\label{fig:err_ext_agglo}
\end{figure}  

We consider first a test case with no contrast and no jumps ($g_D=g_N=0$). The exact solution
is taken to be (in Cartesian coordinates $(x,y)$)
\begin{equation}
\rev{u\upex}(x,y) := \sin(\pi x)\sin(\pi y).
\label{sol_sinsin}
\end{equation} 
Errors are reported in Figure~\ref{fig:err_ext_agglo} for the circular interface (left panel) and the flower-like interface (right panel). Here and below, errors are evaluated using the energy norm considered in Theorem~\ref{thm:error}. In both panels, we compare the results obtained using stabilization by polynomial extension (solid lines) and by cell agglomeration (dashed lines). Both approaches yield similar results, with slightly better errors when using polynomial extension, especially on coarse meshes, except for $k=3$ and the flower-like interface. 

\begin{figure}[!htb]
\centering
\includegraphics[width=0.425\textwidth]{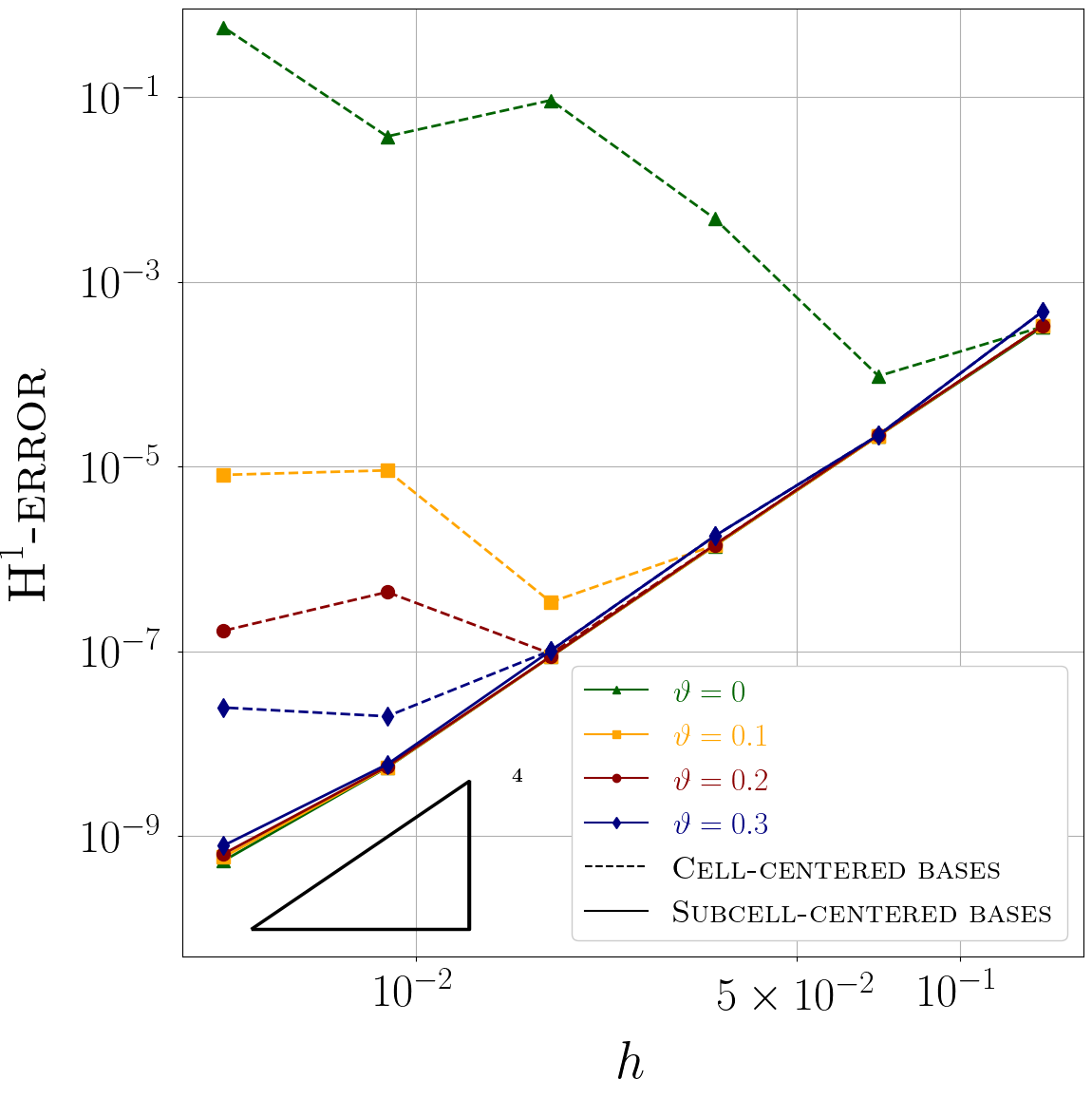}
\includegraphics[width=0.425\textwidth]{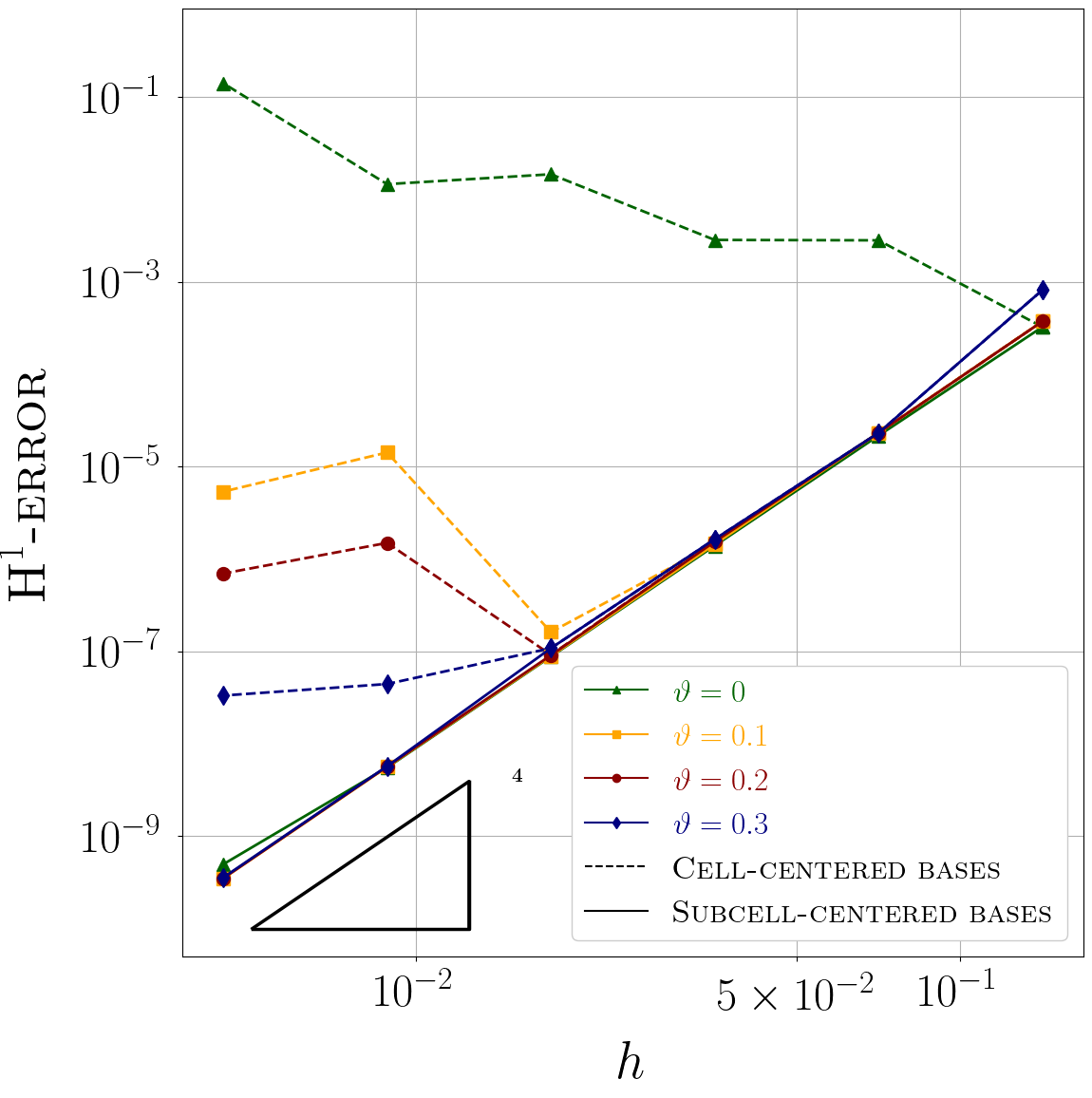} 
\caption{Errors as a function of the mesh size for the exact solution \eqref{sol_sinsin} for $k=3$ and various values of the pairing parameter $\vartheta$ used for flagging ill-cut cells. Left: polynomial extension. Right: cell agglomeration.}
\label{fig:pairing_criterion}
\end{figure} 

We next study the sensitivity of the errors with respect to the value of the pairing parameter $\vartheta$ used for flagging ill-cut cells. We consider again the exact solution~\eqref{sol_sinsin}. Errors are reported in Figure~\ref{fig:pairing_criterion} for $k=3$ and $\vartheta\in \{0,0.1,0.2,0.3\}$ ($\vartheta=0$ corresponds to no pairing, irrespective of the ill cut). The left (resp., right) panel corresponds to polynomial extension (resp., cell agglomeration). Solid lines correspond to polynomial bases centered at the barycenter of each sub-cell (the default choice in our implementation), whereas dashed lines correspond to centering at the barycenter of the original cell for both sides of the cut. We observe that this second choice yields poor results owing to poor conditioning of the local matrices, whereas the first choice is robust to the presence of ill cuts. Somewhat surprisingly, the robustness is such that it allows to flag cells with a very loose criterion ($\vartheta=0.1$) or even to flag no cells at all ($\vartheta=0$), and still obtain close-to-optimal errors. 

\begin{figure}[!htb]
\centering
\includegraphics[width=0.425\textwidth]{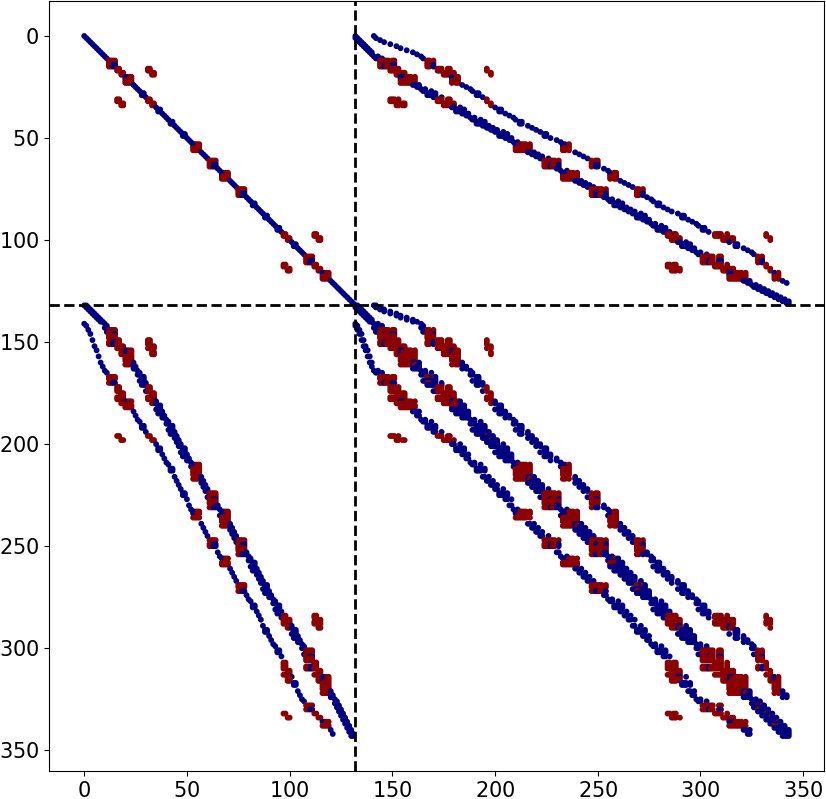}
\includegraphics[width=0.425\textwidth]{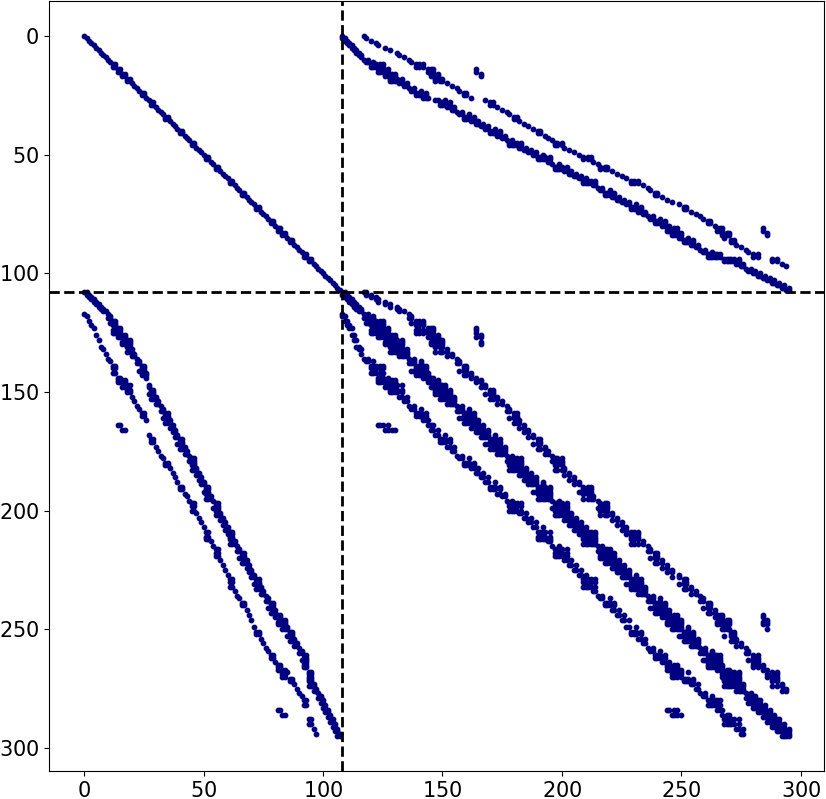} 
\caption{Sparsity profiles of the stiffness matrices obtained using polynomial extension (left) and cell agglomeration (right). In the left panel, the red dots indicate the coupling blocks resulting from polynomial extension. In both panels, the global block-decomposition into cell and face unknowns is indicated by black dashed lines. Coarsest mesh ($\ell=0$).}
\label{fig:sparsity}
\end{figure}  

A further comparison between polynomial extension and cell agglomeration is provided in Figure~\ref{fig:sparsity} which compares the sparsity profiles of the stiffness matrices obtained with both procedures \rev{using the lowest-order polynomial setting}. We observe that both approaches lead to a similar sparsity profile, and that the stiffness matrix corresponding to cell agglomeration is slightly smaller (as expected). The sparsity pattern of the matrix corresponding to cell agglomeration directly results from the data structure of the agglomerated mesh, whereas the sparsity pattern of the matrix corresponding to polynomial extension is derived at the assembly stage by means of the pairing operator.

\begin{figure}[!htb]
\centering
\includegraphics[width=0.425\textwidth]{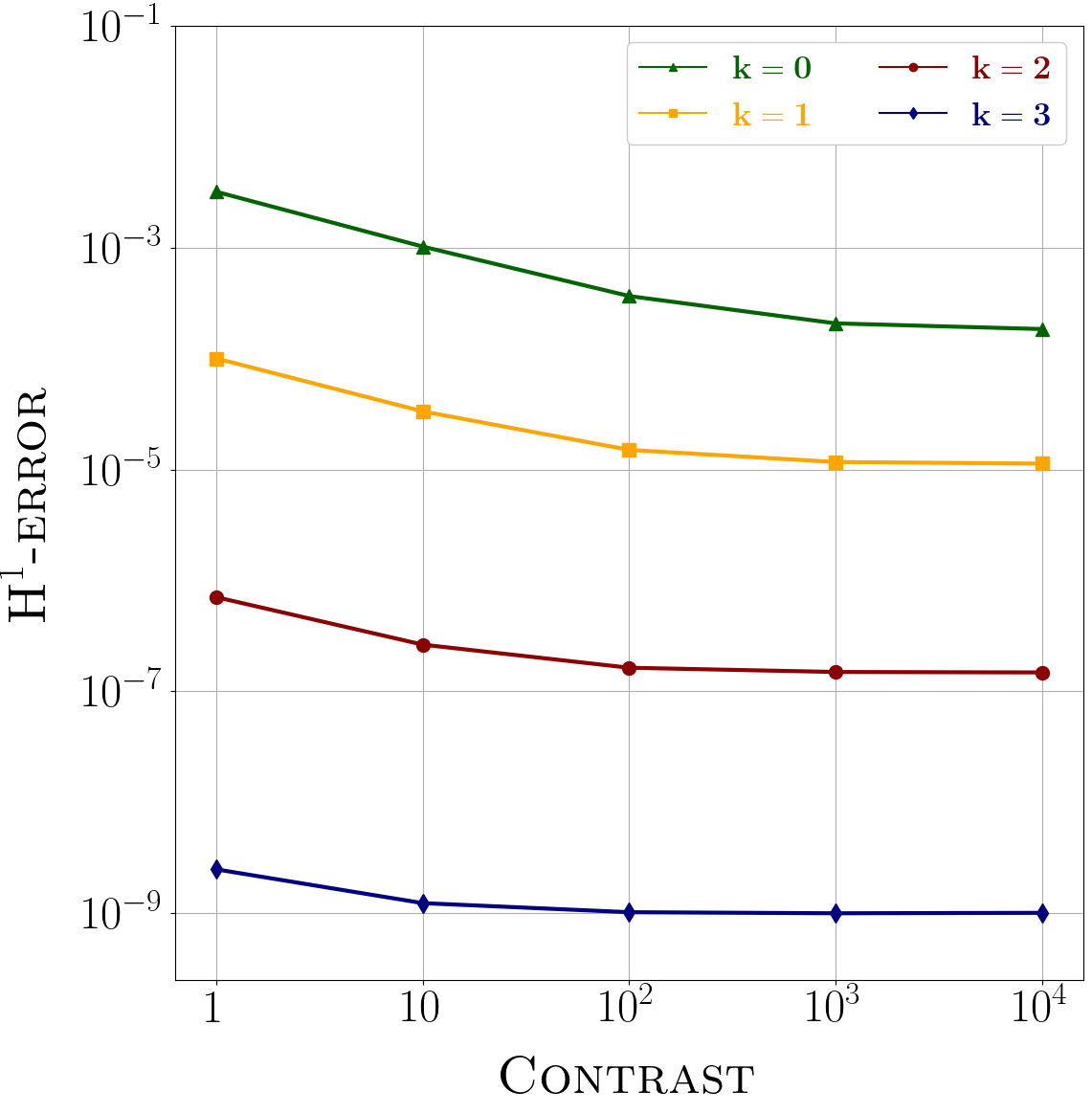}
\includegraphics[width=0.425\textwidth]{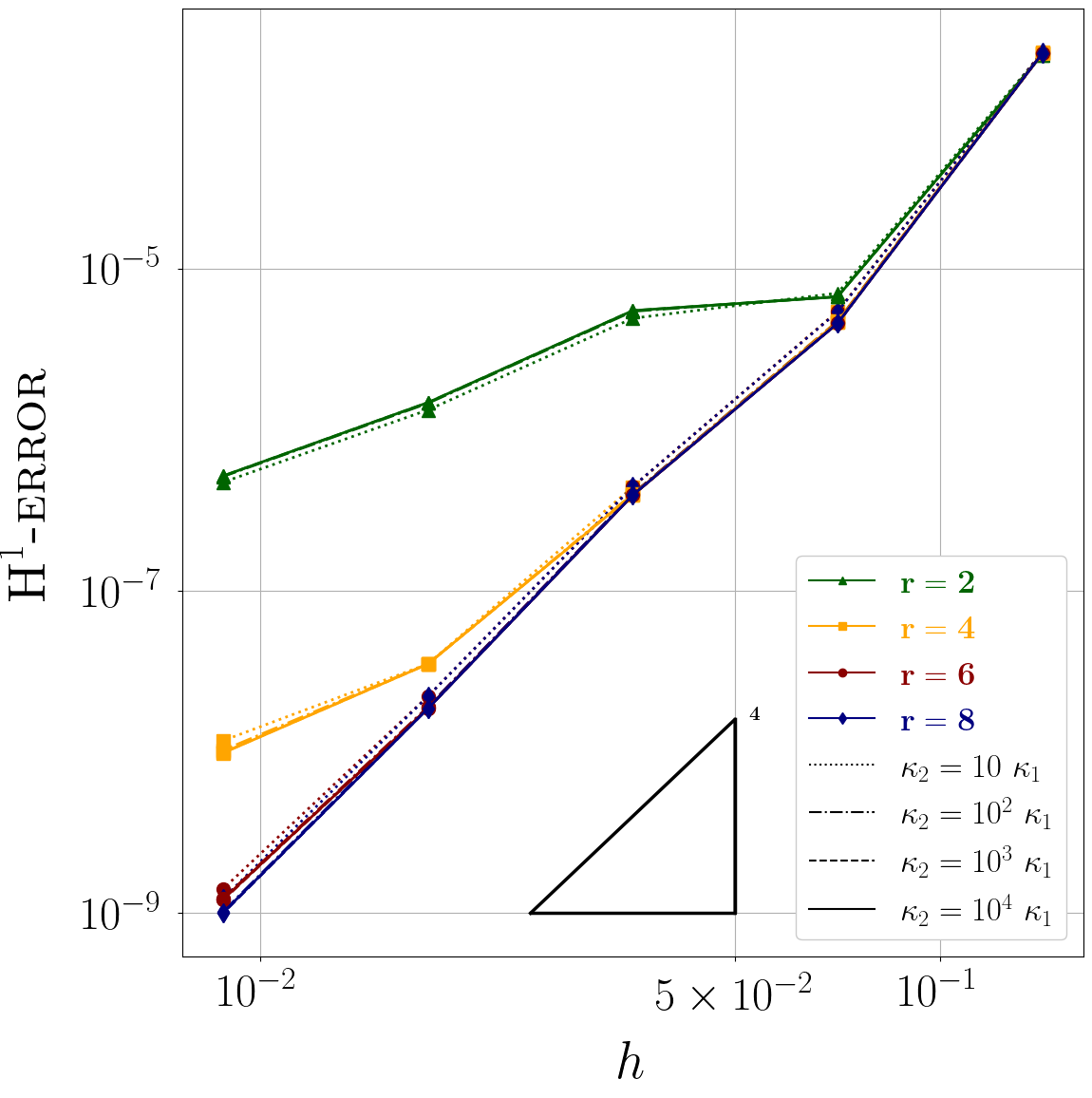} 
\caption{Left: Errors for the exact solution~\eqref{sol_circle_contrast} as a function of the diffusivity contrast on the finest mesh ($\ell = 4$) and for various polynomial degrees ($k\in\{0,\ldots,3\}$). Right: Errors as a function of the mesh size for various values of the sub-triangulation parameter, $k=3$, $r\in\{2,4,6,8\}$, and various diffusivity contrasts, $m\in\{0,\ldots,4\}$.}
\label{fig:contrast}
\end{figure}  

We now turn our attention to a test case with diffusivity contrast and jumps across the 
interface. We first consider the exact solution for the circular interface such that
(in radial coordinates $(\rho,\theta)$)
\begin{equation}\label{sol_circle_contrast}
\rev{u\upex_1}(\rho) := \frac{\rho^6}{\kappa_1},\qquad
\rev{u\upex_2}(\rho) := \frac{\rho^6}{\kappa_2}+R^6\bigg(\frac{1}{\kappa_1}-\frac{1}{\kappa_2}\bigg).
\end{equation}
Notice that this solution does not exhibit jumps across the interface ($g_D=g_N=0$). 
The errors in the highly contrasted case where $\kappa_2 = 10^4 \kappa_1$ exhibit
the same optimal convergence rates as those reported in Figure~\ref{fig:err_ext_agglo}
(not shown for brevity). To highlight the robustness of the approach with respect to
the diffusivity contrast, we report in the left panel of Figure~\ref{fig:contrast}
the errors as a function of $\kappa_2=10^m\kappa_1$, $m\in\{0,\ldots,4\}$,
obtained on the finest mesh ($\ell=4$) and for various polynomial degrees 
($k\in\{0,\ldots,3\}$). The right panel of Figure~\ref{fig:contrast} underlines the
importance of ensuring sufficient geometric resolution when performing quadratures
in the cut sub-cells. Therein, we report the error as a function of the mesh size
for the sub-triangulation parameter $r\in\{2,4,6,8\}$ and various diffusivity 
contrasts. We observe that in all cases, taking $r\ge 6$ is required to avoid that
the geometric errors pollute the optimal decay of the discretization errors.

\begin{figure}[!htb]
\centering
\includegraphics[width=0.425\textwidth]{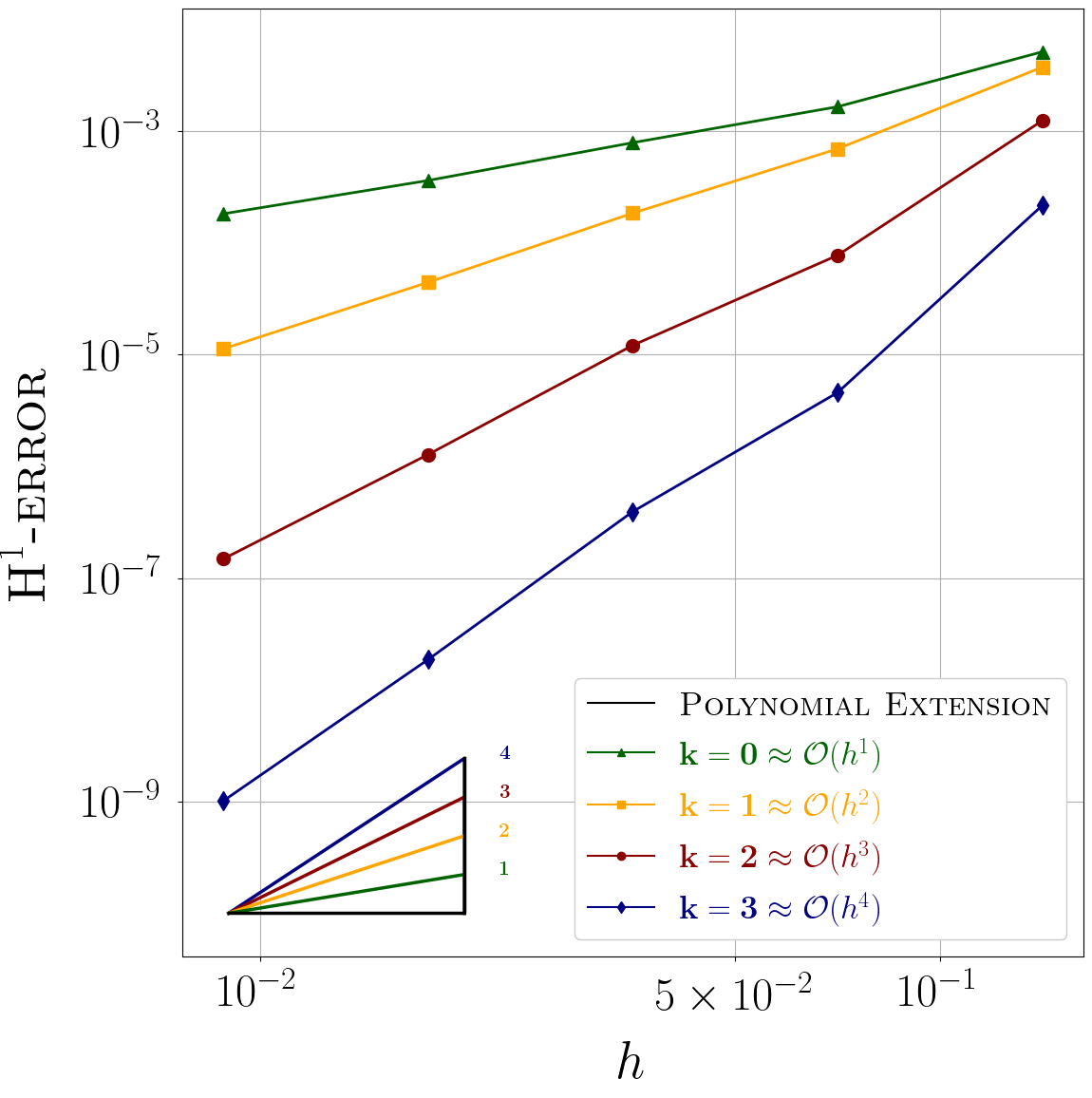} 
\includegraphics[width=0.425\textwidth]{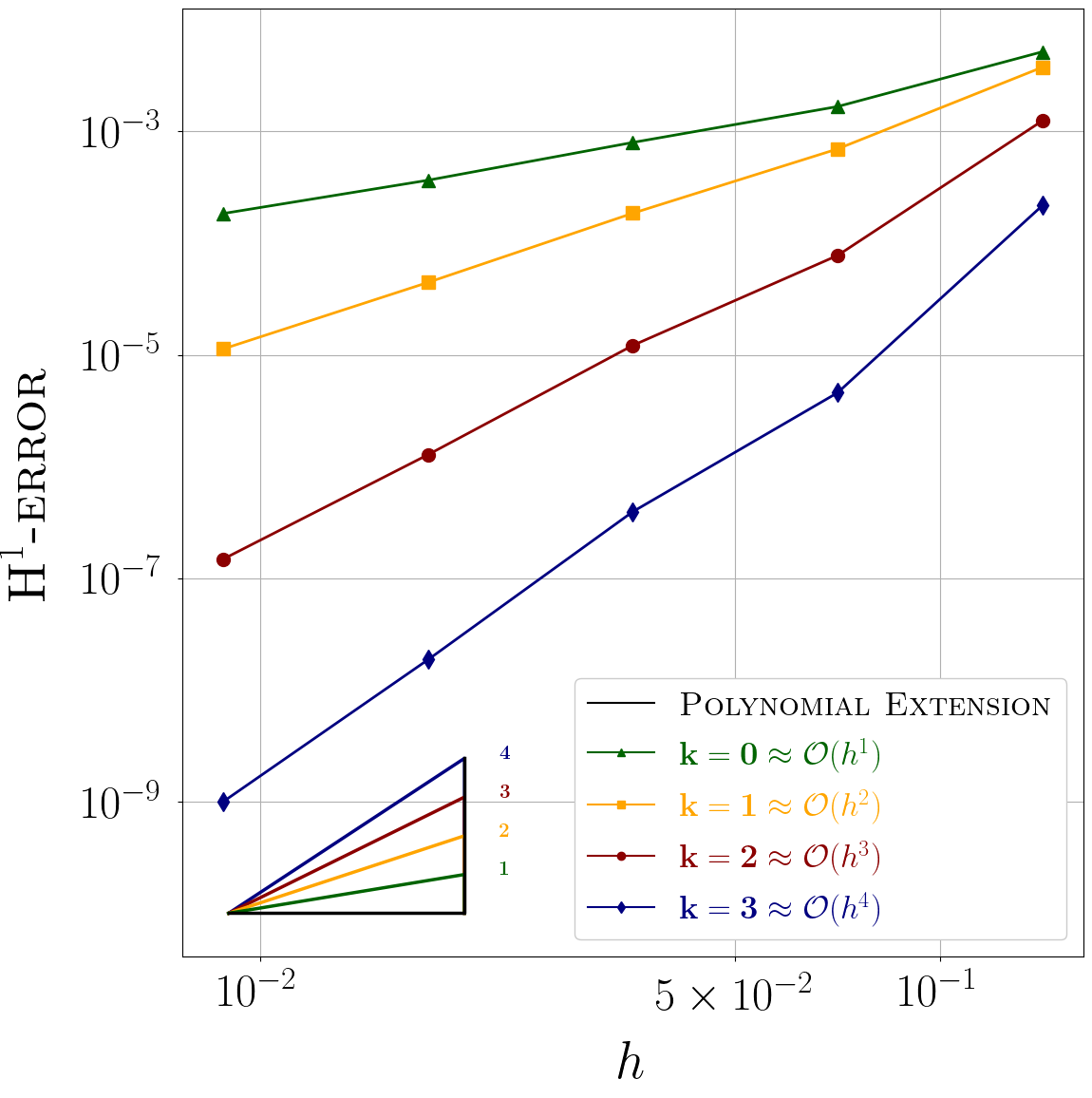} 
\caption{Errors as a function of the mesh size for the exact solution \eqref{sol_gN} (left) and the exact solution \eqref{sol_gD} (right) for various polynomial degrees 
$k\in\{0,\ldots,3\}$; diffusivity contrast set to $\kappa_2=10^4\kappa_1$.} 
\label{fig:jumps_easy}
\end{figure}  

We now include jumps across the interface while considering the highly contrasted setting
where $\kappa_2=10^4\kappa_1$. For this purpose, we modify the exact solution in~\eqref{sol_circle_contrast} and set 
\begin{subequations} 
\begin{alignat}{2} 
\rev{u\upex_1}(\rho) & := \frac{\rho^6}{\kappa_1}, &\qquad 
\rev{u\upex_2}(\rho) & := \frac{\rho^8-R^8}{\kappa_2} + \frac{R^6}{\kappa_1}, \label{sol_gN} \\
\rev{u\upex_1}(\rho) & := \frac{\rho^6}{\kappa_1}, &\qquad
\rev{u\upex_2}(\rho) & := \frac{\rho^6}{\kappa_2}. \label{sol_gD}
\end{alignat}
\end{subequations}
Notice that, for the exact solution~\eqref{sol_gN}, we have $g_N=2R^5(3-4R^2)$ and $g_D=0$, whereas, 
for the exact solution~\eqref{sol_gD}, we have $g_N=0$ and $g_D=R^6(\frac{1}{\kappa_1}-\frac{1}{\kappa_2})$. Errors as a function of the mesh size are reported in Figure~\ref{fig:jumps_easy} for polynomial degrees $k\in\{0,\ldots,3\}$. Optimal convergence rates are observed in all cases. Actually, the errors for both exact solutions are very similar. 

\begin{figure}[!htb]
\centering
\includegraphics[width=0.425\textwidth]{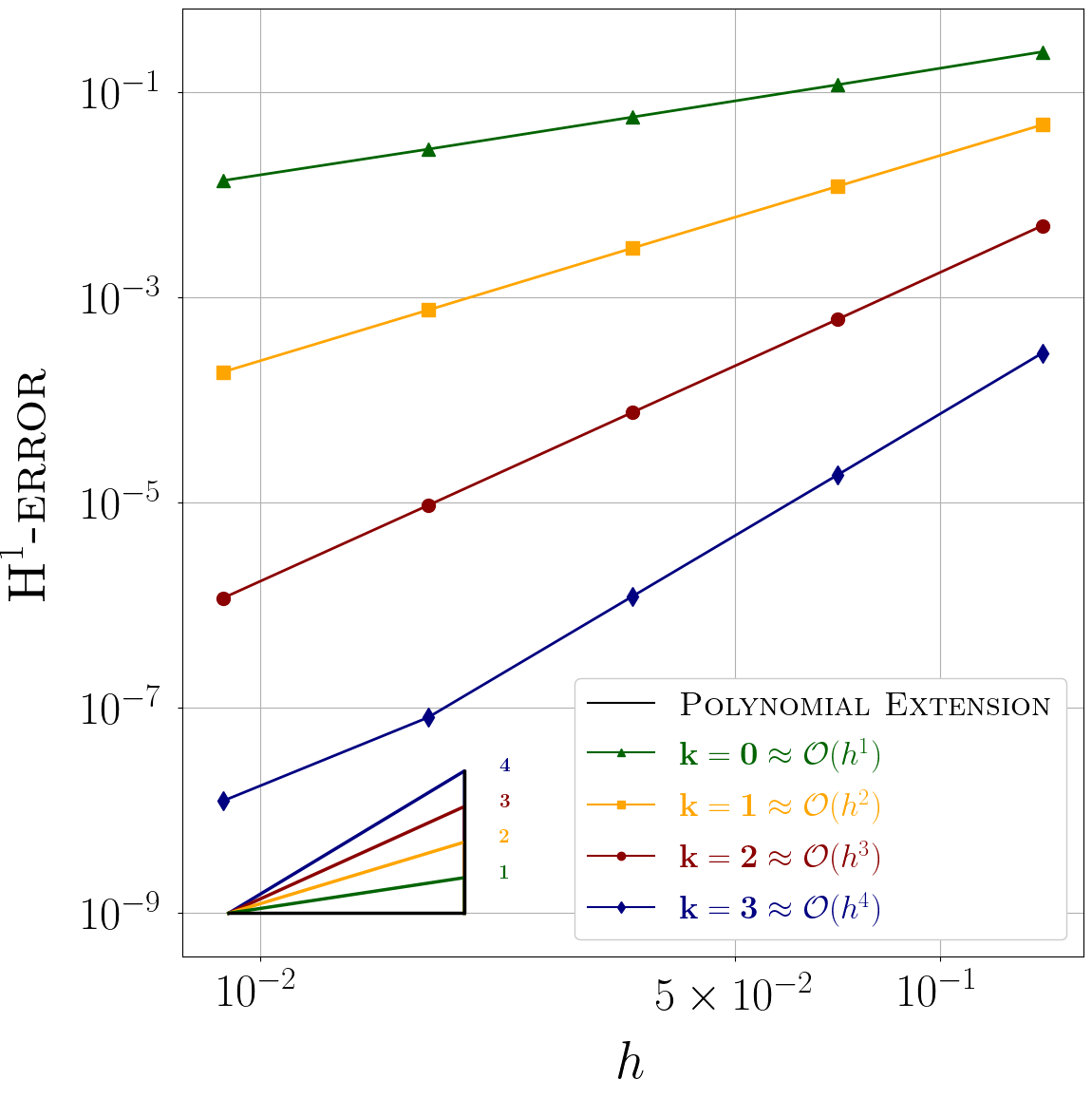} 
\includegraphics[width=0.425\textwidth]{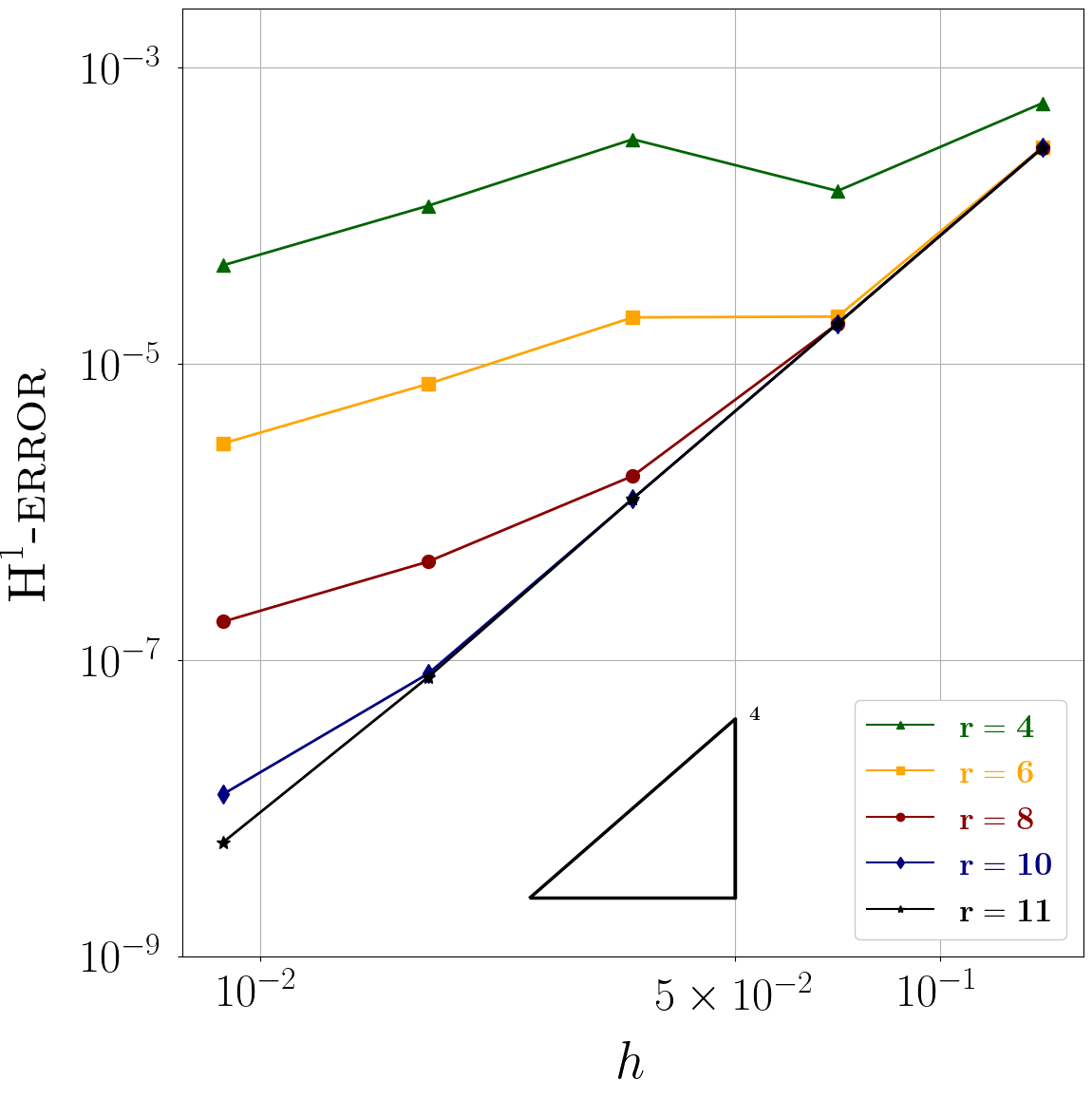} 
\caption{Left: Errors as a function of the mesh size for the exact solution \eqref{sol_gD_gN} for various polynomial degrees $k\in\{0,\ldots,3\}$ (with sub-triangulation parameter set to $r=10$). Right: Errors as a function of the mesh size for $k=3$ and $r \in \{4, 6, 8, 10,11\}$. No diffusivity contrast.}
\label{fig:jumps_hard}
\end{figure} 

A more challenging setting is obtained when considering for the exact solution 
\begin{equation} \label{sol_gD_gN}
\rev{u\upex_1}(x,y) := \cos(y)e^x, \qquad
\rev{u\upex_2}(x,y) := \sin(\pi x)\sin(\pi y),
\end{equation}
leading to variable jump data $g_D$ and $g_N$. Here, we do not consider a diffusivity contrast
($\kappa_1=\kappa_2=1$). Errors as a function of the mesh size are reported in the left panel of Figure~\ref{fig:jumps_hard} for various polynomial degrees $k\in\{0,\ldots,3\}$ and sub-triangulation parameter set to $r=10$. Optimal convergence rates are observed in all cases. To motivate the choice for $r$, we report in the right panel of Figure~\ref{fig:jumps_hard} the errors as a function of the mesh size for $k=3$ and $r \in \{4, 6, 8, 10,11\}$. We observe that the geometric error does not pollute the discretization error only for $r=10$\rev{, except for $k=3$ on the finest mesh where $r=11$ is actually necessary to suppress the geometric error} (notice in passing the very low values attained by the error). 

\begin{figure}[!htb]
\centering
\includegraphics[width=0.425\textwidth]{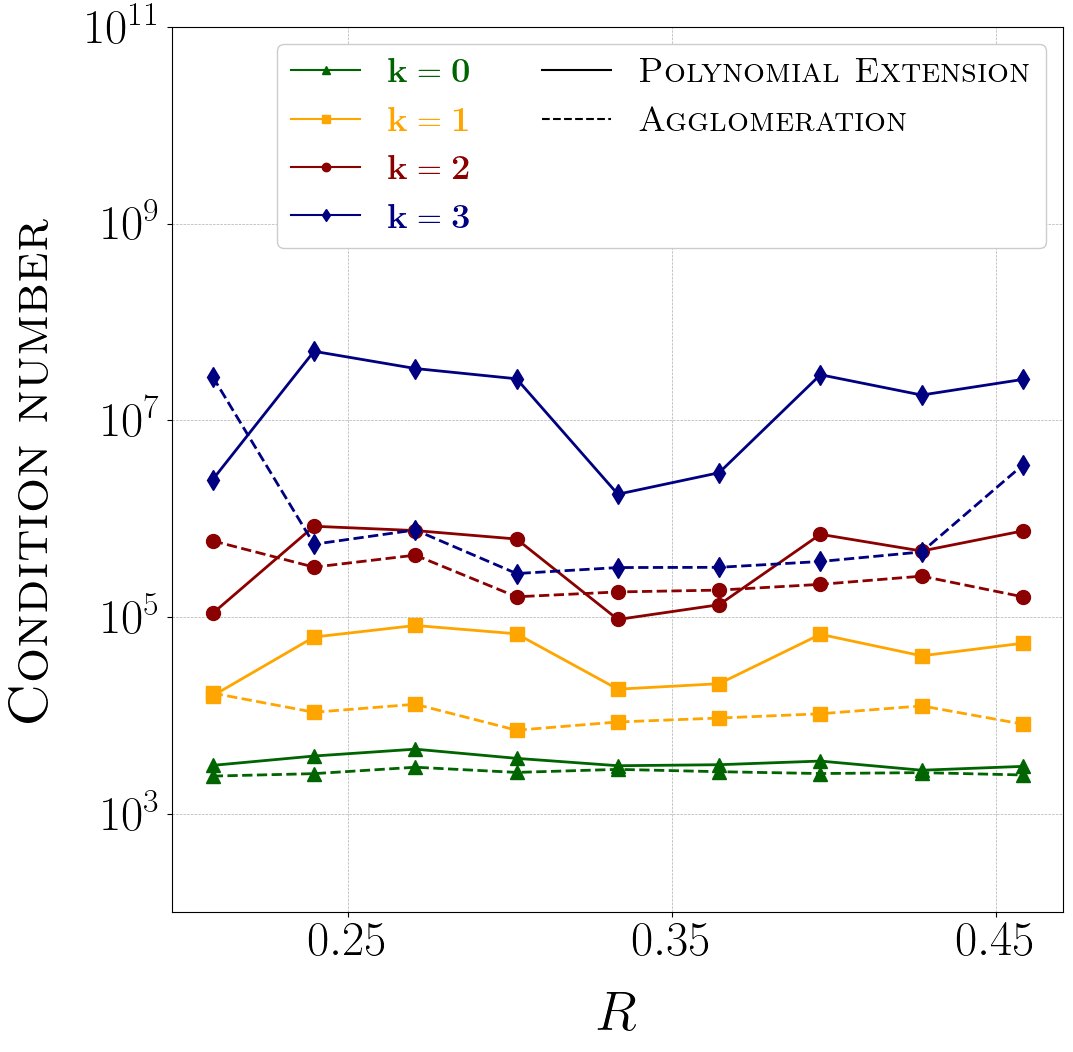} 
\includegraphics[width=0.425\textwidth]{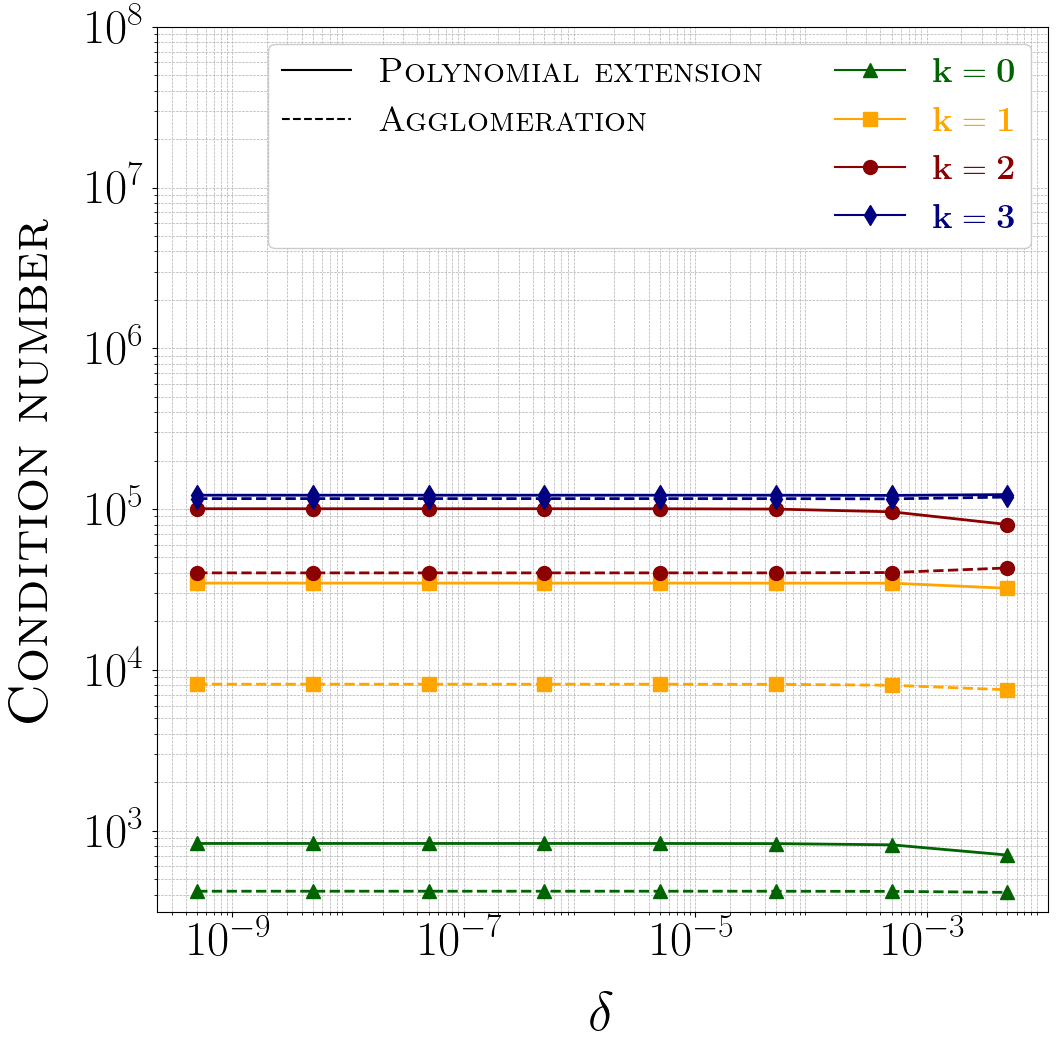} 
\caption{Euclidean condition number of the stiffness matrix on the coarsest mesh ($\ell=0$) as a function of the radius of the circular interface (left) and as a function of the position parameter $\delta$ for the square interface, 
using polynomial extension (solid lines) and cell agglomeration (dashed lines).}
\label{fig:conditioning}
\end{figure} 

Finally, we perform a brief study of the conditioning of the stiffness matrix. 
In the left panel of Figure~\ref{fig:conditioning}, we consider a circular interface
with radius $R:=\frac13 + \frac{i}{32}$, $i\in\{-4,\ldots,4\}$, and we report the Euclidean
condition number as a function of the radius, using either polynomial extension (solid lines) 
or cell agglomeration (dashed lines). We consider the coarsest mesh ($\ell=0$) and
polynomial degrees $k\in\{0,\ldots,3\}$. We observe that, for each polynomial degree, the
conditioning remains fairly insensitive to the stabilization method of the ill-cut cells.
To exacerbate the effect of ill-cut cells, we now consider a square interface with level-set function 
$\Phi_{\textsc{s}}(x,y) := \max(x-a,y-b)-(0.25+\delta)$ and position parameter $\delta:=0.5\times 10^{-p}$, $p\in\{2,\ldots,9\}$. 
The right panel of Figure~\ref{fig:conditioning} reports the 
Euclidean condition number as a function of the position parameter of the square interface. 
We observe \rev{full robustness of the condition number with respect to the cut position, for both polynomial extension and cell agglomeration.
Further numerical results (not displayed for brevity) indicate that a variant of the stabilization bilinear form $s_h^{\N}$ where the penalty is enforced on the faces of $T$ with a weight $h_T^{-1}$ does not lead to a robust behavior, but instead to a linear growth in $\delta^{-1}$ for $k\in\{2,3\}$.}
We also emphasize that the errors still behave optimally (figure 
omitted for brevity), in agreement with the above convergence analysis.

\section{Proofs}
\label{sec:proofs}

This section collects the proofs of the two results stated in Section~\ref{sec:tools}.

\subsection{Proof of Lemma~\ref{lem:disc_trace_app}}

\begin{proof}
(1) Proof of~\eqref{eq:disc_trace_OK}. Let $(T,i)\in\pmeshOK$ and $\phi\in \P^\ell(T^i;\R)$.
\rev{We have}
\[
\sum_{S\in\{T\}\cup \N_i^{-1}(T)} \Big\{ \|\phi^+\|_{S} +
h_S^{\frac12} \|\phi^+\|_{\dSi\cup\SG} \Big\} \lesssim \sum_{S\in\{T\}\cup \N_i^{-1}(T)} \|\phi^+\|_{S}
\lesssim \|\phi^+\|_{\Delta(T)},
\]
where \rev{the first bound follows from the discrete trace inequality from 
\cite[Lemma~3.4]{BE18} upon observing that the aggregated cell $T\cup \N_i^{-1}(T)$
satisfies the ball condition invoked therein,} and 
the second bound follows from the construction of the pairing operator.
We next observe that there is a ball $\ball(\Delta(T))$
of diameter $\gamma h_T$ such that $\Delta(T) \subset \ball(\Delta(T))$, where $\gamma$ 
only depends on the mesh shape-regularity parameter. We also recall that, since 
$(T,i)\in\pmeshOK$, there is a ball $\ball(T,i)$ of diameter $\vartheta h_T$, so that 
$\ball(T,i) \subset T^i$. We then infer that
\[
\|\phi^+\|_{\Delta(T)} \le \|\phi^+\|_{\ball(\Delta(T))} \lesssim \|\phi\|_{\ball(T,i)} 
\le \|\phi\|_{T^i},
\]
where the inverse inequality invoked in the second inequality follows by the arguments 
given in the proof of \cite[Lemma~3.4]{BCDE21}. This completes the proof 
of~\eqref{eq:disc_trace_OK}.

(2) Proof of~\eqref{eq:disc.poinc}. Let $(T,i)\in\pmeshOK$ and $\phi\in \P^{\ell}(T^i;\R)$.
Let $\langle\phi\rangle_{\ball(T,i)}$ denote the mean-value of $\phi$ in $\ball(T,i)$. 
Since $(I-\projS)(\phi^+)=(I-\projS)\big((\phi-\langle\phi\rangle_{\ball(T,i)})^+\big)$,
invoking the $L^2$-stability of $\projS$ followed by~\eqref{eq:disc_trace_OK} gives
\begin{align*}
\sum_{S\in\{T\}\cup \N_i^{-1}(T)} \!\! h_S^{-\frac12} \|(I-\projS)(\phi^+)\|_{\dSi}
& \le \!\!\sum_{S\in\{T\}\cup \N_i^{-1}(T)}\!\! h_S^{-\frac12} \|(\phi-\langle\phi\rangle_{\ball(T,i)})^+\|_{\dSi}
\ifHAL \cuthere \fi
\lesssim h_T^{-1} \|\phi-\langle\phi\rangle_{\ball(T,i)}\|_{T^i},
\end{align*}
where we also used the mesh shape-regularity which implies that $h_T\lesssim h_S\lesssim h_T$
for all $S \in \Delta(T)$. The bound~\eqref{eq:disc.poinc} now follows
from the discrete Poincar\'e inequality established in~\cite[Lemma~3.4]{BCDE21}.
\end{proof}

\subsection{Proof of Lemma~\ref{lem:approx.IkT}}

\begin{proof}
Let $v\in H^s(\Omega_1\cup \Omega_2)$ with $s\in(\frac32,k+2]$ 
and let $(T,i)\in\pmeshOK$. Set $\tilde v_i:=E_i^s(v_i)$.

(1) Recall that the ball $\ball(T,i)$ of diameter $\vartheta h_T$ is a subset of $T^i$ since
$(T,i)\in\pmeshOK$. Let $\langle v_i\rangle_{\ball(T,i)}$ denote the mean-value of $v_i$ in $\ball(T,i)$. In this step, we prove that, for all $S\in\{T\}\cup \N_i^{-1}(T)$,
\begin{equation} \label{eq:Poinc.S}
\|\tilde v_i-\langle v_i\rangle_{\ball(T,i)}\|_{S} \lesssim h_T \|\nabla \tilde v_i\|_{\Delta(T)}.
\end{equation}
We have $\|\tilde v_i-\langle v_i\rangle_{\ball(T,i)}\|_{S} \le \|\tilde v_i
-\langle v_i\rangle_{\ball(T,i)}\|_{\Delta(T)}$ since, by construction, $S\subset \Delta(T)$. 
Let $\langle \tilde v_i\rangle_{\Delta(T)}$ denote the mean-value of $\tilde v_i$ on $\Delta(T)$.
Invoking \cite[Lemma~5.7]{ErnGu:17_quasi} gives
\[
\|\tilde v_i -\langle \tilde v_i\rangle_{\Delta(T)}\|_{\Delta(T)} \lesssim
h_T \|\nabla \tilde v_i\|_{\Delta(T)}.
\]
Moreover, letting $\partial \ball(T,i)$ denote the boundary of the ball $\ball(T,i)$ and since
$\ball(T,i)\subset \Delta(T)$, we have
\begin{align*}
\|\langle v_i\rangle_{\ball(T,i)} -\langle \tilde v_i\rangle_{\Delta(T)}\|_{\Delta(T)} &=
|\Delta(T)|^{\frac12} |\partial\ball(T,i)|^{-\frac12} \|\langle v_i\rangle_{\ball(T,i)} -\langle \tilde v_i\rangle_{\Delta(T)}\|_{\partial\ball(T,i)} \\
&\le |\Delta(T)|^{\frac12} |\partial\ball(T,i)|^{-\frac12} \big( \|v_i-\langle v_i\rangle_{\ball(T,i)}\|_{\partial\ball(T,i)} + \|\tilde v_i-\langle \tilde v_i\rangle_{\Delta(T)}\|_{\partial\ball(T,i)}\big),
\end{align*}
since $\tilde v_i|_{\partial\ball(T,i)}=v_i|_{\partial\ball(T,i)}$.
Invoking a multiplicative trace inequality from $\partial\ball(T,i)$ to $\ball(T,i)$ gives
\[
\|v_i-\langle v_i\rangle_{\ball(T,i)}\|_{\partial\ball(T,i)} \lesssim h_T^{-\frac12}
\|v_i-\langle v_i\rangle_{\ball(T,i)}\|_{\ball(T,i)} + h_T^{\frac12} \|\nabla v_i\|_{\ball(T,i)},
\]
and the Poincar\'e inequality in $\ball(T,i)$ then gives
\[
\|v_i-\langle v_i\rangle_{\ball(T,i)}\|_{\partial\ball(T,i)} \lesssim
h_T^{\frac12} \|\nabla v_i\|_{\ball(T,i)}.
\]
The same arguments, together with $\ball(T,i) \subset \Delta(T)$ and the mesh shape-regularity, lead to
\[
\|\tilde v_i-\langle \tilde v_i\rangle_{\Delta(T)}\|_{\partial\ball(T,i)} \lesssim h_T^{\frac12} \|\nabla \tilde v_i\|_{\Delta(T)}.
\]
Combining the above bounds and since $|\Delta(T)|^{\frac12} |\partial\ball(T,i)|^{-\frac12}\lesssim h_T^{\frac12}$ proves~\eqref{eq:Poinc.S}.

(2) The higher-order version of~\eqref{eq:Poinc.S} is established by invoking the Morrey
polynomial of $v_i$ based on mean-values of higher-order derivatives of $v_i$ (see, e.g.,
the proof of \cite[Lemma~5.6]{ErnGu:17_quasi}). Omitting the details for brevity, this
gives, for all $0\le m < s$, a polynomial $q_m(v_i) \in \P^{k+1}(T^i;\R)$ such that, for all $S\in\{T\}\cup \N_i^{-1}(T)$,
\begin{equation} \label{eq:Poinc.S.higher}
|\tilde v_i-q_m(v_i)|_{H^m(S)} \lesssim h_T^{s-m} |\tilde v_i|_{H^s(\Delta(T))}.
\end{equation}

(3) Since $I_{T^i}^{k+1}(q_0(v_i))=q_0(v_i)$, the triangle inequality gives
\[
\|\tilde v_i-I_{T^i}^{k+1}(v_i)^+\|_{S} \le \|\tilde v_i-q_0(v_i)\|_{S} + \|I_{T^i}^{k+1}(v_i-q_0(v_i))^+\|_{S}.
\]
The first term on the right-hand side is estimated by using~\eqref{eq:Poinc.S.higher} with $m=0$. 
For the second term, the discrete inverse inequality~\eqref{eq:disc_trace_OK} gives
\[
\|I_{T^i}^{k+1}(v_i-q_0(v_i))^+\|_{S} \lesssim \|I_{T^i}^{k+1}(v_i-q_0(v_i))\|_{T^i} 
\le \|I_{T^i}^{k+1}(v_i-q_0(v_i))\|_{T} \le \|\tilde v_i-q_0(v_i)\|_T,
\] 
where we used the $L^2$-orthogonality of $I_{T^i}^{k+1}$ in $T$. 
Invoking again~\eqref{eq:Poinc.S.higher} with $m=0$ and combining the above two bounds proves that
\[
\sum_{S\in\{T\}\cup\N_i^{-1}(T)} \|\tilde v_i-I_{T^i}^{k+1}(v_i)^+\|_{S} 
\lesssim h_T^s |\tilde v_i|_{H^s(\Delta(T))}.
\]
A similar reasoning with $q_1(v_i)$ in~\eqref{eq:Poinc.S.higher} also gives
\[
\sum_{S\in\{T\}\cup\N_i^{-1}(T)} h_S \|\nabla(\tilde v_i-I_{T^i}^{k+1}(v_i)^+)\|_{S} 
\lesssim h_T^{s} |\tilde v_i|_{H^s(\Delta(T))}.
\]
Altogether, this proves that
\begin{multline*}
\sum_{S\in\{T\}\cup\N_i^{-1}(T)} \Big\{ \|v_i-I_{T^i}^{k+1}(v_i)^+\|_{S^i} 
+ h_S \|\nabla(v_i-I_{T^i}^{k+1}(v_i)^+)\|_{S^i}\Big\} \\
\le \sum_{S\in\{T\}\cup\N_i^{-1}(T)} \Big\{ \|\tilde v_i-I_{T^i}^{k+1}(v_i)^+\|_{S} 
+ h_S \|\nabla(\tilde v_i-I_{T^i}^{k+1}(v_i)^+)\|_{S} \Big\}
\lesssim h_T^{s} |\tilde v_i|_{H^s(\Delta(T))}.
\end{multline*}

(4) Invoking a multiplicative trace inequality on all the faces composing $\dSi$
for all $S\in\{T\}\cup \N_i^{-1}(T)$ \rev{(observe that the aggregated cell $T\cup \N_i^{-1}(T)$ satisfies the required properties)}, and since $s>\frac32$ by assumption,
we infer from the above bounds that 
\[
\sum_{S\in\{T\}\cup\N_i^{-1}(T)} \Big\{ h_S^{\frac12}\|v_i-I_{T^i}^{k+1}(v_i)^+\|_{\dSi}  
+ h_S^{\frac32} \|\nabla(v_i-I_{T^i}^{k+1}(v_i)^+)\|_{\dSi} \Big\}
\lesssim h_T^{s} |\tilde v_i|_{H^s(\Delta(T))}.
\]
This completes the proof of~\eqref{eq:approx.i}.

(5) To prove~\eqref{eq:approx.jump}, we bound the jump across $\SG$ by the triangle inequality, so that we need to estimate the traces on $\SG$ from both sides of $\SG$, for all $S\in\{T\}\cup \N_i^{-1}(T)$. On each side, we invoke a multiplicative trace inequality on $\SG$. 
This inequality is established with a slight adaptation of
the arguments in the proof of \cite[Lemma~3.3]{BE18}, whereby we make a specific 
choice for the apex of the cone, say $C(\SG)$, considered in that proof.

(5a) Assume first that $S=T$. The triangle inequality gives
\begin{align*}
h_T^{\frac12} \| \llbracket v-I_T^{k+1}(v) \rrbracket_{\Gamma}\|_{\TG}
\le h_T^{\frac12} \|v_i-I_{T^i}^{k+1}(v_i)\|_{\TG} + h_T^{\frac12} \|v_{\ibar}-I_{T^{\ibar}}^{k+1}(v_{\ibar})\|_{\TG}. 
\end{align*}
For the first term on the right-hand side, the apex of the cone $C(\TG)$
is taken to be the center of the ball $B(T,i)$. Since $C(\TG)\subset \conv(T)\subset \Delta(T)$ by the assumption~\eqref{eq:conv_delta}, we infer that
\[
h_T^{\frac12} \|v_i-I_{T^i}^{k+1}(v_i)\|_{\TG} \lesssim \|\tilde v_i-I_{T^i}^{k+1}(v_i)^+\|_{\Delta(T)}
+ h_T \|\nabla(\tilde v_i-I_{T^i}^{k+1}(v_i)^+)\|_{\Delta(T)},
\]
and adapting the above arguments gives
\[
h_T^{\frac12} \|v_i-I_{T^i}^{k+1}(v_i)\|_{\TG} \lesssim h_T^{s} |\tilde v_i|_{H^s(\Delta(T))}.
\]
To bound $h_T^{\frac12} \|v_{\ibar}-I_{T^{\ibar}}^{k+1}(v_{\ibar})\|_{\TG}$, we can consider
the same cone to derive the multiplicative trace inequality, so that
\[
h_T^{\frac12} \|v_{\ibar}-I_{T^{\ibar}}^{k+1}(v_{\ibar})\|_{\TG} \lesssim h_T^{s} |\tilde v_{\ibar}|_{H^s(\Delta(T))},
\]
where $\tilde v_{\ibar}:=E_{\ibar}^s(v_{\ibar})$. Altogether, this proves that
\[
h_T^{\frac12} \| \llbracket v-I_T^{k+1}(v) \rrbracket_{\Gamma}\|_{\TG} \lesssim \sum_{i\in\{1,2\}} h_T^{s} |\tilde v_i|_{H^s(\Delta(T))}.
\]

(5b) Assume now that $S\in \N_i^{-1}(T)$. Then, the apex of the cone, $C(\SG)$, to establish
the multiplicative trace inequality is taken to be the center of the ball $B(S,\ibar)$ 
(indeed, if $(S,i)
\in \pmeshKO$, then $(S,\ibar)\in \pmeshOK$). This yields 
$C(\SG) \subset \conv(S) \subset \Delta(S) \subset 
\Delta_2(T)$. Invoking the same arguments as above then yields
\[
h_S^{\frac12} \| \llbracket I_T^{k+1}(v)^+-v \rrbracket_{\Gamma}\|_{\SG} \lesssim \sum_{i\in\{1,2\}} h_T^{s} |\tilde v_i|_{H^s(\Delta_2(T))}.
\]
This completes the proof.
\end{proof}

\subsubsection*{Acknowledgments}

EB was partially supported by the EPSRC grants EP/T033126/1 and EP/V050400/1. For the purpose of open access, the author has applied a Creative Commons Attribution (CC BY) licence to any Author Accepted Manuscript version arising.
RM is thankful to Matteo Cicuttin (Politecnico Torino, Italy) and Guillaume Delay (Sorbonne University, France) for their
instrumental help in code development, and to Zhaonan Dong (INRIA Paris) for stimulating discussions.

\bibliography{biblio.bib}
\bibliographystyle{abbrv}

\end{document}

%% file: preamble.tex
\usepackage[utf8]{inputenc}  
\usepackage[T1]{fontenc} 
\usepackage[left=2cm, right=2cm, top=2cm, bottom=2.5cm]{geometry}
\usepackage[english]{babel} 
\usepackage{lmodern}

\usepackage{etoolbox}
\ifHAL 
\usepackage{authblk}
\fi

\usepackage{textcomp}                
\usepackage{gensymb}                 
\usepackage{enumerate}               
\usepackage{enumitem}                
\usepackage{titlesec}                
\usepackage{color}                   
\ifHAL
\usepackage[dvipsnames]{xcolor}      
\else
\usepackage{xcolor}
\fi
\definecolor{ceared}{HTML}{BB0000} 
\usepackage{indentfirst}             
\usepackage{dsfont}
\usepackage{fancyhdr}                
\usepackage{chngpage}                
\usepackage[symbol]{footmisc}        
\usepackage{multicol}
\usepackage{float}

\usepackage{diagbox}
\usepackage[font=small, skip=10pt]{caption} 
\usepackage{subcaption}                     
\usepackage{tikz}                           
\usepackage{tkz-euclide}                    
\usetikzlibrary{patterns, patterns.meta}    
\usepackage{graphicx}
\usepackage{pgfplots}
\usetikzlibrary{shapes, positioning}
\pgfplotsset{compat=1.16}
\usepackage{multirow}
\usepackage{placeins}

\usepackage{amssymb} 
\ifHAL
\usepackage{amsmath,amsfonts,amsthm}
\fi
\usepackage{stmaryrd}
\usepackage{mathtools}                
\usepackage{empheq,etoolbox}          
\usepackage{arydshln}                 
\usepackage{wasysym}                  
\usepackage{mathbbol}
\usepackage{mathrsfs}
\usepackage{yfonts}
\ifHAL
\newtheorem{theorem}{Theorem}[section]

\newtheorem{lemma}[theorem]{Lemma}

\fi

\usepackage{pdfpages}
\usepackage{bbm}
\usepackage{scalerel}

\ifHAL
\titleformat*{\section}{\LARGE\bfseries}
\titleformat*{\subsection}{\Large\bfseries}
\titleformat*{\subsubsection}{\large\bfseries}
\titlespacing*{\section}{0pt}{0.25cm}{1em}
\titlespacing*{\subsection}{0pt}{0.75em}{1em}

\usepackage{nameref} 
\usepackage[hidelinks]{hyperref} 
\usepackage{cleveref}

\usepackage[subfigure]{tocloft}      
\usepackage{titletoc}
\fi

%% file: model.tikz
\def\L{1.5}    
\def\A{1.8}    
\def\d{2.2}    
\def\ang{-20}  
\def\anga{72}  
\def\bodyshape{
(-100:1.0*\L) to[out=190,in=-90]
(180:1.4*\L) to[out=90,in=175]
(110:1.0*\L) to[out=-5,in=185]
($(R)+(110:0.6*\L)$) to[out=5,in=90]
($(R)+(15:0.7*\L)$) to[out=-90,in=-20]
($(R)+(-120:0.7*\L)$) to[out=160,in=10] cycle
}
\def\body{
\draw[line width=1, draw=red] \bodyshape;
}
\begin{tikzpicture}[scale=0.4]
\draw[line width=1, draw=black] (0, 0) rectangle (10, 7);
\shade[left color=blue!10, right color=blue!20] (10,0) rectangle (0,7);
\node at (0.75, 0.5) {$\boldsymbol{\Omega_2}$};
\begin{scope}[xshift=40]
\node at (6, 5) {$\boldsymbol{\Gamma}$};
\draw [->, line width=0.75, color=red] (3.75, 2.5) -- (3.5, 1.5);
\node at (4.5, 1.5) {\color{red} \footnotesize $\boldsymbol{n_\Gamma}$};
\end{scope}
\begin{scope}[scale=1.225, xshift=100, yshift=97.5]
\coordinate (R) at (\ang:\d); 
\begin{scope}
\clip \bodyshape;
\shade[shading=axis,shading angle=90,left color=gray!10, right color=gray!40] \bodyshape;
\end{scope}
\body
\end{scope}
\node at (3, 4.5) {$\boldsymbol{\Omega_1}$};
\end{tikzpicture}

%% file: cut_mesh.tikz
\def\L{1.5}    
\def\A{1.8}    
\def\d{2.2}    
\def\ang{-20}  
\def\anga{72}  
\def\bodyshape{
(-100:1.0*\L) to[out=190,in=-90]
(180:1.4*\L) to[out=90,in=175]
(110:1.0*\L) to[out=-5,in=185]
($(R)+(110:0.6*\L)$) to[out=5,in=90]
($(R)+(15:0.7*\L)$) to[out=-90,in=-20]
($(R)+(-120:0.7*\L)$) to[out=160,in=10] cycle
}
\def\body{
\draw[line width=1, draw=red] \bodyshape;
}
\newdimen\LineSpace 
\LineSpace=8cm 
\definecolor{mynwlinescolor}{RGB}{0,0,255} 
\pgfdeclarepatternformonly[\LineSpace]{hachure_bleues}
{\pgfqpoint{-1pt}{-1pt}}
{\pgfqpoint{8pt}{8pt}}
{\pgfqpoint{\LineSpace}{\LineSpace}}
{
\pgfsetlinewidth{0.25pt}
\pgfsetstrokecolor{mynwlinescolor} 
\pgfpathmoveto{\pgfqpoint{0pt}{0pt}}
\pgfpathlineto{\pgfqpoint{\LineSpace + 0.1pt}{\LineSpace + 0.1pt}}
\pgfusepath{stroke}
}
\begin{tikzpicture}[scale=0.6]
\def\n{7} 
\def\dx{1} 
\def\dy{1} 
\fill[color=blue!20] (0, 0) rectangle (7, 2);
\fill[color=blue!20] (0, 6) rectangle (7, 7);
\fill[color=blue!20] (4, 5) rectangle (7, 6);
\fill[color=gray!60]  (1, 4) rectangle (3, 5);
\fill[color=gray!60]  (1, 3) rectangle (6, 4);
\shade[bottom color=blue!20, top color=gray!60] (1, 2) rectangle (4, 3);
\shade[top color=blue!20, bottom color=gray!60] (4, 4) rectangle (5, 5);
\shade[left color=blue!20, right color=gray!60] (0, 3) rectangle (1, 4);
\shade[left color=gray!60, right color=blue!20] (6, 3) rectangle (7, 4);
\pgfmathsetmacro{\LineSpace}{10} 
\draw[pattern=north east lines, pattern color=blue] (0, 2) rectangle (1, 3);
\draw[pattern=north east lines, pattern color=blue] (6, 2) rectangle (7, 3);
\draw[pattern=north east lines, pattern color=blue] (5, 4) rectangle (7, 5);
\draw[pattern=north east lines, pattern color=blue] (0, 5) rectangle (1, 6);
\draw[pattern=north east lines, pattern color=blue] (3, 5) rectangle (4, 6);
\draw[pattern=north west lines, pattern color=red] (0, 4) rectangle (1, 5);
\draw[pattern=north west lines, pattern color=red] (3, 4) rectangle (4, 5);
\draw[pattern=north west lines, pattern color=red] (1, 5) rectangle (3, 6);
\draw[pattern=north west lines, pattern color=red] (4, 2) rectangle (6, 3);

\begin{scope}[xshift= -12cm]
\draw[line width = 0.75] -- (10, 0) rectangle (11, 1);
\draw[line width = 0.5] -- (10, 1.5) rectangle (11, 2.5);
\draw[line width = 0.5] -- (10, 3) rectangle (11, 4);
\draw[line width = 0.75] -- (10, 4.5) rectangle (11, 5.5);
\draw[line width = 0.75] -- (10, 6) rectangle (11, 7);
\shade[left color=blue!20, right color=gray!50] (10, 0) rectangle (11, 1);
\draw[line width = 1, color=red] -- (10.5, 0) rectangle (10.5, 1);
\draw[pattern=north west lines, pattern color=red] (10, 1.5) rectangle (11, 2.5);
\draw[pattern=north east lines, pattern color=blue] (10, 3) rectangle (11, 4);
\fill[color=blue!20] (10, 4.5) rectangle (11, 5.5);
\fill[color=gray!60] (10, 6) rectangle (11, 7);
\node at (8, 6.5) {$\mesh^1$};
\node at (8, 5) {$\mesh^2$};
\node at (8.525, 3.5) {${\TKOonemesh}$};
\node at (8.525, 2) {${\TKOtwomesh}$};
\node at (8.3, 0.5) {${\TOKmesh}$};
\end{scope}
\foreach \i in {0,...,\n}
{
\draw (\i*\dx, 0) -- (\i*\dx, \n*\dy);
}
\foreach \j in {0,...,\n}
{
\draw (0, \j*\dy) -- (\n*\dx, \j*\dy);
}
\begin{scope}[scale=1.225, xshift=65, yshift=97.5]
\coordinate (R) at (\ang:\d); 
\body
\end{scope}
\draw[line width = 0.75, dashed] (6, 3) -- (11, 1);
\draw[line width = 0.75, dashed] (6, 4) -- (11, 6);
\draw[line width = 0.75, dashed] (7, 3) -- (16, 1);
\draw[line width = 0.75, dashed] (7, 4) -- (16, 6);
\fill[color=blue!20]  (11, 1) rectangle (16, 6);
\fill[gray!40] (12.75, 6) to[out=-20, in=100] (14, 4.75) to[out=-80, in=45] (12.5, 1) -- (11,1) -- (11,6) -- cycle;
\draw[red, line width=1.25pt, smooth] (12.75, 6) to[out=-20, in=100] (14, 4.75) to[out=-80, in=45] (12.5, 0.99);
\draw[line width = 1.25] (11, 1) rectangle (16, 6);
\node at (11.75, 1.75) {\color{gray!100} $\boldsymbol{\textbf{T}^1}$};
\node at (15.25, 1.75) {\color{blue!75} $\boldsymbol{\textbf{T}^2}$};
\node at (11.75, 6.75) {\color{gray!100} ${(\partial \text{T})^1}$};
\node at (15.25, 6.75) {\color{blue!75} ${(\partial \text{T})^2}$};
\node at (13, 5) {\color{red} $\boldsymbol{\textbf{T}^\Gamma}$};
\draw[black, line width = 2.75] (12.25, 3.675) circle (5pt);
\fill[gray] (12.25, 3.675) circle (5pt);
\draw[black, line width = 2.75] (12, 3.25) circle (5pt);
\fill[gray] (12, 3.25)  circle (5pt);
\draw[black, line width = 2.75] (12.5, 3.25) circle (5pt);
\fill[gray] (12.5, 3.25)  circle (5pt);
\begin{scope}[xshift=2.7cm]
\draw[black, line width = 2.75] (12.25, 3.675) circle (5pt);
\fill[blue] (12.25, 3.675) circle (5pt);
\draw[black, line width = 2.75] (12, 3.25) circle (5pt);
\fill[blue] (12, 3.25)  circle (5pt);
\draw[black, line width = 2.75] (12.5, 3.25) circle (5pt);
\fill[blue] (12.5, 3.25)  circle (5pt);
\end{scope}
\filldraw[black, line width=1pt] (11.75, 1) node[shape=diamond, draw, fill=gray, inner sep=0pt, minimum size=10pt] {};
\filldraw[black, line width=1pt] (11.875, 6) node[shape=diamond, draw, fill=gray, inner sep=0pt, minimum size=10pt] {};
\filldraw[black, line width=1pt] (11, 3.5) node[shape=diamond, draw, fill=gray, inner sep=0pt, minimum size=10pt] {};
\filldraw[black, line width=1pt] (14.25, 1) node[shape=diamond, draw, fill=blue, inner sep=0pt, minimum size=10pt] {};
\filldraw[black, line width=1pt] (14.5, 6) node[shape=diamond, draw, fill=blue, inner sep=0pt, minimum size=10pt] {};
\filldraw[black, line width=1pt] (16, 3.5) node[shape=diamond, draw, fill=blue, inner sep=0pt, minimum size=10pt] {};
\draw[line width=1.5] (6*\dx, 3*\dy) -- (6*\dx, 4*\dy);
\draw[line width=1.5] (6*\dx, 3*\dy) -- (7*\dx, 3*\dy);
\draw[line width=1.5] (6*\dx, 4*\dy) -- (7*\dx, 4*\dy);
\draw[line width=1.5] (7*\dx, 3*\dy) -- (7*\dx, 4*\dy);
\end{tikzpicture}

%% file: pairing.tikz
\begin{tikzpicture}[scale=1]
\draw [->, color=red] (1.55, 1.25) to [out=0, in=180] (3.95, 1/6);
\draw [->, color=red] (1.55, 1.5)  to [out=0, in=180] (3.95, 2+1/6);
\draw [->, color=red] (1.55, 1.75) to [out=0, in=180] (3.95, 3+1/6);
\draw [->, color=blue] (1.55, 2.25) to [out=0, in=180] (3.95, 1/3);
\draw [->, color=blue] (1.55, 2.5)  to [out=0, in=180] (3.95, 1+1/6);
\draw [->, color=blue] (1.55, 2.75) to [out=0, in=180] (3.95, 4+1/6);
\node at (2, 3.5) {\color{blue} $\N_1$};
\node at (2, 0.5) {\color{red} $\N_2$};
\pgfmathsetmacro{\LineSpace}{8} 
\draw[pattern=north east lines, pattern color=blue] (0, 2) rectangle (1.5, 3);
\draw[pattern=north west lines, pattern color=red] (0, 1) rectangle (1.5, 2);
\draw[line width=0.1pt] (4, 1/2) -- (5.5, 1/2);
\draw[line width=0.1pt] (4, 1/3) -- (5.5, 1/3);
\draw[line width=0.1pt] (4, 0.25-1/12) -- (5.5, 0.25-1/12);
\fill[pattern=north east lines, pattern color=blue] (4, 0.25-1/12) rectangle (5.5, 0.5);
\fill[pattern=north west lines, pattern color=red] (4, 0) rectangle (5.5, 0.25+1/12);
\draw[line width=0.1pt] (4, 1+1/3) -- (5.5, 1+1/3);
\fill[pattern=north east lines, pattern color=blue] (4, 1) rectangle (5.5, 1+1/3);
\draw[line width=0.1pt] (4, 2+1/3) -- (5.5, 2+1/3);
\fill[pattern=north west lines, pattern color=red] (4, 2) rectangle (5.5, 2+1/3);
\draw[line width=0.1pt] (4, 3+1/3) -- (5.5, 3+1/3);
\fill[pattern=north west lines, pattern color=red] (4, 3) rectangle (5.5, 3+1/3);
\draw[line width=0.1pt] (4, 4+1/3) -- (5.5, 4+1/3);
\fill[pattern=north east lines, pattern color=blue] (4, 4) rectangle (5.5, 4+1/3);
\node at (4.75, 0.5) {$\TOKmesh$};
\node at (4.75, 1.5) {$\TKOtwomesh$};
\node at (4.75, 2.5) {$\TKOonemesh$};
\node at (4.75, 3.5) {$\mesh^2$};
\node at (4.75, 4.5) {$\mesh^1$};
\node at (0.75, 0.5) {$\TOKmesh$};
\node at (0.75, 1.5) {$\TKOtwomesh$};
\node at (0.75, 2.5) {$\TKOonemesh$};
\node at (0.75, 3.5) {$\mesh^2$};
\node at (0.75, 4.5) {$\mesh^1$};
\draw[line width=0.5pt] (0, 0) rectangle (1.5, 5);
\draw[line width=1.5pt] (0, 1) -- (1.5, 1);
\draw[line width=1.5pt] (0, 2) -- (1.5, 2);
\draw[line width=1.5pt] (0, 3) -- (1.5, 3);
\draw[line width=1.5pt] (0, 4) -- (1.5, 4);
\draw[line width=0.5pt] (0, 5) -- (1.5, 5);
\draw[line width=0.5pt] (4, 0) rectangle (5.5, 5);
\draw[line width=1.5pt] (4, 1) -- (5.5, 1);
\draw[line width=1.5pt] (4, 2) -- (5.5, 2);
\draw[line width=1.5pt] (4, 3) -- (5.5, 3);
\draw[line width=1.5pt] (4, 4) -- (5.5, 4);
\draw[line width=0.5pt] (4, 5) -- (5.5, 5);
\end{tikzpicture}

%% file: grad_uncut.tex
\begin{tikzpicture}[scale=2.5]
        \begin{scope}
        \clip (0, 0) rectangle (1,1);
        \draw[pattern=north east lines, pattern color=black] (0.7,0) rectangle (1,1);
        \end{scope}
        \fill[blue!20] (1, 0) rectangle (2,1);
        \draw[line width=1.25pt, color=black] (0,0) rectangle (1,1);
        \draw[line width=1.25pt, color=black] (1,0) rectangle (2,1);
        \draw[line width=1.25pt, color=black] (1,0) -- (1,1);
        \draw[line width=1.5pt, color=red] (0.7,-0.1) -- (0.7,1.1);
        \node at (0.575,0.75) {\footnotesize $\bf \color{red} \textbf{S}^\Gamma$};
        \draw[black, line width = 2.75] (1.5, 0.585) circle (1pt);
        \fill[blue] (1.5, 0.585) circle (1pt);
        \draw[black, line width = 2.75] (1.45, 0.5) circle (1pt);
        \fill[blue] (1.45, 0.5)  circle (1pt);
        \draw[black, line width = 2.75] (1.55, 0.5) circle (1pt);
        \fill[blue] (1.55, 0.5)  circle (1pt); 
        \filldraw[black, line width=1pt] (2, 0.5) node[shape=diamond, draw, fill=blue, inner sep=0pt, minimum size=10pt] {};
        \filldraw[black, line width=1pt] (1.5, 1) node[shape=diamond, draw, fill=blue, inner sep=0pt, minimum size=10pt] {};
        \filldraw[black, line width=1pt] (1.5, 0) node[shape=diamond, draw, fill=blue, inner sep=0pt, minimum size=10pt] {};
        \filldraw[black, line width=1pt] (0.9, 1) node[shape=diamond, draw, fill=teal, inner sep=0pt, minimum size=10pt] {};
        \filldraw[black, line width=1pt] (0.9, 0) node[shape=diamond, draw, fill=teal, inner sep=0pt, minimum size=10pt] {};
        \filldraw[black, line width=1pt] (1, 0.5) node[shape=diamond, draw, fill=blue, inner sep=0pt, minimum size=10pt] {};
        \begin{scope}[xshift=-1.1cm]
        \draw[black, line width = 2.75] (1.45, 0.56) circle (1pt);
        \fill[gray] (1.45, 0.56) circle (1pt);
        \draw[black, line width = 2.75] (1.4, 0.475) circle (1pt);
        \fill[gray] (1.4, 0.475)  circle (1pt);
        \draw[black, line width = 2.75] (1.5, 0.475) circle (1pt);
        \fill[gray] (1.5, 0.475)  circle (1pt);  
        \end{scope}
        \begin{scope}[xshift=-0.575cm]
        \draw[black, line width = 2.75] (1.4, 0.6) circle (1pt);
        \fill[teal] (1.4, 0.6) circle (1pt);
        \draw[black, line width = 2.75] (1.4, 0.5) circle (1pt);
        \fill[teal] (1.4, 0.5)  circle (1pt);
        \draw[black, line width = 2.75] (1.4, 0.4) circle (1pt);
        \fill[teal] (1.4, 0.4)  circle (1pt);  
        \end{scope}
        \filldraw[black, line width=1pt] (0, 0.5) node[shape=diamond, draw, fill=gray, inner sep=0pt, minimum size=10pt] {};
        \filldraw[black, line width=1pt] (0.4, 1) node[shape=diamond, draw, fill=gray, inner sep=0pt, minimum size=10pt] {};
        \filldraw[black, line width=1pt] (0.4, 0) node[shape=diamond, draw, fill=gray, inner sep=0pt, minimum size=10pt] {};
        \node at (0.1,0.85) {$\tiny \color{darkgray} \boldsymbol{S^{\ibar}}$};
        \node at (0.8,-0.115) {$\tiny \color{darkgray} \boldsymbol{S^{i}}$};
        \node at (1.85,0.9) {$\tiny \color{blue} \boldsymbol{T^{i}}$};
        \node at (1.45,0.3) {$\tiny \color{teal} \boldsymbol{\mathcal{N}_i(S)^{i}}$};
        \draw[line width=1.25pt, color=teal, ->] (0.85,0.175) -- (1.25,0.175);
\end{tikzpicture}

%% file: grad_OK.tex
\begin{tikzpicture}[scale=2.5]
    \fill[color=blue!20]  (1, 0) rectangle (2,1);
    \begin{scope}
    \clip (0, 0) rectangle (1,1);
    \draw[pattern=north east lines, pattern color=black] (0.4,-0.15) to[out=95,in=180] (2,0.575) -- (2,-0.15) -- (0.4,-0.15);
    \end{scope}
    \begin{scope}
    \clip (2, 0) rectangle (3,1);
    \draw[pattern=north east lines, pattern color=black] (2,0.55) to[out=0,in=-90] (2.65,1.1) -- (2,1.1) -- cycle;
    \end{scope}
    \fill[gray!40] (1,0.48) to[out=15,in=180] (2,0.55) -- (2,0) -- (1,0) -- cycle;
    \draw[line width=1.25pt, color=black] (0,0) rectangle (1,1);
    \draw[line width=1.25pt, color=black] (1,0) rectangle (2,1);
    \draw[line width=1.25pt, color=black] (2,0) rectangle (3,1);
    \draw[line width=1.25pt, color=black] (1,0) -- (1,1);
    \draw[red, line width=1.5pt] (0.4,-0.1) to[out=90,in=180] (2,0.55) to[out=0,in=-90] (2.65,1.1);
    \node at (2.65,0.66)  {\footnotesize $\bf \color{red} S_2^\Gamma$};
    \node at (1.80,0.425) {\footnotesize $\bf \color{red} T^\Gamma$};
    \node at (0.425,0.33) {\footnotesize $\bf \color{red} S_1^\Gamma$};
    \draw[black, line width = 2.75] (1.5, 0.435) circle (1pt);
    \fill[gray] (1.5, 0.435) circle (1pt);
    \draw[black, line width = 2.75] (1.45, 0.35) circle (1pt);
    \fill[gray] (1.45, 0.35)  circle (1pt);
    \draw[black, line width = 2.75] (1.55, 0.35) circle (1pt);
    \fill[gray] (1.55, 0.35)  circle (1pt); 
    \begin{scope}[xshift=1cm]
    \draw[black, line width = 2.75] (1.5, 0.435) circle (1pt);
    \fill[gray] (1.5, 0.435) circle (1pt);
    \draw[black, line width = 2.75] (1.45, 0.35) circle (1pt);
    \fill[gray] (1.45, 0.35)  circle (1pt);
    \draw[black, line width = 2.75] (1.55, 0.35) circle (1pt);
    \fill[gray] (1.55, 0.35)  circle (1pt);         
    \end{scope}
    \begin{scope}[yshift=0.325cm]
    \begin{scope}[xshift=-0.15cm]
    \draw[black, line width = 2.75] (1.5, 0.435) circle (1pt);
    \fill[blue] (1.5, 0.435) circle (1pt);
    \draw[black, line width = 2.75] (1.45, 0.35) circle (1pt);
    \fill[blue] (1.45, 0.35)  circle (1pt);
    \draw[black, line width = 2.75] (1.55, 0.35) circle (1pt);
    \fill[blue] (1.55, 0.35)  circle (1pt); 
    \end{scope}
    \end{scope}
    \filldraw[black, line width=1pt] (2, 0.325)  node[shape=diamond, draw, fill=gray, inner sep=0pt, minimum size=10pt] {};
    \filldraw[black, line width=1pt] (1.5, 1)    node[shape=diamond, draw, fill=blue, inner sep=0pt, minimum size=10pt] {};
    \filldraw[black, line width=1pt] (1.5, 0)    node[shape=diamond, draw, fill=gray, inner sep=0pt, minimum size=10pt] {};
    \filldraw[black, line width=1pt] (2.5, 0)    node[shape=diamond, draw, fill=gray, inner sep=0pt, minimum size=10pt] {};
    \filldraw[black, line width=1pt] (3, 0.5)    node[shape=diamond, draw, fill=gray, inner sep=0pt, minimum size=10pt] {};
    \filldraw[black, line width=1pt] (2.875, 1)    node[shape=diamond, draw, fill=gray, inner sep=0pt, minimum size=10pt] {};
    \filldraw[black, line width=1pt] (0.725, 0)  node[shape=diamond, draw, fill=teal, inner sep=0pt, minimum size=10pt] {};
    \filldraw[black, line width=1pt] (1, 0.25)   node[shape=diamond, draw, fill=gray, inner sep=0pt, minimum size=10pt] {};
    \filldraw[black, line width=1pt] (1, 0.75)   node[shape=diamond, draw, fill=blue, inner sep=0pt, minimum size=10pt] {};
    \filldraw[black, line width=1pt] (2, 0.825)  node[shape=diamond, draw, fill=blue, inner sep=0pt, minimum size=10pt] {};
    \filldraw[black, line width=1pt] (2.375, 1)  node[shape=diamond, draw, fill=teal, inner sep=0pt, minimum size=10pt] {};
    \begin{scope}[xshift=-1.05cm]
    \begin{scope}[yshift=0.15cm]
    \draw[black, line width = 2.75] (1.5, 0.585) circle (1pt);
    \fill[blue] (1.5, 0.585) circle (1pt);
    \draw[black, line width = 2.75] (1.45, 0.5) circle (1pt);
    \fill[blue] (1.45, 0.5)  circle (1pt);
    \draw[black, line width = 2.75] (1.55, 0.5) circle (1pt);
    \fill[blue] (1.55, 0.5)  circle (1pt);  
    \end{scope}
    \end{scope}
    \begin{scope}[xshift=-0.775cm]
    \begin{scope}[yshift=-0.325cm]
    \draw[black, line width = 2.75] (1.5, 0.585) circle (1pt);
    \fill[teal] (1.5, 0.585) circle (1pt);
    \draw[black, line width = 2.75] (1.45, 0.5) circle (1pt);
    \fill[teal] (1.45, 0.5)  circle (1pt);
    \draw[black, line width = 2.75] (1.55, 0.5) circle (1pt);
    \fill[teal] (1.55, 0.5)  circle (1pt);  
    \end{scope}
    \end{scope}
    \begin{scope}[xshift=0.75cm]
    \begin{scope}[yshift=0.25cm]
    \draw[black, line width = 2.75] (1.5, 0.585) circle (1pt);
    \fill[teal] (1.5, 0.585) circle (1pt);
    \draw[black, line width = 2.75] (1.45, 0.5) circle (1pt);
    \fill[teal] (1.45, 0.5)  circle (1pt);
    \draw[black, line width = 2.75] (1.55, 0.5) circle (1pt);
    \fill[teal] (1.55, 0.5)  circle (1pt);  
    \end{scope}
    \end{scope}
    \filldraw[black, line width=1pt] (0, 0.5) node[shape=diamond, draw, fill=blue, inner sep=0pt, minimum size=10pt] {};
    \filldraw[black, line width=1pt] (0.5, 1) node[shape=diamond, draw, fill=blue, inner sep=0pt, minimum size=10pt] {};
    \filldraw[black, line width=1pt] (0.2, 0) node[shape=diamond, draw, fill=blue, inner sep=0pt, minimum size=10pt] {};
    \node at (0.15,0.9) {$\tiny \color{blue} \boldsymbol{S^{\ibar}_1}$};
    \node at (0.55,-0.1) {$\tiny \color{darkgray} \boldsymbol{S^{i}_1}$};
    \node at (1.85,0.1) {$\tiny \color{darkgray} \boldsymbol{T^{i}}$};
    \node at (1.2,0.9) {$\tiny \color{blue} \boldsymbol{T^{\ibar}}$};
    \node at (1.35,0.2) {$\tiny \color{teal} \boldsymbol{\mathcal{N}_i(S_1)^{i}}$};
    \draw[line width=0.75pt, color=teal, ->] (0.85,0.075) -- (1.175,0.075);
    \draw[line width=0.75pt, color=teal, <-] (1.85,0.65) -- (2.175,0.65);
    \node at (1.7,0.75) {$\tiny \color{teal} \boldsymbol{\mathcal{N}_{\ibar}(S_2)^{\ibar}}$};
    \node at (2.15,1.1) {$\tiny \color{blue} \boldsymbol{S^{\ibar}_2}$};
    \node at (2.85,0.1) {$\tiny \color{darkgray} \boldsymbol{S^{i}_2}$};
\end{tikzpicture}

%% file: grad_KO.tex
\begin{tikzpicture}[scale=2.5]
\fill[color=blue!20]  (0, 0) rectangle (1,1);
\fill[gray!40] (0.5,-0.15) to[out=90,in=180] (1,0.5) -- (2,0.675) -- (2,0) -- (0.5,0);
\fill[color=gray!10]  (1, 0) rectangle (2,1);
\draw[line width=1.25pt, color=black] (0,0) rectangle (1,1);
\draw[line width=1.25pt, color=black] (1,0) rectangle (2,1);
\draw[line width=1.25pt, color=black] (1,0) -- (1,1);
\node at (0.45,0.25) {\footnotesize $\bf \color{red} T^\Gamma$};
\node at (1.6,0.80)  {\footnotesize $\bf \color{red} S^\Gamma$};
\begin{scope}
\clip (0, 0) rectangle (1,1);
\draw[pattern=north east lines, pattern color=black] (0.5,-0.15) to[out=90,in=180] (1,0.5) -- (1,0) -- cycle;
\end{scope}
\begin{scope}
\clip (1, 0) rectangle (2,1);
\draw[pattern=north east lines, pattern color=black] (1,0.5) to[out=0,in=-90] (1.5,1.1) -- (1,1) -- cycle;
\end{scope}
\draw[red, line width=1.5pt] (0.5,-0.15) to[out=90,in=180] (1,0.5) to[out=0,in=-90] (1.5,1.1);
\begin{scope}
\begin{scope}[yshift=0.1cm]
\draw[black, line width = 2.75] (1.5, 0.435) circle (1pt);
\fill[red] (1.5, 0.435) circle (1pt);
\draw[black, line width = 2.75] (1.45, 0.35) circle (1pt);
\fill[red] (1.45, 0.35)  circle (1pt);
\draw[black, line width = 2.75] (1.55, 0.35) circle (1pt);
\fill[red] (1.55, 0.35)  circle (1pt);
\end{scope}
\end{scope} 
\begin{scope}[yshift=0.4cm]
\begin{scope}[xshift=-0.25cm]
\draw[black, line width = 2.75] (1.5, 0.435) circle (1pt);
\fill[teal] (1.5, 0.435) circle (1pt);
\draw[black, line width = 2.75] (1.45, 0.35) circle (1pt);
\fill[teal] (1.45, 0.35)  circle (1pt);
\draw[black, line width = 2.75] (1.55, 0.35) circle (1pt);
\fill[teal] (1.55, 0.35)  circle (1pt); 
\end{scope}
\end{scope}
\filldraw[black, line width=1pt] (2, 0.5)  node[shape=diamond, draw, fill=gray, inner sep=0pt, minimum size=10pt] {};
\filldraw[black, line width=1pt] (1.25, 1)    node[shape=diamond, draw, fill=teal, inner sep=0pt, minimum size=10pt] {};
\filldraw[black, line width=1pt] (1.5, 0)    node[shape=diamond, draw, fill=gray, inner sep=0pt, minimum size=10pt] {};
\filldraw[black, line width=1pt] (0.75, 0)  node[shape=diamond, draw, fill=teal, inner sep=0pt, minimum size=10pt] {};
\filldraw[black, line width=1pt] (1, 0.25)   node[shape=diamond, draw, fill=gray, inner sep=0pt, minimum size=10pt] {};
\filldraw[black, line width=1pt] (1, 0.75)   node[shape=diamond, draw, fill=blue, inner sep=0pt, minimum size=10pt] {};
\begin{scope}[xshift=-1.1cm]
\draw[black, line width = 2.75] (1.5, 0.585) circle (1pt);
\fill[blue] (1.5, 0.585) circle (1pt);
\draw[black, line width = 2.75] (1.45, 0.5) circle (1pt);
\fill[blue] (1.45, 0.5)  circle (1pt);
\draw[black, line width = 2.75] (1.55, 0.5) circle (1pt);
\fill[blue] (1.55, 0.5)  circle (1pt);  
\end{scope}
\begin{scope}[xshift=-0.7cm]
\begin{scope}[yshift=-0.3cm]
\draw[black, line width = 2.75] (1.5, 0.585) circle (1pt);
\fill[teal] (1.5, 0.585) circle (1pt);
\draw[black, line width = 2.75] (1.45, 0.5) circle (1pt);
\fill[teal] (1.45, 0.5)  circle (1pt);
\draw[black, line width = 2.75] (1.55, 0.5) circle (1pt);
\fill[teal] (1.55, 0.5)  circle (1pt);  
\end{scope}
\end{scope}
\filldraw[black, line width=1pt] (0, 0.5) node[shape=diamond, draw, fill=blue, inner sep=0pt, minimum size=10pt] {};
\filldraw[black, line width=1pt] (0.5, 1) node[shape=diamond, draw, fill=blue, inner sep=0pt, minimum size=10pt] {};
\filldraw[black, line width=1pt] (0.3, 0) node[shape=diamond, draw, fill=blue, inner sep=0pt, minimum size=10pt] {};
\node at (0.1,0.9) {$\tiny \color{blue} \boldsymbol{T^{\ibar}}$};
\node at (0.9,-0.1) {$\tiny \color{darkgray} \boldsymbol{T^{i}}$};
\node at (1.9,0.1) {$\tiny \color{darkgray} \boldsymbol{S^{i}}$};
\node at (1.1,1.1) {$\tiny \color{blue} \boldsymbol{S^{\ibar}}$};
\node at (1.45,0.2) {$\tiny \color{teal} \boldsymbol{\mathcal{N}_i(T)^{i}}$};
\node at (0.7,0.72) {$\tiny \color{teal} \boldsymbol{\mathcal{N}_{\ibar}(S)^{\ibar}}$};
\draw[line width=1.25pt, color=teal, ->] (0.85,0.10) -- (1.25,0.10);
\begin{scope}[yshift=0.5cm]
\begin{scope}[xshift=-0.05cm]
\draw[line width=1.25pt, color=teal, <-] (0.85,0.10) -- (1.25,0.10);
\end{scope}
\end{scope}
\end{tikzpicture}
\vspace{-0.25cm}

%% file: main.bbl
\begin{thebibliography}{10}

\bibitem{AbErPi:18}
M.~Abbas, A.~Ern, and N.~Pignet.
\newblock {H}ybrid {H}igh-{O}rder methods for finite deformations of hyperelastic materials.
\newblock {\em Comput. Mech.}, 62(4):909--928, 2018.

\bibitem{AyLiM:16}
B.~Ayuso~de Dios, K.~Lipnikov, and G.~Manzini.
\newblock The nonconforming virtual element method.
\newblock {\em ESAIM Math. Model. Numer. Anal.}, 50(3):879--904, 2016.

\bibitem{Bab73}
I.~Babu\v{s}ka.
\newblock The finite element method with penalty.
\newblock {\em Math. Comp.}, 27:221--228, 1973.

\bibitem{BVM18}
S.~Badia, F.~Verdugo, and A.~F. Mart\'{\i}n.
\newblock The aggregated unfitted finite element method for elliptic problems.
\newblock {\em Comput. Methods Appl. Mech. Engrg.}, 336:533--553, 2018.

\bibitem{BE84}
J.~W. Barrett and C.~M. Elliott.
\newblock A finite-element method for solving elliptic equations with {N}eumann data on a curved boundary using unfitted meshes.
\newblock {\em IMA J. Numer. Anal.}, 4(3):309--325, 1984.

\bibitem{BE86}
J.~W. Barrett and C.~M. Elliott.
\newblock Finite element approximation of the {D}irichlet problem using the boundary penalty method.
\newblock {\em Numer. Math.}, 49(4):343--366, 1986.

\bibitem{BE09}
P.~Bastian and C.~Engwer.
\newblock An unfitted finite element method using discontinuous {G}alerkin.
\newblock {\em Internat. J. Numer. Methods Engrg.}, 79(12):1557--1576, 2009.

\bibitem{BBCMMR13}
L.~Beir\~{a}o~da Veiga, F.~Brezzi, A.~Cangiani, G.~Manzini, L.~D. Marini, and A.~Russo.
\newblock Basic principles of virtual element methods.
\newblock {\em Math. Models Methods Appl. Sci.}, 23(1):199--214, 2013.

\bibitem{BdVLMR:24}
L.~Beir\~ao~da Veiga, Y.~Liu, L.~Mascotto, and A.~Russo.
\newblock The nonconforming virtual element method with curved edges.
\newblock {\em J. Sci. Comput.}, 99(1):Paper No. 23, 35, 2024.

\bibitem{BPV20}
A.~Buffa, R.~Puppi, and R.~V\'{a}zquez.
\newblock A minimal stabilization procedure for isogeometric methods on trimmed geometries.
\newblock {\em SIAM J. Numer. Anal.}, 58(5):2711--2735, 2020.

\bibitem{Bu10}
E.~Burman.
\newblock Ghost penalty.
\newblock {\em C. R. Math. Acad. Sci. Paris}, 348(21-22):1217--1220, 2010.

\bibitem{BCDE21}
E.~Burman, M.~Cicuttin, G.~Delay, and A.~Ern.
\newblock An unfitted hybrid high-order method with cell agglomeration for elliptic interface problems.
\newblock {\em SIAM J. Sci. Comput.}, 43(2):A859--A882, 2021.

\bibitem{BCHLM15}
E.~Burman, S.~Claus, P.~Hansbo, M.~G. Larson, and A.~Massing.
\newblock Cut{FEM}: discretizing geometry and partial differential equations.
\newblock {\em Internat. J. Numer. Methods Engrg.}, 104(7):472--501, 2015.

\bibitem{BDE21}
E.~Burman, G.~Delay, and A.~Ern.
\newblock An unfitted hybrid high-order method for the {S}tokes interface problem.
\newblock {\em IMA J. Numer. Anal.}, 41(4):2362--2387, 2021.

\bibitem{BDE22}
E.~Burman, O.~Duran, and A.~Ern.
\newblock Unfitted hybrid high-order methods for the wave equation.
\newblock {\em Comput. Methods Appl. Mech. Engrg.}, 389:Paper No. 114366, 23, 2022.

\bibitem{BE18}
E.~Burman and A.~Ern.
\newblock An unfitted hybrid high-order method for elliptic interface problems.
\newblock {\em SIAM J. Numer. Anal.}, 56(3):1525--1546, 2018.

\bibitem{CDGH17}
A.~Cangiani, Z.~Dong, E.~H. Georgoulis, and P.~Houston.
\newblock {\em {$hp$}-version discontinuous {G}alerkin methods on polygonal and polyhedral meshes}.
\newblock SpringerBriefs in Mathematics. Springer, Cham, 2017.

\bibitem{CoDPE16}
B.~Cockburn, D.~A. Di~Pietro, and A.~Ern.
\newblock Bridging the {H}ybrid {H}igh-{O}rder and hybridizable discontinuous {G}alerkin methods.
\newblock {\em ESAIM Math. Model. Numer. Anal.}, 50(3):635--650, 2016.

\bibitem{CGL09}
B.~Cockburn, J.~Gopalakrishnan, and R.~Lazarov.
\newblock Unified hybridization of discontinuous {G}alerkin, mixed, and continuous {G}alerkin methods for second order elliptic problems.
\newblock {\em SIAM J. Numer. Anal.}, 47(2):1319--1365, 2009.

\bibitem{CocQS14}
B.~Cockburn, W.~Qiu, and M.~Solano.
\newblock A priori error analysis for {HDG} methods using extensions from subdomains to achieve boundary conformity.
\newblock {\em Math. Comp.}, 83(286):665--699, 2014.

\bibitem{DiPDr:17}
D.~A. Di~Pietro and J.~Droniou.
\newblock A {H}ybrid {H}igh-{O}rder method for {L}eray-{L}ions elliptic equations on general meshes.
\newblock {\em Math. Comp.}, 86(307):2159--2191, 2017.

\bibitem{DiPEr:12}
D.~A. Di~Pietro and A.~Ern.
\newblock {\em Mathematical aspects of discontinuous {G}alerkin methods}, volume~69 of {\em Math\'ematiques \& Applications [Mathematics \& Applications]}.
\newblock Springer-Verlag, Berlin, 2012.

\bibitem{DPEL14}
D.~A. Di~Pietro, A.~Ern, and S.~Lemaire.
\newblock An arbitrary-order and compact-stencil discretization of diffusion on general meshes based on local reconstruction operators.
\newblock {\em Comput. Methods Appl. Math.}, 14(4):461--472, 2014.

\bibitem{ErnGu:17_quasi}
A.~Ern and J.-L. Guermond.
\newblock Finite element quasi-interpolation and best approximation.
\newblock {\em ESAIM Math. Model. Numer. Anal.}, 51(4):1367--1385, 2017.

\bibitem{Han02}
A.~Hansbo and P.~Hansbo.
\newblock An unfitted finite element method, based on {N}itsche's method, for elliptic interface problems.
\newblock {\em Comput. Methods Appl. Mech. Engrg.}, 191(47-48):5537--5552, 2002.

\bibitem{HR09}
J.~Haslinger and Y.~Renard.
\newblock A new fictitious domain approach inspired by the extended finite element method.
\newblock {\em SIAM J. Numer. Anal.}, 47(2):1474--1499, 2009.

\bibitem{JL13}
A.~Johansson and M.~G. Larson.
\newblock A high order discontinuous {G}alerkin {N}itsche method for elliptic problems with fictitious boundary.
\newblock {\em Numer. Math.}, 123(4):607--628, 2013.

\bibitem{LehOl:19}
C.~Lehrenfeld and M.~Olshanskii.
\newblock An {E}ulerian finite element method for {PDE}s in time-dependent domains.
\newblock {\em ESAIM Math. Model. Numer. Anal.}, 53(2):585--614, 2019.

\bibitem{LS:16}
C.~Lehrenfeld and J.~Sch\"{o}berl.
\newblock High order exactly divergence-free hybrid discontinuous {G}alerkin methods for unsteady incompressible flows.
\newblock {\em Comput. Methods Appl. Mech. Engrg.}, 307:339--361, 2016.

\bibitem{Mass12}
R.~Massjung.
\newblock An unfitted discontinuous {G}alerkin method applied to elliptic interface problems.
\newblock {\em SIAM J. Numer. Anal.}, 50(6):3134--3162, 2012.

\bibitem{Preuss:18}
J.~Preu{\ss}.
\newblock Higher order unfitted isoparametric space-time fem on moving domains.
\newblock Master's thesis, NAM, University of G\"ottingen, 2018.

\bibitem{QiuSV16}
W.~Qiu, M.~Solano, and P.~Vega.
\newblock A high order {HDG} method for curved-interface problems via approximations from straight triangulations.
\newblock {\em J. Sci. Comput.}, 69(3):1384--1407, 2016.

\bibitem{SBvdV11}
W.~E.~H. Sollie, O.~Bokhove, and J.~J.~W. van~der Vegt.
\newblock Space-time discontinuous {G}alerkin finite element method for two-fluid flows.
\newblock {\em J. Comput. Phys.}, 230(3):789--817, 2011.

\bibitem{WY13}
J.~Wang and X.~Ye.
\newblock A weak {G}alerkin finite element method for second-order elliptic problems.
\newblock {\em J. Comput. Appl. Math.}, 241:103--115, 2013.

\bibitem{Yemm:24}
L.~Yemm.
\newblock A new approach to handle curved meshes in the hybrid high-order method.
\newblock {\em Found. Comput. Math.}, 24(3):1049--1076, 2024.

\end{thebibliography}
